\documentclass[11pt]{article}

\usepackage{amsfonts}
\usepackage{amscd}
\usepackage{amssymb}
\usepackage{amsthm}
\usepackage{amsmath, xspace}
\usepackage{blkarray}
\usepackage{cancel}
\usepackage{color}
\usepackage{graphics}
\usepackage{graphicx}
\usepackage{enumitem}
\usepackage{fancyhdr}
\usepackage{mathdots}
\usepackage{mathrsfs}
\usepackage{multicol}
\usepackage{stmaryrd}
\usepackage{ytableau}
\usepackage[all]{xy}

\usepackage[plainpages,backref]{hyperref}

\usepackage{lscape}

\theoremstyle{plain}
\newtheorem{thm}{Theorem}[section]

\newtheorem{prop}[thm]{Proposition}

\theoremstyle{definition}
	
\newtheorem{remark}[thm]{Remark}
\theoremstyle{example}

\theoremstyle{remark}

\numberwithin{equation}{section}

\providecommand{\keywords}[1]{\textbf{\textit{Key words---}} #1}

\def\CC{\mathbb{C}}

\def\ZZ{\mathbb{Z}}

\def\ev{\mathrm{ev}}

\usepackage{etex} 
\usepackage{pictexwd}
\usepackage{tikz, tikz-cd} 
\usetikzlibrary{decorations.pathreplacing,angles,quotes} 
	\usepgflibrary[patterns] 
	\usetikzlibrary{patterns} 
	\usepgflibrary{shapes.geometric}

\newcommand{\TikZ}[1]{
\begin{matrix}\begin{tikzpicture}#1\end{tikzpicture}\end{matrix}
}

\newcounter{r}
\newcounter{s}

\newcommand\Part[1]{
        \setcounter{r}{1}
	 \foreach \x in {#1}{
 	{\ifnum\value{r}=1
		\draw (0,\value{r}-1)--(\x,\value{r}-1); 
		\fi}
	\draw (0,\value{r}) to (\x,\value{r});
   	\foreach \y in {0, ..., \x} {\draw (\y,\value{r})--(\y,\value{r}-1);}
	\addtocounter{r}{1}
 }}
 
\newcommand\Tableau[1]{
        \foreach \x [count = \c from 1] in {#1} {
		\foreach \y [count = \d from 1] in \x{
			\node at (\d-.5,\c-.5) {\scriptsize$\y$}; 
			\draw (\d,\c) to (\d,\c-1);
			{\ifnum\d=1
				\draw (0,\c) to (0,\c-1);
				\fi}
			\setcounter{r}{\d}
		}
		{\ifnum\c=1
			\draw (0,0)--(\value{r},0);
			\fi}
		\draw(0,\c) to (\value{r},\c);
		\setcounter{s}{\c}}}
		
\newcommand\sTableau[1]{
        \foreach \x [count = \c from 1] in {#1} {
		\foreach \y [count = \d from 1] in \x{
			\node at (\d-.5,\c-.5) {\tiny$\y$}; 
			\draw (\d,\c) to (\d,\c-1);
			{\ifnum\d=1
				\draw (0,\c) to (0,\c-1);
				\fi}
			\setcounter{r}{\d}
		}
		{\ifnum\c=1
			\draw (0,0)--(\value{r},0);
			\fi}
		\draw(0,\c) to (\value{r},\c);
		\setcounter{s}{\c}}}

\newcommand{\PartB}[1]{
 \foreach \x [count=\s from 1] in {#1}{
 	{\ifnum\s=1
		\draw (0,\s-1)--(\x,\s-1); 
		\fi}
   \draw (0,\s) to (\x,\s);
   \foreach \y in {0, ..., \x} {\draw (\y,\s)--(\y,\s-1);}
 }}

\tikzstyle{V}=[draw, fill =black, circle, inner sep=0pt, minimum size=1.5pt]
\tikzstyle{wV}=[draw, fill =white, circle, inner sep=0pt, minimum size=4.5pt]
\tikzstyle{bV}=[draw, fill =black, circle, inner sep=0pt, minimum size=4.5pt]
\tikzstyle{over}=[draw=white,double=black,line width=2pt, double distance=.5pt]

\def\Over[#1,#2][#3,#4]{ 
	\draw[style=over]   (#2,#1) .. controls ++(#4*.5-#2*.5,0) and ++(-#4*.5+#2*.5,0) .. (#4,#3);}
\def\Under[#1,#2][#3,#4]{ 
	\draw  (#2,#1) .. controls ++(#4*.5-#2*.5,0) and ++(-#4*.5+#2*.5,0) .. (#4,#3);}
\def\Cross[#1,#2][#3,#4]{
	\Under[#3,#2][#1,#4]\Over[#1,#2][#3,#4]}

\def\Tops[#1][#2][#3]{
	\foreach\x in {#1}{
		\draw (#2,\x+.15) -- (#2+.1, \x+.15) (#2, \x-.15) -- (#2+.1, \x-.15) ;
		\draw (#2+.1,\x) arc (0:360:.75mm and 1.5mm);}
	\foreach \x in {1,...,#3} {\draw (#2,\x)  to (#2+.05,\x); \node[V] at (#2+.05,\x){};}
	}
\def\Bottoms[#1][#2][#3]{
	\foreach\x in {#1}{
		\draw (#2, \x+.15) -- (#2-.1, \x+.15) (#2, \x-.15) -- (#2-.1, \x-.15) ;
		\draw (#2-.1, \x+.15) arc (90:270:.75mm and 1.5mm);}
	\foreach \x in {1,...,#3} {\draw (#2, \x)  to (#2-.05, \x); \node[V] at (#2-.05, \x){};}
	}
\def\Caps[#1][#2,#3][#4]{
	\Tops[#1][#3][#4]
	\Bottoms[#1][#2][#4]
	}
\def\Pole[#1][#2,#3]{
	\shade[left color=white,right color=white] (#2,#1+.15) rectangle (#3,#1-.15);
	\draw[over] (#2,#1+.15) to (#3,#1+.15) (#2,#1-.15) to (#3,#1-.15) ;}
\def\Label[#1,#2][#3][#4]{
	\node[right] at (#2+.1,#3) {#4};
	\node[left] at (#1-.1,#3) {#4};		}
\def\Nodes[#1][#2]{
	 \foreach \x in {1,...,#2} {\node[V] at (#1,\x){};	}
	}
\def\PoleCaps[#1][#2,#3]{
	\foreach\x in {#1}{
		\draw (#2,\x+.15) -- (#2-.1,\x+.15) (#2,\x-.15) -- (#2-.1,\x-.15) ;
		\draw (#2-.1,\x+.15) arc (0:-180:1.5mm and .75mm);}
	\foreach\x in {#1}{
		\draw (#3,\x+.15) -- (#3+.1,\x+.15) (#3,\x-.15) -- (#3+.1,\x-.15) ;
		\draw (#3+.1,\x+.15) arc (0:360:1.5mm and .75mm);}
	}
\def\PoleTwist[#1,#2]{
	\foreach \x/\y in {-1/1L, -.7/1R, 0/2L, .3/2R}{\coordinate(T\y) at (#2,\x); \coordinate(B\y) at (#1,\x);}
	\draw[thin] (B1R) .. controls ++(#2*.5-#1*.5-.1,0) and ++(-#2*.5+#1*.5-.1,0) ..  (T2R)
			(B1L)   .. controls ++(#2*.5-#1*.5+.1,0) and ++(-#2*.5+#1*.5+.1,0) ..    (T2L) ;
	\draw[line width=2pt, white]
			(#1,.15)  .. controls +(#2*.5-#1*.5,0) and +(-#2*.5+#1*.5,0) ..   (#2,-.85) ;
	\draw[thin,over] 
		(B2R) .. controls ++(#2*.5-#1*.5+.1,0) and ++(-#2*.5+#1*.5+.1,0) ..  (T1R) 
			(B2L)  .. controls +(#2*.5-#1*.5-.1,0) and +(-#2*.5+#1*.5-.1,0) ..   (T1L) ;
			}

\def\SymPolesCaps[#1,#2][#3]{
	\draw (#1,.3) -- (#1-.1,.3) (#1,.15) -- (#1-.1, .15) ;
	\draw (#1-.1, .3) arc (0:-180:2pt and 1.5pt);
	\draw (#1,#3+.7) -- (#1-.1,#3+.7) (#1,#3+.85) -- (#1-.1,#3+.85) ;
	\draw (#1-.1,#3+.85)  arc (0:-180:2pt and 1.5pt);
	\draw (#2,.3) -- (#2+.1, .3) (#2, .15) -- (#2+.1, .15) ;
	\draw (#2+.1, .3) arc (0:360:2pt and 1.5pt);
	\draw (#2, #3+.7) -- (#2+.1, #3+.7) (#2, #3+.85) -- (#2+.1, #3+.85) ;
	\draw (#2+.1, #3+.85) arc (0:360:2pt and 1.5pt);}

\newcommand{\posleq}[1]{
	\hspace{0.1cm}
	\begin{tikzpicture}
	\draw (-0.8ex, -0.5ex) -- (0.8ex, -0.5ex);
	\draw (-0.8ex, 0.4ex) -- (0.7ex, -0.2ex);
	\draw (-0.8ex, 0.4ex) -- (0.7ex, 1ex);
	\draw (0.4ex,0.4ex) --(1.1ex, 0.4ex);
	\draw (0.75ex,0.75ex) --(0.75ex, 0.05ex);
	\end{tikzpicture}
	\hspace{0.1cm}
	}
\newcommand{\negleq}[1]{
	\hspace{0.1cm}
	\begin{tikzpicture}
	\draw (-0.8ex, -0.5ex) -- (0.8ex, -0.5ex);
	\draw (-0.8ex, 0.4ex) -- (0.7ex, -0.2ex);
	\draw (-0.8ex, 0.4ex) -- (0.7ex, 1ex);
	\draw (0.4ex,0.4ex) --(1.1ex, 0.4ex);
	\end{tikzpicture}
	\hspace{0.1cm}
	}
	
\newcommand{\zeroleq}[1]{
	\hspace{0.1cm}
	\begin{tikzpicture}
	\draw (-0.8ex, -0.5ex) -- (0.8ex, -0.5ex);
	\draw (-0.8ex, 0.4ex) -- (0.7ex, -0.2ex);
	\draw (-0.8ex, 0.4ex) -- (0.7ex, 1ex);
	\draw  (0.75ex,0.4ex) ellipse (0.2ex and 0.35ex);
	\end{tikzpicture}
	\hspace{0.1cm}
	}
	
\newcommand{\posgeq}[1]{
	\hspace{0.1cm}
	\begin{tikzpicture}
	\draw (-0.8ex, -0.5ex) -- (0.8ex, -0.5ex);
	\draw (0.8ex, 0.4ex) -- (-0.7ex, -0.2ex);
	\draw (0.8ex, 0.4ex) -- (-0.7ex, 1ex);
	\draw (-0.4ex,0.4ex) --(-1.1ex, 0.4ex);
	\draw (-0.75ex,0.75ex) --(-0.75ex, 0.05ex);
	\end{tikzpicture}
	\hspace{0.1cm}
	}
\newcommand{\neggeq}[1]{
	\hspace{0.1cm}
	\begin{tikzpicture}
	\draw (-0.8ex, -0.5ex) -- (0.8ex, -0.5ex);
	\draw (0.8ex, 0.4ex) -- (-0.7ex, -0.2ex);
	\draw (0.8ex, 0.4ex) -- (-0.7ex, 1ex);
	\draw (-0.4ex,0.4ex) --(-1.1ex, 0.4ex);
	\end{tikzpicture}
	\hspace{0.1cm}
	}
	
\newcommand{\zerogeq}[1]{
	\hspace{0.1cm}
	\begin{tikzpicture}
	\draw (-0.8ex, -0.5ex) -- (0.8ex, -0.5ex);
	\draw (0.8ex, 0.4ex) -- (-0.7ex, -0.2ex);
	\draw (0.8ex, 0.4ex) -- (-0.7ex, 1ex);
	\draw  (-0.75ex,0.4ex) ellipse (0.2ex and 0.35ex);
	\end{tikzpicture}
	\hspace{0.1cm}
	}

\newcommand{\posl}[1]{
	\hspace{0.1cm}
	\begin{tikzpicture}
	\draw (-0.8ex, 0.4ex) -- (0.7ex, -0.2ex);
	\draw (-0.8ex, 0.4ex) -- (0.7ex, 1ex);
	\draw (0.4ex,0.4ex) --(1.1ex, 0.4ex);
	\draw (0.75ex,0.75ex) --(0.75ex, 0.05ex);
	\end{tikzpicture}
	\hspace{0.1cm}
	}
\newcommand{\negl}[1]{
	\hspace{0.1cm}
	\begin{tikzpicture}
	\draw (-0.8ex, 0.4ex) -- (0.7ex, -0.2ex);
	\draw (-0.8ex, 0.4ex) -- (0.7ex, 1ex);
	\draw (0.4ex,0.4ex) --(1.1ex, 0.4ex);
	\end{tikzpicture}
	\hspace{0.1cm}
	}
	
\newcommand{\zerol}[1]{
	\hspace{0.1cm}
	\begin{tikzpicture}
	\draw (-0.8ex, 0.4ex) -- (0.7ex, -0.2ex);
	\draw (-0.8ex, 0.4ex) -- (0.7ex, 1ex);
	\draw  (0.75ex,0.4ex) ellipse (0.2ex and 0.35ex);
	\end{tikzpicture}
	\hspace{0.1cm}
	}
	
\newcommand{\posg}[1]{
	\hspace{0.1cm}
	\begin{tikzpicture}
	\draw (0.8ex, 0.4ex) -- (-0.7ex, 1ex);
	\draw (0.8ex, 0.4ex) -- (-0.7ex, -0.2ex);
	\draw (-0.4ex,0.4ex) --(-1.1ex, 0.4ex);
	\draw (-0.75ex,0.75ex) --(-0.75ex, 0.05ex);
	\end{tikzpicture}
	\hspace{0.1cm}
	}
\newcommand{\negg}[1]{
	\hspace{0.1cm}
	\begin{tikzpicture}
	\draw (0.8ex, 0.4ex) -- (-0.7ex, -0.2ex);
	\draw (0.8ex, 0.4ex) -- (-0.7ex, 1ex);
	\draw (-0.4ex,0.4ex) --(-1.1ex, 0.4ex);
	\end{tikzpicture}
	\hspace{0.1cm}
	}
	
\newcommand{\zerog}[1]{
	\hspace{0.1cm}
	\begin{tikzpicture}
	\draw (0.8ex, 0.4ex) -- (-0.7ex, -0.2ex);
	\draw (0.8ex, 0.4ex) -- (-0.7ex, 1ex);
	\draw  (-0.75ex,0.4ex) ellipse (0.2ex and 0.35ex);
	\end{tikzpicture}
	\hspace{0.1cm}
	}

\makeatletter
\renewcommand{\@makefnmark}{\mbox{\textsuperscript{}}}
\makeatother

\title{Clebsch-Gordan coefficients for Macdonald polynomials}
\author{
Aritra Bhattacharya\quad\ email:\ baritra@imsc.res.in \\
Arun Ram\quad\ \ email:\ aram@unimelb.edu.au \\
\\
}
\date{\today}

\usetikzlibrary{arrows.meta}

\begin{document}

\maketitle



\begin{abstract}
\noindent
In this paper we use the double affine Hecke algebra to compute the Macdonald polynomial
products $E_\ell P_m$ and $P_\ell P_m$ for type $SL_2$ and type $GL_2$ Macdonald polynomials.
Our method follows the ideas of Martha Yip but executes a compression to reduce the sum from 
$2\cdot 3^{\ell-1}$ signed terms to $2\ell$ positive terms.  We show that our rule for $P_\ell P_m$ is equivalent 
to a special case of the Pieri rule of Macdonald.  
Our method shows that computing $E_\ell\mathbf{1}_0$
and $\mathbf{1}_0 E_\ell \mathbf{1}_0$ in terms of a special basis of the double affine Hecke algebra
provides universal compressed formulas for multiplication by $E_\ell$ and $P_\ell$.
The formulas for a specific products $E_\ell P_m$ and $P_\ell P_m$
are obtained by evaluating the universal formulas at $t^{-\frac12}q^{-\frac{m}{2}}$.
\end{abstract}



\section{Introduction}

The type $SL_2$ Macdonald polynomals $P_\ell(x)$ are special cases of the Askey-Wilson polynomlals,
sometimes called the $q$-ultraspherical polynomials (see \cite[p.156-7]{Mac03}).  
The $P_\ell(x)$ are two-parameter $q$-$t$-generalizations of the characters of finite dimensional
representations of $SU_2$; i.e. the characters of $SU_2$ which play a pivotal role in the 
``standard model'' in particle physics and in the analysis of Heisenberg spin chains in mathematical physics.  
The polynomial representation of the double affine Hecke 
algebra, which is the source of the type $SL_2$ Macdonald polynomials, is a generalization of the 
``Dirac sea'', a representation of the Heisenberg algebra which controls the mathematics behind the
quantum harmonic oscillator (see \eqref{etaonpolypic} and Proposition \ref{etaonEnew} and compare to the discussion, 
for example,  in the neighborhood of Figure 10.2 in \cite{SF14}).

As in \cite[\S 6.1 and 6.3]{Mac03}, we denote the electronic (nonsymmetric) Macdonald polynomials for type $SL_2$ 
by $E_m$, for $m\in \ZZ$, and
the bosonic (symmetric) Macdonald polynomials for type $SL_2$ are denoted $P_m$, for $m\in \ZZ_{\ge 0}$
(see \cite[\S1]{CR22} for the motivation for the terminology `electronic' and `bosonic').
The type $SL_2$ Macdonald polynomials are, by a coordinate transformation, ``equivalent'' to the type 
$GL_2$ Macdonald polynomials.  We review this coordinate transformation in Section 2 and explain how 
a product rule for type $SL_2$ Macdonald polynomials translates to a product rule for type $GL_2$ 
Macdonald polynomials.

Following the development of \cite[\S3.4]{CR22}, we use 
the calculus of the bosonic symmetrizer $\mathbf{1}_0$ and the
normalized intertwining operators $\eta_{s_1}$, $\eta_\pi$ 
to compute the elements $E_\ell(X) \mathbf{1}_0$ and $\mathbf{1}_0E_\ell(X)\mathbf{1}_0$
in terms of a special basis of the (localized) double affine Hecke algebra.
Continuing the main conceptual idea of \cite{HR22} the  expansions of the elements
 $E_\ell(X) \mathbf{1}_0$ and $\mathbf{1}_0E_\ell(X)\mathbf{1}_0$
 function as universal formulas for multiplying Macdonald 
polynomials, since they contain enough information to compute arbitrary products
$E_\ell P_m$ and $P_\ell P_m$.  We review this calculus, in our $SL_2$ setting, in section 3.
The use of this calculus enables to cast the product framework from \cite{Yi10} in a form which is
tractable for executing the compression from $2\cdot 3^{\ell-1}$ terms to $2\ell$ terms.

The key computation for the proof of the product rules is done in sections 4 and 5.  In section 4 we use the
basic structural calculus reviewed in section 3 to
compute recursions satsfied by the coefficients of the operators $E_\ell(X)\mathbf{1}_0$ and $\mathbf{1}_0E_\ell(X)\mathbf{1}_0$
when expanded in terms of the $\{ \eta^\ell\mathbf{1}_0\ |\ \ell\in \ZZ\}$ basis of the completed $\mathbf{1}_0$-projected 
double affine Hecke algebra.
In section 5 we solve these recursions to provide product expressions for the coefficients similar to product expressions for binomial coefficients.

Yip \cite[Theorem 4.2 and Theorem 4.4]{Yi10} gives alcove walk expansions of the 
products $E_\ell P_m$ and $P_\ell P_m$.  The illustrative \cite[Example 5.1]{Yi10}
computes the alcove walk expansion of the product $E_3 P_m$ for the $SL_2$ case.  
In this example there are 18 alcove walks which, after simpification, produce 6 terms.  
In general, for the product $E_\ell P_m$ for type $SL_2$, the alcove walk expansion of Yip will be a sum over
$2\cdot 3^{\ell-1}$ alcove walks which simplifies to $2\ell$ terms.

The result of Theorem \ref{finalthm} of this paper provides an explicit closed formula for each of the $2\ell$ terms which
appear in the expansion of $E_\ell P_m$.
To our knowledge, this formula for $E_\ell P_m$ is new, 
particularly the execution of the desired compression after executing
the general double affine Hecke algebra method of deriving product rules given by Yip \cite{Yi10}.
The $q$-$t$-binomial coefficients are given by
\begin{equation}
\genfrac[]{0pt}{0}{\ell}{j}_{q,t} 
= \frac{ \frac{(q;q)_\ell}{(t;q)_\ell} } { \frac{(q;q)_j}{(t;q)_j} \frac{(q;q)_{\ell-j}}{(t;q)_{\ell-j}} },
\qquad\hbox{where}\qquad
(a;q)_j = (1-a)(1-qa)(1-q^2a)\cdots (1-q^{j-1}a).
\label{introqtbin}
\end{equation}
Then Theorem \ref{finalthm} proves that, for $\ell, m\in \ZZ_{>0}$,
\begin{align*}
P_\ell P_m &= \sum_{j=0}^\ell c_j^{(\ell)}(q^m)\, P_{m+\ell-2j},
\\
E_\ell P_m &= \sum_{j=0}^{\ell-1} a_j^{(\ell)}(q^m) E_{m+\ell-2j}+b_j^{(\ell)}(q^m) E_{-m+\ell-2j},
\qquad\hbox{and}
\\
E_{-\ell} P_m &= \sum_{j=0}^{\ell} t\cdot b_j^{(\ell+1)}(q^m) E_{m-(\ell-2j)}+a_j^{(\ell+1)}(q^m) E_{-m-(\ell-2j)},
\end{align*}
where
\begin{align}
c_j^{(\ell)}(q^m) 
&= 
\genfrac[]{0pt}{0}{\ell}{j}_{q,t} 
\frac{(q^mq^{-(j-1)};q)_j}{(tq^mq^{-j};q)_j}
\frac{(t^2q^m q^{\ell-2j};q)_j}{(tq^mq^{\ell-2j+1};q)_j},
\label{cdefn}
\\
a_j^{(\ell)}(q^m) 
&= c_j^{(\ell)}(q^m)\cdot \frac{(1-q^{\ell-j})}{(1-q^\ell)}\cdot \frac{(1-t q^m q^{\ell-j})}{(1-t q^m q^{\ell-2j})}
\qquad\hbox{and}
\nonumber
\\
b_j^{(\ell)}(q^m) 
&= c_{\ell-j}^{(\ell)}(q^m)\cdot q^{j}\cdot \frac{(1-q^{\ell-j})}{(1-q^\ell)}\cdot \frac{(1-t q^m q^{-(\ell-j)})}{(1-t^2 q^m q^{-(\ell-2j)}) }
\nonumber
\end{align}
The rule for the product $P_\ell P_m$ is equivalent to a special case of the Pieri formula given in 
\cite[Ch.\ VI (6.24)]{Mac}.  The precise connection is derived in Proposition \ref{convtoMac}, reproducing work of Soojin Cho
\cite{Cho19}.  Indeed, the rule for the product $P_\ell P_m$ is the ``linearization formula'' for $q$-ultraspherical polynomials and 
appears in \cite[Theorem 13.3.2]{Is05}, where it is stated that it is an identity of Rogers from 1894.

We have taken some care to try to make our exposition so that it contains all necessary definitions
and complete and thorough proofs of the results.  Our goal, in hope that the powerful tools provided
by the double affine Hecke will become broadly accessible and utilized, has been to make this paper so that 
it can be read from scratch with no previous knowledge 
of Macdonald polynomials or the double affine Hecke algebra.  Section 7 provides explicit examples.




\smallskip\noindent
\textbf{Acknowledgements.}  We thank Beau Anasson, Yifan Guo and Haris Rao for conversations and
Mathematica computations which made it possible to establish the product formula for
$D_j^{(\ell)}(Y)$ which appears in \eqref{tildeDdefn}.   Thank you to Soojin Cho, Charles Dunkl, Dennis Stanton, and Ole Warnaar
for helpful references on the orthogonal polynomial literature and to Jean-\'Emile Bourgine, Sasha Garbali and Jasper Stokman for very helpful comments to improve the exposition.
We thank Institute of Mathematical Sciences Chennai 
for support which enabled A.\ Bhattacharya
to visit University of Melbourne in February-March 2023.


\section{Macdonald polynomials for $SL_2$ and for $GL_2$}

In this section we introduce the electronic and bosonic Macdonald polynomials for types $SL_2$ and $GL_2$
and explain the relation between them.  We show how product rules for type $SL_2$ Macdonald polynomials
convert to product rules for type $GL_2$ Macdonald polynomials.  In Section \ref{Maccomp} we check that the Pieri rule 
for multiplying bosonic polynomials given in \cite[Ch.\ VI (6.24)]{Mac} matches with the rule for the product $P_\ell P_m$
stated in the introduction (and proved in Theorem \ref{finalthm}).

For working with Macdonald polynomials, fix $q,t\in \CC^\times$ such that the only pair of integers $(a,b)$ for which
$q^a t^b = 1$ is the pair $(a,b) = (0,0)$.  Alternatively, one may think of $q$ and $t$ as parameters
and to work with polynomials over the coefficient ring $\CC(q,t)$ instead of over the coefficient ring $\CC$.

\subsection{Macdonald polynomials for type $SL_2$}

The electronic Macdonald polynomials for type $SL_2$, 
$$E_\ell(x) \in \CC[x,x^{-1}],\quad\hbox{are indexed by}\quad \ell\in \ZZ,$$
and the bosonic Macdonald polynomials for type $SL_2$,
$$P_\ell(x) \in \CC[x,x^{-1}],\quad\hbox{are indexed by} \quad \ell\in \ZZ_{\ge0}.$$
Let $\ell\in \ZZ_{\ge 0}$.  Using the notation of \eqref{introqtbin}, the electronic Macdonald polynomials are given by
$$ E_{-\ell}(x) =  \sum_{j=0}^\ell \genfrac[]{0pt}{0}{\ell}{j}_{q,t}
\frac{(1-tq^j)}{(1-tq^{\ell})} x^{\ell-2j}
\qquad\hbox{and}\qquad
E_{\ell}(x) = \sum_{j=0}^{\ell-1} \genfrac[]{0pt}{0}{\ell-1}{j}_{q,t}
\frac{q^{\ell-1-j} (1-tq^j)}{(1-tq^{\ell-1})} x^{-\ell+2j+2},
$$
and the bosonic Macdonald polynomials are given by 
$$
P_\ell(x) =  \sum_{j=0}^\ell \genfrac[]{0pt}{0}{\ell}{j}_{q,t} x^{\ell-2j}.
$$
See \S \ref{qtbintoMac} for the connection between these formulas and 
the formulas in \cite[(6.2.7), (6.2.8), (6.3.7)]{Mac03}.

%

\subsection{Macdonald polynomials for type $GL_2$}

The electronic Macdonald polynomials for type $GL_2$,
$$E_{(\mu_1,\mu_2)}(x_1,x_2)\in \CC[x_1,x_1^{-1},x_2,x^{-1}_2],
\quad\hbox{are indexed by $(\mu_1,\mu_2)\in \ZZ^2$,}
$$
and the bosonic Macdonald polynomials for type $GL_2$,
$$P_{(\lambda_1,\lambda_2)}(x_1,x_2)\in \CC[x_1,x_1^{-1},x_2,x^{-1}_2],
\quad\hbox{are indexed by 
$\lambda = (\lambda_1, \lambda_2)$ with $\lambda_1,\lambda_2\in \ZZ$ and $\lambda_1\ge \lambda_2$.}
$$
The $E_{(\mu_1,\mu_2)}(x_1,x_2)$ and $P_{(\lambda_1,\lambda_2)}(x_1,x_2)$ are given,
in terms of the Macdonald polynomials for type $SL_2$, by 
\begin{align}
E_{(\mu_1,\mu_2)}(x_1,x_2) 
&= (x_1^{\frac12}x_2^{\frac12})^{\mu_1+\mu_2} E_{\mu_1-\mu_2}(x_1^{\frac12}x_2^{-\frac12})
\qquad\hbox{and}
\nonumber
\\
P_{(\lambda_1,\lambda_2)}(x_1,x_2) 
&= (x_1^{\frac12}x_2^{\frac12})^{\lambda_1+\lambda_2} P_{\lambda_1-\lambda_2}(x_1^{\frac12}x_2^{-\frac12}).
\label{SL2toGL2}
\end{align}
Equivalently, if $m_1, m_2\in \frac12\ZZ$ then
$(x_1^{\frac12}x_2^{\frac12})^{2m_2}E_{2m_1}(x_1^{\frac12}x_2^{-\frac12}) 
= E_{(m_1+m_2, -m_1+m_2)}(x_1, x_2).$
Another way to express this conversion is to let
\begin{equation}
\hbox{$y=x_1^{\frac12}x_2^{\frac12}$ and 
$x=x_1^{\frac12}x_2^{-\frac12}$}
\qquad\hbox{so that}\qquad
\hbox{$x_1=yx$ and $x_2 = yx^{-1}$}.
\label{varchange}
\end{equation}
Then the Macdonald polynomials for type $SL_2$ are given in terms of the Macdonald polynomials for type $GL_2$ by
\begin{equation}
E_{2m_1}(x) 
= y^{-2m_2} E_{(m_1+m_2, -m_1+m_2)}(yx,yx^{-1})
=E_{(m_1+m_2, -m_1+m_2)}(x,x^{-1}),
\label{GL2toSL2}
\end{equation}
for $m_1, m_2\in \frac12\ZZ$.  The following picture illustrates the conversion between $E_m$ and $E_{(\mu_1,\mu_2)}$ given by 
the formulas \eqref{GL2toSL2} and \eqref{SL2toGL2}.
	
		\begin{tikzpicture}[xscale=1.00, yscale=1.00]
			\foreach \x in {-4,...,4}
			\foreach \y in {-4,...,4}
			{
				\filldraw[gray] (\x,\y) circle (2pt);
			}
			\draw[thick, magenta] (-4.4,-4.4) -- (4.4,4.4);
			\draw[thick, magenta] (-4.4,4.4) -- (4.4,-4.4);
			\foreach \x in {-8,...,8}
			\filldraw [black] (\x/2, -\x/2) circle (2pt); 
			\foreach \x in {-8,...,8}
			\draw (\x/2-0.25,-\x/2-0.25) node{$E_{\x}$};
			\draw[thick, gray] (-4.4,0) -- (4.4,0); 
			\draw[thick, gray] (0,4.4) -- (0,-4.4); 
			\filldraw[blue] (1,1) circle (2pt);
			\draw[blue] (1,1.25) node{$E_{(1,1)}$}; 
			\filldraw[blue] (1,3) circle (2pt);
			\draw[blue] (1,3.25) node{$E_{(1,3)}$}; 
			\filldraw[blue] (-3,2) circle (2pt);
			\draw[blue] (-3.25,1.75) node{$E_{(-3,2)}$}; 
			\filldraw[blue] (-3,0) circle (2pt);
			\draw[blue] (-3,-0.25) node{$E_{(-3,0)}$}; 
			\filldraw[blue] (-1,-2) circle (2pt);
			\draw[blue] (-1,-2.25) node{$E_{(-1,-2)}$}; 
			
			\filldraw[blue] (4,1) circle (2pt);
			\draw[blue] (4,1.25) node{$E_{(\mu_1,\mu_2)}$};
			
%
%
%
\begin{scope}[<->, shorten >=2pt, shorten <=2pt, every node/.style={fill=white,inner sep=2pt}]
\draw (0.1,0.1) to node{\footnotesize$m_1$} (1.6,-1.4);
\draw (1.55,-1.45) to node{\footnotesize$m_2$} (3.95,0.95);
\draw (4,0.95) to node{\footnotesize$\mu_2$} (4,0.05);
\draw (4,0.1) to node{\footnotesize$\mu_1$} (0,0.1);
\end{scope}
			
			\draw (7,3.25) node{$(\mu_1,\mu_2) \in \ZZ^2$,};
			\draw (7,2.75) node{$m_1,m_2 \in \frac12\ZZ$,};
			\draw (7,0.25) node{$\mu_1 = m_1 + m_2,$};
			\draw (7,-0.25) node{$\mu_2 = m_1-m_2,$};
			\draw (7,-2.75) node{$m_1 = \frac12(\mu_1-\mu_2),$};
			\draw (7,-3.25) node{$m_2 = \frac12(\mu_1+\mu_2),$};
		\end{tikzpicture}

\subsection{Converting product rules for type $SL_2$ to product formulas for type $GL_2$}

Assume that multiplication rules for multiplying type $SL_2$ Macdonald polynomials are given by
$$E_\ell(x) P_m(x) = \sum_{j=0}^{\ell-1} a_j^{(\ell)}(q^m) E_{m+\ell-2j}(x)
+\sum_{j=0}^{\ell-1} b_j^{(\ell)}(q^m) E_{-m+\ell-2j}(x),
$$
and
$$
P_\ell(x) P_m(x) = \sum_{j=0}^\ell c_j^{(\ell)}(q^m) E_{m+\ell-2j}(x).
$$
Assume $(\nu_1, \nu_2)\in \ZZ^2$ and $(\mu_1, \mu_2)\in \ZZ^2$ with $\mu_1\ge \mu_2$.  Let
$$\ell = \nu_1-\nu_2\qquad\hbox{and}\qquad m = \mu_1-\mu_2
\qquad\hbox{and}\qquad
d = \mu_1+\mu_2+\nu_1+\nu_2.
$$
Then
$$\begin{array}{c}
m+\ell-2j+d = 2(\mu_1+\nu_1-j), \\
-(m+\ell-2j)+d = 2(\mu_2+\nu_2+j),
\end{array}
\qquad\hbox{and}\qquad
\begin{array}{c}
-m+\ell-2j+d = 2(\mu_2+\nu_1-j), \\
-(-m+\ell-2j)+d = 2(\mu_1+\nu_2+j).
\end{array}
$$
Thus, with $y=x_1^{\frac12}x_2^{\frac12}$ and 
$x=x_1^{\frac12}x_2^{-\frac12}$ as in \eqref{varchange}, the conversions in \eqref{SL2toGL2} and \eqref{GL2toSL2}
give
\begin{align*}
P_{(\nu_1,\nu_2)}&(x_1,x_2)P_{(\mu_1,\mu_2)}(x_1,x_2)
=y^{\nu_1+\nu_2} P_{\ell}(x)
y^{\mu_1+\mu_2} P_{m}(x)
= \sum_{j=0}^{\nu_1-\nu_2} c_j^{(\ell)}(q^m) y^d P_{m+\ell-2j}(x)
\\
&= \sum_{j=0}^{\nu_1-\nu_2} c_j^{(\nu_1-\nu_2)}(q^{\mu_1-\mu_2}) P_{(\mu_1+\nu_1-j, \mu_2+\nu_2+j)}(x_1,x_2),
\end{align*}
and
\begin{align*}
&E_{(\nu_1,\nu_2)}(x_1,x_2)P_{(\mu_1,\mu_2)}(x_1,x_2)
=y^{\nu_1+\nu_2} E_{\nu_1-\nu_2}(x)
y^{\mu_1+\mu_2} P_{\mu_1-\mu_2}(x)
= y^d E_\ell(x) P_m(x)
\\
&= \sum_{j=0}^{\nu_1-\nu_2-1} a_j^{(\ell)}(q^m) y^d
E_{m+\ell-2j}(x) 
+ \sum_{j=0}^{\nu_1-\nu_2-1} b_j^{(\ell)}(q^m) y^d
E_{-m+\ell-2j}(x)
\\
&= \sum_{j=0}^{\nu_1-\nu_2-1} a_j^{(\nu_1-\nu_2)} (q^{\mu_1-\mu_2})\, E_{(\mu_1+\nu_1-j, \mu_2+\nu_2+j)}(x_1,x_2)
\\
&\qquad
+ \sum_{j=0}^{\nu_1-\nu_2-1} b_j^{(\nu_1-\nu_2)}(q^{\mu_1-\mu_2})\, E_{(\mu_2+\nu_1-j,\mu_1+\nu_2+j)}(x_1,x_2),
\end{align*}
and these are the multiplication rules for type $GL_2$ Macdonald polynomials.

%

\subsection{Comparison of the $GL_2$ case to Macdonald}\label{Maccomp}

A horizontal strip $\lambda/\mu$ of length $\ell$ is a pair
$\lambda = (\lambda_1,\lambda_2)$ and $\mu= (\mu_1,\mu_2)$ of partitions such that
{\def\A{8} \def\B{3} \def\C{2} \def\K{10} \def\H{.65} \def\L{7} 
$$\mu_2\le \lambda_2\le \mu_1\le \lambda_1
\quad\hbox{and}\quad
\lambda_1-\mu_1+\lambda_2-\mu_2=\ell.
\qquad\qquad
\TikZ{[xscale=.4,yscale=-.4]
		\filldraw[black!20] (0,0) to (\A + \B - \C,0) to (\A + \B - \C,1) to (\C,1) to (\C,2) to (0,2) to (0,0);
		\draw (0,0) rectangle (\A + \B + \K - \L,1);
		\draw (0,1) to (0,2) to (\L,2) to (\L,1);
		\draw (\A + \B - \C,0) to (\A + \B - \C,1) (\C,2) to (\C,1); 
		\begin{scope}[<->, shorten >=2pt, shorten <=2pt, every node/.style={fill=white,inner sep=2pt}]
		\draw (0,2+\H) to node{\footnotesize$\mu_2$} (\C,2+\H);
		\draw (\C,2.7) to node{\footnotesize$\lambda_2-\mu_2$} (\L,2.7);
		\draw(0,-\H) to node{\footnotesize$\mu_1$} (\A + \B - \C,-\H);
		\draw(\A + \B - \C,-\H) to node{\footnotesize$\lambda_1-\mu_1$} (\A + \B + \K - \L,-\H);
		\end{scope}
		} $$}
Following \cite[VI \S6 Ex. 2a]{Mac}, define
\begin{align*}
\varphi_{\lambda/\mu} 
&= \prod_{1\le i\le j\le \ell(\lambda)} 
\frac{f(q^{\lambda_i-\lambda_j}t^{j-i})}{f(q^{\lambda_i-\mu_j}t^{j-i})}
\frac{f(q^{\mu_i-\mu_{j+1}}t^{j-i}) }{ f(q^{\mu_i-\lambda_{j+1}}t^{j-i})},
\qquad\hbox{where}\qquad
f(u) = \frac{(tu;q)_\infty}{(qu;q)_\infty}
\end{align*}
and $(z,q)_\infty = (1-z)(1-zq)(1-zq^2)\cdots$.
Then \cite[VI \S6 Ex. 2a]{Mac} gives that
$$\hbox{if}\qquad g_\ell = \frac{(t,q)_\ell}{(q;q)_\ell} P_{(\ell,0)}(x_1,x_2)
\qquad\hbox{then}\qquad
g_\ell P_{(\mu_1,\mu_2)}  = \sum_{\lambda} \varphi_{\lambda/\mu} P_{(\lambda_1,\lambda_2)},
$$
where the sum is over $\lambda = (\lambda_1,\lambda_2)$ such that $\lambda/\mu$ is a horizontal strip
of length $\ell$.
Indeed this matches our results, in view of the following Proposition.

\begin{prop} \label{convtoMac}
Let $c_j^{(\ell)}(q^m)$ be as defined in \eqref{cdefn}.  Let $\lambda=(\lambda_1,\lambda_2)$
and $\mu=(\mu_1,\mu_2)$ be such that $\lambda/\mu$ is a horizontal strip of length $\ell$ and let
$m=\mu_1-\mu_2$ and $j = \lambda_2-\mu_2$. Then
$$\frac{(q,q)_\ell}{(t;q)_\ell} \varphi_{\lambda/\mu} = c_j^{(\ell)}(q^m).
$$
\end{prop}
\begin{proof}
Letting $\lambda_i = \mu_i + a_i$ gives
\begin{align*}
\varphi_{\lambda/\mu} 
&= \prod_{1\le i\le j\le \ell(\lambda)} 
\frac{f(q^{a_i-a_j}q^{\mu_i-\mu_j} t^{j-i})}{f(q^{a_i}q^{\mu_i-\mu_j} t^{j-i})}
\frac{f(t^{-1}q^{\mu_i-\mu_{j+1}}t^{j+1-i} ) }{f(t^{-1}q^{-a_{j+1}}q^{\mu_i-\mu_{j+1}}t^{j+1-i})}
\\
&= \prod_{1\le i\le j\le \ell(\lambda)} 
\frac{(tq^{a_i-a_j}q^{\mu_i-\mu_j}t^{j-i};q)_{\infty}}
{(qq^{a_i-a_j}q^{\mu_i-\mu_j}t^{j-i};q)_{\infty}}
\frac{(qq^{a_i}q^{\mu_i-\mu_j}t^{j-i};q)_{\infty}}
{(tq^{a_i}q^{\mu_i-\mu_j}t^{j-i};q)_{\infty}}
\\
&\qquad
\cdot
\frac{(q^{\mu_i-\mu_{j+1}}t^{j+1-i};q)_\infty}
{(qt^{-1}q^{\mu_i-\mu_{j+1}}t^{j+1-i};q)_\infty}
\frac{(qt^{-1}q^{-a_{j+1}}q^{\mu_i-\mu_{j+1}}t^{j+1-i};q)_\infty}
{(q^{-a_{j+1}}q^{\mu_i-\mu_{j+1}}t^{j+1-i};q)_\infty}
\\
&= \prod_{1\le i\le j\le \ell(\lambda)} 
\frac{(tq^{a_i-a_j}q^{\mu_i-\mu_j}t^{j-i};q)_{a_j}}
{(qq^{a_i-a_j}q^{\mu_i-\mu_j}t^{j-i};q)_{a_j}}
\frac{(qt^{-1}q^{-a_{j+1}}q^{\mu_i-\mu_{j+1}}t^{j+1-i};q)_{a_{j+1}}}
{(q^{-a_{j+1}}q^{\mu_i-\mu_{j+1}}t^{j+1-i};q)_{a_{j+1}}}
\end{align*}
When $i=j$ the first factor is
$$\frac{(tq^{a_i-a_j}q^{\mu_i-\mu_j}t^{j-i};q)_{a_j}}
{(qq^{a_i-a_j}q^{\mu_i-\mu_j}t^{j-i};q)_{a_j}}
=\frac{(t;q)_{a_j}}{(q;q)_{a_j}},
$$
and when $j+1=n$ so that $a_{j+1}=0$ and $\mu_{j+1}=0$.
$$
\frac{(qt^{-1}q^{-a_{j+1}}q^{\mu_i-\mu_{j+1}}t^{j+1-i};q)_{a_{j+1}}}
{(q^{-a_{j+1}}q^{\mu_i-\mu_{j+1}}t^{j+1-i};q)_{a_{j+1}}}
=
\frac{(qt^{-1}q^{\mu_i}t^{j+1-i};q)_0}
{(q^{\mu_i}t^{j+1-i};q)_0}=1.
$$
Thus, when $\ell(\lambda)=2$,
\begin{align*}
\varphi_{\lambda/\mu} 
&= \frac{(t;q)_{a_1}}{(q;q)_{a_1}}\frac{(t;q)_{a_2}}{(q;q)_{a_2}}
\cdot \frac{(tq^{a_1-a_2}q^{\mu_1-\mu_2}t^{2-1};q)_{a_2}}{(q^{a_1-a_2+1}q^{\mu_1-\mu_2}t^{2-1};q)_{a_2}}
\cdot \frac{(t^{-1}q^{-a_2+1}q^{\mu_1-\mu_2}t^{2-1};q)_{a_2}}{(q^{-a_2}q^{\mu_1-\mu_2}t^{2-1};q)_{a_2}}.
\end{align*}
Since $m=\mu_1-\mu_2$ and $j= \lambda_2-\mu_2 = a_2$ then
$a_1-a_2 = (a_1+a_2)-2a_2 = \ell-2j$ and
\begin{align*}
\frac{(q;q)_\ell}{(tq;q)_\ell}
\varphi_{\lambda/\mu} 
&=
\genfrac[]{0pt}{0}{\ell}{j}_{q,t}
\cdot \frac{(t^2q^m q^{\ell-2j};q)_j} {(tq^m q^{\ell-2j+1};q)_j}
\frac{(q^m q^{-(j-1)};q)_j} {(tq^m q^{-j};q)_j}
=c_j^{(\ell)}(q^m).
\end{align*}
\end{proof}


\section{DAHA for $SL_2$ and the polynomial representation}

In this section we introduce the type $SL_2$ double affine Hecke algebra and its polynomial
representation.  The double affine Hecke algebra is a source for a myriad of operators acting on 
polynomials.  In this section we carefully establish the identites between operators that will
enable us to compute products of Macdonald polynomials.

\subsection{The double affine Hecke algebra (DAHA) for type $SL_2$}\label{DAHAdefn}

Fix $q^{\frac12}, t^{\frac12}\in \CC^\times$.  
Following \cite[(6.1.2), (6.1.3)]{Mac03}, the double affine Hecke algebra for $SL_2$ is the $\CC$-algebra
$\tilde H_{\mathrm{int}}$ generated by
$T^{\pm1}_1, X^{\pm1}, Y^{\pm1}, T^{\pm1}_\pi$ with relations $T_1T_1^{-1} = T_1^{-1}T_1 = 1$,
$XX^{-1}=X^{-1}X=1$, $YY^{-1}=Y^{-1}Y=1$, $T_\pi T^{-1}_\pi=T^{-1}_\pi T_\pi=1$ and
$$T_\pi =YT_1^{-1} = T_1Y^{-1}, 
\qquad T_\pi X T_\pi^{-1} = q^{\frac12}X^{-1},
$$
\begin{equation}
T_1XT_1 = X^{-1}, \qquad T_1Y^{-1}T_1 = Y,
\qquad (T_1-t^{\frac12})(T_1+t^{-\frac12}) = 0.
\label{Hdefn}
\end{equation}
It follows from the relations $T_1XT_1 = X^{-1}$ and $T_1-T^{-1}_1 = t^{\frac12}-t^{-\frac12}$ that
\begin{equation}
T_1X^r = X^{-r}T_1 + (t^{\frac12}-t^{-\frac12})\frac{X^r-X^{-r}}{1-X^2},
\qquad\hbox{for $r\in \ZZ$.}
\label{TpastX}
\end{equation}

As a left module for the Laurent polynomial ring $\CC[Y,Y^{-1}]$, the double affine Hecke algebra
$\tilde H_{\mathrm{int}}$ has basis
$\{X^k\ |\ k\in \ZZ\} \sqcup \{ X^kT_1\ |\ k\in \ZZ\}$.  Letting $\CC(Y)$ denote the field of fractions of
$\CC[Y,Y^{-1}]$, the \emph{localized double affine Hecke algebra},
$$\tilde H = \CC(Y)\otimes_{\CC[Y,Y^{-1}]} \tilde H_{\mathrm{int}},$$
is the algebra with $\CC(Y)$-basis  $\{X^k\ |\ k\in \ZZ\} \sqcup \{ X^kT_1\ |\ k\in \ZZ\}$ (as a left $\CC(Y)$-module) and
the relations in \eqref{Hdefn}.  Although the polynomial representation of $\tilde H_{\mathrm{int}}$ 
(which is where the Macdonald polynomials live, see \S \ref{sec:polyrep}) 
is not a $\tilde H$-module, there are operators
on the polynomial representation which we can source from the larger algebra $\tilde H$.  The operators
which we wish to access are the intertwiners $\tau^\vee_\pi$ and $\tau^\vee_1$ and
the normalized interwiners $\eta_{s_1}$, $\eta_\pi$, $\eta$ and $\eta^{-1}$, which are defined below in 
\eqref{intdefn}, \eqref{nintdefn}, \eqref{etadefn} and \eqref{etainvdefn}.

\subsubsection{Intertwiners and the bosonic symmetrizer}

The \emph{intertwiners} $\tau^\vee_1$ and $\tau^\vee_\pi$ and 
the \emph{bosonic symmetrizer} $\mathbf{1}_0$ are defined by
\begin{equation}
\tau^\vee_1
= T_1+t^{-\frac12}\frac{(1-t)}{(1-Y^{-2})},\qquad
\tau^\vee_\pi  = XT_1,
\qquad\hbox{and}\qquad
\mathbf{1}_0 = T_1+t^{-\frac12}.
\label{intdefn}
\end{equation}
In \cite[(6.1.6) and (6.18)]{Mac03}, $\tau^\vee_1$ is denoted $\mathbf{\alpha}$ and $\tau^\vee_\pi$ is denoted $\mathbf{\beta}$.

\begin{prop}
Then
\begin{align}
\tau^\vee_1 Y &= Y^{-1}\tau^\vee_1,
\qquad
&\tau^\vee_\pi Y &= Y^{-1} q^{-\frac12} \tau^\vee_\pi,
\label{taupastY}
\\
(\tau^\vee_1)^2  
&=t^{-1} \frac{(1-tY^2)(1-tY^{-2})}{(1-Y^2)(1-Y^{-2})},
\qquad
&(\tau^\vee_\pi)^2 &= 1,
\label{tausq}
\end{align}
\begin{equation}
T_1 \mathbf{1}_0 = \mathbf{1}_0 T_1 = t^{\frac12} \mathbf{1}_0,
\qquad
\mathbf{1}_0\tau^\vee_1 = \mathbf{1}_0 t^{-\frac12} \frac{(1-tY^{-2})}{(1-Y^{-2})},
\qquad
\tau^\vee_1 \mathbf{1}_0 = t^{-\frac12} \frac{(1-tY^{-2})}{(1-Y^{-2})}\mathbf{1}_0,
\label{10rels}
\end{equation}
\begin{align}
\mathbf{1}_0^2=\mathbf{1}_0(t^{\frac12}+t^{-\frac12}) \qquad\hbox{and}\qquad
\mathbf{1}_0 
=\tau^\vee_1+t^{-\frac12}\frac{(1-tY^2)}{(1-Y^2)}
=\tau^\vee_1+t^{\frac12}\frac{(1-t^{-1}Y^{-2})}{(1-Y^{-2})}.
\label{10intau}
\end{align}
\end{prop}
\begin{proof}  Using the relations in \eqref{Hdefn},
$(\tau^\vee_\pi)^2 = XT_1XT_1 = XX^{-1} = 1$ and
\begin{align*}
\tau^\vee_\pi Y 
&= XT_1 Y=T_1^{-1}X^{-1}Y 
= T_1^{-1}q^{-\frac12}T_\pi X T_\pi^{-1}Y \\
&= T_1^{-1}q^{-\frac12}T_\pi X T_\pi^{-1}T_\pi T_1 
= q^{-\frac12}T_1^{-1}T_\pi X T_1
= q^{-\frac12} T_1^{-1}T_\pi \tau^\vee_\pi
= q^{-\frac12} Y^{-1} \tau^\vee_\pi.
\end{align*}
Using
\begin{align}
\tau^\vee_1
&= T_1+t^{-\frac12}\frac{(1-t)}{(1-Y^{-2})}
= T_1^{-1} +(t^{\frac12}-t^{-\frac12})+t^{-\frac12}\frac{(1-t)}{(1-Y^{-2})} 
\nonumber \\
&=T_1^{-1} +(1-t) \frac{-t^{-\frac12}(1-Y^{-2})+t^{-\frac12} }{(1-Y^{-2})}
= T_1^{-1} + t^{-\frac12}\frac{(1-t)Y^{-2}}{(1-Y^{-2})},
\label{tau1Tinv}
\end{align}
then
\begin{align*}
\tau^\vee_1 Y 
&= \Big(T_1^{-1}+t^{-\frac12}\frac{(1-t)Y^{-2}}{(1-Y^{-2})}\Big)Y
= T_1^{-1}Y+t^{-\frac12}\frac{(1-t)Y^{-1}}{(1-Y^{-2})} 
\\
&= Y^{-1}T_1+t^{-\frac12}\frac{(1-t)Y^{-1}}{(1-Y^{-2})}
= Y^{-1}\Big(T_1+t^{-\frac12}\frac{(1-t)}{(1-Y^{-2})} \Big)
=Y^{-1}\tau^\vee_1
\end{align*}
and
\begin{align*}
(\tau^\vee_1)^2
&= \Big(T_1+t^{-\frac12}\frac{(1-t)}{(1-Y^{-2})}\Big)\tau^\vee_1
= T_1\tau^\vee_1+\tau^\vee_1 t^{-\frac12}\frac{(1-t)}{(1-Y^{2})}
\\
&= T_1 \Big(T_1^{-1} + t^{-\frac12}\frac{(1-t)Y^{-2}}{(1-Y^{-2})}\Big)
+\Big(T_1+t^{-\frac12}\frac{(1-t)}{(1-Y^{-2})}\Big)t^{-\frac12}\frac{(1-t)}{(1-Y^{2})}
\\
&=1+t^{-1}\frac{(1-t)}{(1-Y^{-2})}\frac{(1-t)}{(1-Y^{2})}
=\frac{(1-Y^{-2}-Y^2+1+t^{-1}-2+t)}{(1-Y^2)(1-Y^{-2})}
\\
&=t^{-1} \frac{(1-tY^2)(1-tY^{-2})}{(1-Y^2)(1-Y^{-2})}.
\end{align*}
Since
$\mathbf{1}_0 = T_1+t^{-\frac12} = T_1-T^{-1}_1+T^{-1}_1 +t^{-\frac12} 
= (t^{\frac12}-t^{-\frac12})+T^{-1}_1 +t^{-\frac12} = T^{-1}_1+t^{\frac12}
$
then
$$\mathbf{1}_0T_1 = (T^{-1}_1+t^{\frac12})T_1 = 1+t^{\frac12}T_1 
= t^{\frac12}(T_1+t^{-\frac12}) = t^{\frac12}\mathbf{1}_0
\quad\hbox{and}\quad
\mathbf{1}_0^2 = \mathbf{1}_0(T_1+t^{-\frac12}) = \mathbf{1}_0(t^{\frac12}+t^{-\frac12}).
$$
Similarly for the product $T_1\mathbf{1}_0$.  Then
$$\mathbf{1}_0\tau_1^\vee 
= \mathbf{1}_0\Big(T_1+t^{-\frac12}\frac{(1-t)}{(1-Y^{-2})}\Big)
= \mathbf{1}_0\Big(t^{\frac12}+t^{-\frac12}\frac{(1-t)}{(1-Y^{-2})}\Big)
= \mathbf{1}_0 t^{-\frac12}\frac{(1-tY^{-2})}{(1-Y^{-2})}
$$
and similarly for the product $\tau^\vee_1\mathbf{1}_0$.
Finally
$$
\mathbf{1}_0 
= \tau^\vee_1 -t^{-\frac12}\frac{(1-t)}{(1-Y^{-2})} +t^{-\frac12}
=\tau^\vee_1+t^{-\frac12}\frac{(t-Y^{-2})}{(1-Y^{-2})}
=\tau^\vee_1+t^{-\frac12}\frac{(1-tY^2)}{(1-Y^2)}.
$$
\end{proof}

\subsubsection{Normalized intertwiners}

Define normalized intertwiners
\begin{equation}
\eta_\pi = \tau^\vee_\pi
\qquad\hbox{and}\qquad
\eta_{s_1} = t^{\frac12} \frac{(1-Y^{-2})}{(1-tY^{-2})}\tau^\vee_1.
\label{nintdefn}
\end{equation}
Then define
\begin{align}
\eta &= \eta_\pi \eta_{s_1} = \tau^\vee_\pi t^{\frac12}\frac{(1-Y^{-2})}{(1-tY^{-2})} \tau^\vee_1
= t^{\frac12}\frac{(1-Y^{2}q)}{(1-tY^{2}q)}\tau^\vee_\pi \tau^\vee_1
\label{etadefn}
\\
\hbox{and}\qquad\qquad
\eta^{-1} &= \eta_{s_1}\eta_\pi = t^{\frac12}\frac{(1-Y^{-2})}{(1-tY^{-2})} \tau^\vee_1\tau^\vee_\pi.
\label{etainvdefn}
\end{align}

\noindent
\textbf{Warning:} Although $\eta$ and $\eta^{-1}$ are inverses of each other as elements of $\tilde H$,
and these are well defined operators on the polynomial representation (see Proposition \ref{etaonEnew}),
as operators on the polynomial representation $\eta$ and $\eta^{-1}$ are not invertible operators.

\begin{prop} The following relations hold in $\tilde H$:
\begin{equation}
\eta_\pi^2 = 1, \qquad \eta_{s_1}^2 = 1, \qquad
\eta\eta_{s_1} = \eta_{s_1}\eta^{-1}.
\label{etarels}
\end{equation}
\begin{equation}
\eta_\pi Y = Y^{-1}q^{-\frac12} \eta_\pi, \qquad \eta_{s_1} Y = Y^{-1}\eta_{s_1},
\qquad \eta Y =  Y q^{\frac12} \eta,
\label{etapastY}
\end{equation}
\begin{equation}
\mathbf{1}_0 = (1+\eta_{s_1})t^{-\frac12}\frac{(1-tY^2)}{(1-Y^2)},
\qquad
\eta_{s_1}\mathbf{1}_0 = \mathbf{1}_0,
\qquad
\mathbf{1}_0 \eta_{s_1} =\mathbf{1}_0 t^{-1} \frac{(1-tY^{-2})}{(1-t^{-1}Y^{-2})}.
\label{10ineta}
\end{equation}
\end{prop}
\begin{proof}
From \eqref{tausq}, $\eta_\pi^2 = (\tau^\vee_\pi)^2 = 1$.  Using
$\tau^\vee_1 Y = Y^{-1}\tau^\vee_1$ and the formula for $(\tau^\vee_1)^2$ in \eqref{tausq} gives
\begin{align*}
\eta_{s_1}^2 
&= t^{\frac12}\frac{(1-Y^{-2})}{(1-tY^{-2})}\tau^\vee_1 t^{\frac12}\frac{(1-Y^{-2})}{(1-tY^{-2})}\tau^\vee_1
= t \frac{(1-Y^{-2})}{(1-tY^{-2})}\frac{(1-Y^{2})}{(1-tY^{2})} \tau^\vee_1 \tau^\vee_1 \\
&= t \frac{(1-Y^{-2})}{(1-tY^{-2})}\frac{(1-Y^{2})}{(1-tY^{2})}
\cdot t^{-1}\frac{(1-tY^2)(1-tY^{-2})}{(1-Y^2)(1-Y^{-2})}
=1.
\end{align*}
Then
$$\eta \eta_{s_1} = \eta_\pi \eta_{s_1}\eta_{s_1} = \eta_\pi = \eta_{s_1}\eta_{s_1}\eta_\pi = \eta_{s_1}\eta^{-1}.
$$
The relations $\eta_\pi Y = Y^{-1}q^{-\frac12} \eta_\pi$ and $\eta_{s_1}Y = Y^{-1}\eta_{s_1}$
follow from \eqref{taupastY} and
$$\eta Y = \eta_\pi \eta_{s_1} Y = \eta_\pi Y^{-1} \eta_{s_1} = Y q^{\frac12} \eta_\pi\eta_{s_1}
= Y q^{\frac12} \eta.$$
Using \eqref{10intau}, 
$$
\mathbf{1}_0 =\tau^\vee_1+t^{-\frac12}\frac{(1-tY^2)}{(1-Y^2)}
=\Big(t^{\frac12}\tau^\vee_1\frac{(1-Y^2)}{(1-tY^2)} +1 \Big)\cdot t^{-\frac12}\frac{(1-tY^2)}{(1-Y^2)}
=(\eta_{s_1}+1)t^{-\frac12}\frac{(1-tY^2)}{(1-Y^2)}.
$$
By the last identity in \eqref{10rels},
\begin{align*}
\eta_{s_1}\mathbf{1}_0 = t^{\frac12}\frac{(1-Y^{-2})}{(1-tY^{-2})} \tau^\vee_1\mathbf{1}_0 = 
t^{\frac12}\frac{(1-Y^{-2})}{(1-tY^{-2})}\cdot t^{-\frac12}\frac{(1-tY^{-2})}{(1-Y^{-2})}\mathbf{1}_0
=\mathbf{1}_0
\end{align*}
and, by the second identity in \eqref{10rels},
\begin{align*}
\mathbf{1}_0 \eta_{s_1} 
= \mathbf{1}_0 t^{\frac12}\frac{(1-Y^{-2})}{(1-tY^{-2})} \tau^\vee_1 
= \mathbf{1}_0 \tau^\vee_1 t^{\frac12}\frac{(1-Y^{2})}{(1-tY^{2})} 
= \mathbf{1}_0 t^{-\frac12}\frac{(1-tY^{-2})}{(1-Y^{-2})} t^{\frac12}\frac{(1-Y^{2})}{(1-tY^{2})} 
=\mathbf{1}_0 t^{-1} \frac{(1-tY^{-2})}{(1-t^{-1}Y^{-2})}.
\end{align*}
\end{proof}


\subsection{The polynomial representation $\tilde H_{\mathrm{int}}\mathbf{1}_Y$} \label{sec:polyrep}

Let $\tilde H_{\mathrm{int}} \mathbf{1}_Y$ be the $\tilde H_{\mathrm{int}}$ 
module generated by a single generator $\mathbf{1}_Y$ with
relations
$$T_1\mathbf{1}_Y = t^{\frac12}\mathbf{1}_Y
\qquad\hbox{and}\qquad
T_\pi \mathbf{1}_Y = \mathbf{1}_Y.$$
Then $Y\mathbf{1}_Y = T_\pi T_1\mathbf{1}_Y = t^{\frac12}\mathbf{1}_Y$ and
$$\tilde H_{\mathrm{int}}\mathbf{1}_Y \qquad\hbox{has $\CC$-basis}\qquad \{ X^k\mathbf{1}_Y\ |\ k\in \ZZ\}.$$
Using the second relation in \eqref{Hdefn} and \eqref{TpastX}, the action of $\tilde H_{\mathrm{int}}$ 
in the basis $\{X^k\mathbf{1}_Y\ |\ k\in \ZZ\}$ is given explicitly by
$$XX^r\mathbf{1}_Y = X^{r+1}\mathbf{1}_Y,
\qquad\qquad
T_\pi X^r\mathbf{1}_Y = q^{\frac{r}{2}}X^{-r}\mathbf{1}_Y
\qquad\hbox{and}
$$
$$
T_1X^r\mathbf{1}_Y = t^{\frac12}X^{-r}\mathbf{1}_Y + (t^{\frac12}-t^{-\frac12})\frac{X^r-X^{-r}}{1-X^2}\mathbf{1}_Y,
\qquad\hbox{for $r\in \ZZ$.}
$$
The electronic Macdonald polynomials are the elements $E_m(X)\in \CC[X,X^{-1}]$, $m\in \ZZ$, determined by
\begin{align}
YE_m(X)\mathbf{1}_Y &= t^{-\frac12} q^{\frac{-m}{2}}E_m(X)\mathbf{1}_Y,\quad\hbox{if $m\in \ZZ_{>0}$, and} 
\nonumber \\
YE_{-m}(X)\mathbf{1}_Y &= t^{\frac12} q^{\frac{m}{2}} E_{-m}(X)\mathbf{1}_Y,\quad\hbox{if $m\in \ZZ_{\ge0}$,}
\label{Eeig}
\end{align}
with normalization such that the coefficient of $X^m$ in $E_m(X)$ is 1.
The electronic Macdonald polynomials are given recursively (see \cite[(6.2.3)]{Mac03}) by
$E_0(X) = 1$ and $E_1(X) = X$ and
\begin{align}
\tau^\vee_1 E_r(X)\mathbf{1}_Y &= t^{-\frac12} E_{-r}(X)\mathbf{1}_Y,
&\tau^\vee_1 E_{-r}(X)\mathbf{1}_Y &= t^{-\frac12}\frac{(1-tY^2)(1-tY^{-2})}{(1-Y^2)(1-Y^{-2})} E_r(X)\mathbf{1}_Y,
\nonumber \\
\tau_\pi^\vee E_r(X)\mathbf{1}_Y &= t^{-\frac12}E_{-(r-1)}(X)\mathbf{1}_Y.
&\tau_\pi^\vee E_{-r}(X)\mathbf{1}_Y &= t^{\frac12} E_{r+1}(X)\mathbf{1}_Y,
\label{tauonE}
\end{align}
for $r\in \ZZ_{>0}$.
Note that $\tau^\vee_\pi E_0(X)\mathbf{1}_Y = XT_1 \mathbf{1}_Y = t^{\frac12} X\mathbf{1}_Y
=t^{\frac12} E_1(X)\mathbf{1}_Y$ and
$$\tau^\vee_\pi E_1(X)\mathbf{1}_Y = XT_1X\mathbf{1}_Y = t^{-\frac12}XT_1XT_1\mathbf{1}_Y
=t^{-\frac12}XX^{-1}\mathbf{1}_Y = t^{-\frac12}E_0(X)\mathbf{1}_Y$$
and
\begin{equation}
\tau^\vee_1E_0(X)\mathbf{1}_Y
=\Big(T_1+t^{-\frac12}\frac{(1-t)}{(1-Y^{-2})}\Big)\mathbf{1}_Y
=\Big(t^{\frac12}+t^{-\frac12}\frac{(1-t)}{(1-t^{-1})}\Big)\mathbf{1}_Y
=(t^{\frac12}-t^{\frac12})\mathbf{1}_Y = 0.
\label{tauonE0}
\end{equation}
As pictured below, the elements $\tau^\vee_1$ and $\tau^\vee_\pi$ can be used to recursively construct the electronic Macdonald polynomials.
$$
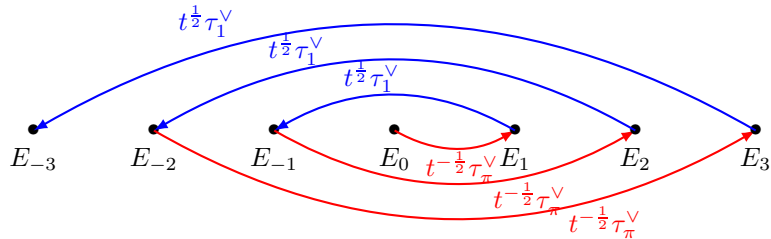
\begin{tikzpicture}[scale=0.8]
\tikzstyle{every node}=[font=\small]
\node at (-6, 0) {$\bullet$};  
    \filldraw (-6, 0) node[anchor=north,yshift=-0.1cm] {$E_{-3}$};
\node at (-4, 0) {$\bullet$};  
    \filldraw (-4, 0) node[anchor=north,yshift=-0.1cm] {$E_{-2}$};
	\draw[-latex,  bend right=30, thick, red] (-4, 0) to node[below,xshift=2cm,yshift=0.3cm]{$t^{-\frac12}\tau^\vee_\pi$} (6, 0);
\node at (-2, 0) {$\bullet$};  
    \filldraw (-2, 0) node[anchor=north,yshift=-0.1cm] {$E_{-1}$};
	\draw[-latex,  bend right=30, thick, red] (-2, 0) to node[below,xshift=1cm,yshift=0.15cm]{$t^{-\frac12}\tau^\vee_\pi$} (4, 0);
\node at (0, 0) {$\bullet$};  
    \filldraw (0, 0) node[anchor=north,yshift=-0.1cm] {$E_0$};
	\draw[-latex,  bend right=30, thick, red] (0, 0) to node[below,xshift=0.1cm,yshift=0.1cm]{$t^{-\frac12}\tau^\vee_\pi$} (2, 0);
\node at (2, 0) {$\bullet$};  
    \filldraw (2, 0) node[anchor=north,yshift=-0.1cm] {$E_1$};
	\draw[-latex,  bend right=30, thick, blue] (2, 0) to node[above,xshift=-0.3cm,yshift=-0.1cm]{$t^{\frac12}\tau^\vee_1$} (-2, 0);
\node at (4, 0) {$\bullet$};  
    \filldraw (4, 0) node[anchor=north,yshift=-0.1cm] {$E_2$};
	\draw[-latex,  bend right=30, thick, blue] (4, 0) to node[above,xshift=-1.3cm,yshift=-0.2cm]{$t^{\frac12}\tau^\vee_1$} (-4, 0);
\node at (6, 0) {$\bullet$};  
    \filldraw (6, 0) node[anchor=north,yshift=-0.1cm] {$E_3$};
	\draw[-latex,  bend right=30, thick, blue] (6, 0) to node[above,xshift=-2.5cm,yshift=-0.3cm]{$t^{\frac12}\tau^\vee_1$}  (-6,0);
\end{tikzpicture}
$$
The bosonic Macdonald polynomials $P_m(X)\in \CC[X,X^{-1}]$, for $m\in \ZZ_{\ge 0}$,
can be given \cite[(6.3.10)]{Mac03} by
\begin{equation}
P_m(X)\mathbf{1}_Y = E_{-m}(X)\mathbf{1}_Y + \frac{t(1-q^m)}{(1-tq^m)} E_m(X)\mathbf{1}_Y.
\label{Pdefn}
\end{equation}
Applying \eqref{Eeig} to \eqref{Pdefn} and using 
$\tau^\vee_1 E_r(X)\mathbf{1}_Y = t^{-\frac12} E_{-r}(X)\mathbf{1}_Y$
gives, for $m\in \ZZ_{>0}$,
\begin{equation}
P_m(X)\mathbf{1}_Y 
=t^{\frac12}\tau^\vee_1 E_m(X)\mathbf{1}_Y + t\frac{(1-t^{-1}Y^{-2})}{(1-Y^{-2})} E_m(X)\mathbf{1}_Y 
= t^{\frac12} \mathbf{1}_0 E_m(X)\mathbf{1}_Y,
\label{creationP}
\end{equation}
where the last equality follows from \eqref{10intau}.

The following Proposition analyzes the action of $\eta$ and $\eta^{-1}$ as operators on the polynomial
representation.  It shows that $\eta$ acts as a raising operator with $\eta E_0(X)\mathbf{1}_Y = 0$
and that $\eta^{-1}$ acts as a lowering operator with $\eta^{-1}E_1(X)\mathbf{1}_Y = 0$.
The operator $\eta$ is pictured in blue and the operator $\eta^{-1}$ is pictured in red.
The coefficients below and above the arrows provide the constants which appear in the formulas
for $\eta E_m$, $\eta_{E_m}$, $\eta^{-1}E_m$ and $\eta^{-1}E_m$ which are derived in 
\eqref{etaEm}, \eqref{etaEmm}, \eqref{etainvEm} and \eqref{etainvEmm}.
\begin{equation}
\begin{matrix}
\begin{tikzpicture}[scale=0.8]
\tikzstyle{every node}=[font=\small]
	\draw[-latex,  bend left=30, thick, blue] (-6.8, 0.25) to node[above]{}  (-6, 0);
	\draw[-latex,  bend left=30, thick, red] (-6, 0) to node[below]{} (-6.8,-0.25);
\node at (-6, 0) {$\bullet$};  
    \filldraw (-6, 0) node[anchor=north,yshift=-0.1cm] {$E_{-3}$};
	\draw[-latex,  bend left=30, thick, blue] (-6, 0) to node[above]{$\frac{t^{\frac12}(1-q^3)}{1-tq^3}$} (-4, 0);
	\draw[-latex,  bend left=30, thick, red] (-4, 0) to node[below,yshift=-0.2cm]{$\frac{t^{-\frac12}(1-tq^3)}{1-q^3}$} (-6, 0);
\node at (-4, 0) {$\bullet$};  
    \filldraw (-4, 0) node[anchor=north,yshift=-0.1cm] {$E_{-2}$};
	\draw[-latex,  bend left=30, thick, blue] (-4, 0) to node[above]{$\frac{t^{\frac12}(1-q^2)}{1-tq^2}$} (-2, 0);
	\draw[-latex,  bend left=30, thick, red] (-2, 0) to node[below,yshift=-0.2cm]{$\frac{t^{-\frac12}(1-tq^2)}{1-q^2}$} (-4, 0);
\node at (-2, 0) {$\bullet$};  
    \filldraw (-2, 0) node[anchor=north,yshift=-0.1cm] {$E_{-1}$};
	\draw[-latex,  bend left=30, thick, blue] (-2, 0) to node[above]{$\frac{t^{\frac12}(1-q)}{1-tq}$} (0, 0);
	\draw[-latex,  bend left=30, thick, red] (0, 0) to node[below,yshift=-0.2cm]{$\frac{t^{-\frac12}(1-tq)}{1-q}$} (-2, 0);
\node at (0, 0) {$\bullet$};  
    \filldraw (0, 0) node[anchor=north,yshift=-0.1cm] {$E_0$};
	\draw[-latex,  bend right=30, thick, blue] (0, 0) to node[above,xshift=0.1cm,yshift=0.3cm]{$0$} (0.5, 0.8);
	\draw[-latex,  bend right=30, thick, red] (2, 0) to node[below,xshift=-0.1cm,yshift=-0.3cm]{$0$} (1.5, -0.8);
\node at (2, 0) {$\bullet$};  
    \filldraw (2, 0) node[anchor=north,yshift=-0.1cm] {$E_1$};
	\draw[-latex,  bend left=30, thick, blue] (2, 0) to node[above]{$\frac{t^{-\frac12}(1-tq)}{1-q}$} (4, 0);
	\draw[-latex,  bend left=30, thick, red] (4, 0) to node[below,yshift=-0.2cm]{$\frac{t^{\frac12}(1-q)}{1-tq}$} (2, 0);
\node at (4, 0) {$\bullet$};  
    \filldraw (4, 0) node[anchor=north,yshift=-0.1cm] {$E_2$};
	\draw[-latex,  bend left=30, thick, blue] (4, 0) to node[above]{$\frac{t^{-\frac12}(1-tq^2)}{1-q^2}$} (6, 0);
	\draw[-latex,  bend left=30, thick, red] (6, 0) to node[below,yshift=-0.2cm]{$\frac{t^{\frac12}(1-q^2)}{1-tq^2}$} (4, 0);
\node at (6, 0) {$\bullet$};  
    \filldraw (6, 0) node[anchor=north,yshift=-0.1cm] {$E_3$};
	\draw[-latex,  bend left=30, thick, blue] (6, 0) to node[above]{}  (6.8, 0.25);
	\draw[-latex,  bend left=30, thick, red] (6.8, -0.25) to node[below]{} (6, 0);
\end{tikzpicture}
\end{matrix}
\label{etaonpolypic}
\end{equation}
It is important to note that $\eta$ and $\eta^{-1}$ are not invertible as operators on the
polynomial representation
(even though they are inverses of each other as elements of $\tilde H$).  This phenomenon is of the same
nature as the fact that $(1-t^{-1}Y)$ is a well defined element of $\tilde H$ with inverse
$\frac{1}{(1-t^{-1}Y)}$ in $\tilde H$, and $(1-t^{-1}Y)$ is a well defined operator on
$\CC[X,X^{-1}]$ that is not invertible as an operator on the polynomial representation $\CC[X,X^{-1}]$. 

The identities
\begin{align}
\eta^{-(\ell-j)}\eta^j  E_m(X)\mathbf{1}_Y
&=t^{\frac12(\ell-2j)}\cdot \ev_m\Big( 
\frac{(t^{-1}Y^{-2}q^{-(\ell-2j)};q)_{\ell-j} }{ (Y^{-2}q^{-(\ell-2j)};q)_{\ell-j} }\cdot
\frac{(Y^{-2};q)_j }{(t^{-1}Y^{-2};q)_j }\Big) E_{m-(\ell-2j)}(X)\mathbf{1}_Y, 
\label{normpB}
\\
\eta^j \eta^{-(\ell-j)} E_{-m}(X)\mathbf{1}_Y
&= t^{-\frac12(\ell-2j)} \cdot
\ev_m\Big(
\frac{(t^{-1}Y^{-2}q^{\ell-2j+1};q)_j }{(Y^{-2}q^{\ell-2j+1};q)_j }
\frac{(Y^{-2}q;q)_{\ell-j} }{ (t^{-1}Y^{-2}q;q)_{\ell-j} } \Big) E_{-m-(\ell-2j)}(X)\mathbf{1}_Y.
\label{normnB}
\end{align}
follow from \eqref{normn} and \eqref{normp} of the following Proposition by replacing $j$ with $\ell-j$ (we keep the same
conditions on $j$ and $\ell$ as in Proposition \ref{etaonEnew}).  They will be used in the 
proof of Theorem \ref{finalthm}.

\begin{prop} \label{etaonEnew}
As in \eqref{etadefn} and \eqref{etainvdefn}, let
$$\eta 
= t^{\frac12} \frac{(1-Y^{2}q)}{(1-tY^{2}q)}\tau^\vee_\pi \tau^\vee_1
\qquad\hbox{and}\qquad
\eta^{-1} = t^{\frac12} \frac{(1-Y^{-2})}{(1-tY^{-2})} \tau^\vee_1\tau^\vee_\pi.$$
Let $\ev_m\colon \CC[Y,Y^{-1}]\to \CC$ be the homomorphism given by $\ev_m(Y) = t^{-\frac12} q^{-\frac12 m}$
and extend $\ev_m$ to elements of $\CC(Y)$ such that the denominator does not evaluate to $0$. 

\smallskip\noindent
If $\ell\in \ZZ_{\ge 0}$ and $m\in \ZZ_{>0}$ and $j\in \{0,\ldots, \ell\}$ then
\begin{align}
\eta^{-j}\eta^{\ell-j} E_m(X)\mathbf{1}_Y
&=t^{-\frac12(\ell-2j)}\cdot \ev_m\Big( 
\frac{(t^{-1}Y^{-2}q^{\ell-2j};q)_j}{(Y^{-2}q^{\ell-2j};q)_j}\cdot
\frac{(Y^{-2};q)_{\ell-j} }{(t^{-1}Y^{-2};q)_{\ell-j} }\Big) E_{m+\ell-2j}(X)\mathbf{1}_Y, 
\label{normp}
\\
\eta^{\ell-j}\eta^{-j} E_{-m}(X)\mathbf{1}_Y
&= t^{\frac12(\ell-2j)} \cdot
\ev_m\Big(
\frac{(t^{-1}Y^{-2}q^{-(\ell-2j)+1};q)_{\ell-j} }{(Y^{-2}q^{-(\ell-2j)+1};q)_{\ell-j} }
\frac{(Y^{-2}q;q)_j }{(t^{-1}Y^{-2}q;q)_j} \Big) E_{-m+\ell-2j}(X)\mathbf{1}_Y.
\label{normn}
\end{align}
\end{prop}
%
%
\begin{proof}
Assume $m\in \ZZ_{>0}$.
By \eqref{tauonE} and \eqref{Eeig},
\begin{align}
\eta E_m(X)\mathbf{1}_Y
&=t^{\frac12} \frac{(1-Y^{2}q)}{(1-tY^{2}q)}\tau^\vee_\pi \tau^\vee_1 E_m(X)\mathbf{1}_Y
=t^{\frac12}\frac{(1-Y^{2}q)}{(1-tY^{2}q)}\tau^\vee_\pi t^{-\frac12}E_{-m}(X)\mathbf{1}_Y
\nonumber \\
&= t^{\frac12} t^{-\frac12}\frac{(1-Y^{2}q)}{(1-tY^{2}q)}t^{\frac12}E_{m+1}(X)\mathbf{1}_Y
= t^{\frac12} \frac{(1-t^{-1}q^{-(m+1)}q)}{(1-tt^{-1}q^{-(m+1)} q)} E_{m+1}(X)\mathbf{1}_Y
\nonumber \\
&= t^{\frac12} \frac{(1-t^{-1}q^{-m})}{(1-q^{-m})} E_{m+1}(X)\mathbf{1}_Y
= t^{-\frac12} \frac{(1-tq^m)}{(1-q^m)}E_{m+1}(X)\mathbf{1}_Y.
\label{etaEm}
\end{align}
Thus, for $\ell\in \ZZ_{\ge 0}$ and $m\in \ZZ_{>0}$,
\begin{align*}
\eta^\ell E_m(X)\mathbf{1}_Y 
&= t^{-\frac12\ell} \frac{(1-tq^m)(1-tq^{m+1})\cdots(1-tq^{m+\ell-1}) }{ (1-q^m)(1-q^{m+1})\cdots (1-q^{m+\ell-1}) }  E_{m+\ell}(X)\mathbf{1}_Y
\\
&= t^{-\frac12\ell} \ev_m\Big(
\frac{(Y^{-2};q)_{\ell}}{(t^{-1}Y^{-2};q)_{\ell}}\Big)  E_{m+\ell}(X)\mathbf{1}_Y.
\end{align*}

\smallskip\noindent
Assume $m\in \ZZ_{>0}$.  Using \eqref{tauonE}, \eqref{taupastY} and \eqref{Eeig},
\begin{align}
\eta E_{-m}(X)&\mathbf{1}_Y
=t^{\frac12} \frac{(1-Y^{2}q)}{(1-tY^{2}q)}\tau^\vee_\pi \tau^\vee_1 E_{-m}(X)\mathbf{1}_Y
=t^{\frac12} \frac{(1-Y^{2}q)}{(1-tY^{2}q)}\tau^\vee_\pi 
t^{-\frac12} \frac{(1-tY^2)(1-tY^{-2}) }{ (1-Y^2)(1-Y^{-2})}E_m(X)\mathbf{1}_Y
\nonumber \\
&= \frac{(1-q^{-m})(1-t^2q^m) }{ (1-t^{-1}q^{-m})(1-tq^m)}
\frac{(1-Y^{2}q)}{(1-tY^{2}q)}\tau^\vee_\pi E_m(X)\mathbf{1}_Y
\nonumber \\
&= \frac{(1-q^{-m})(1-t^2q^m) }{ (1-t^{-1}q^{-m})(1-tq^m)}
\frac{(1-Y^{2}q)}{(1-tY^{2}q)}t^{-\frac12}E_{-(m-1)}\mathbf{1}_Y
\nonumber \\
&= \frac{(1-q^{-m})(1-t^2q^m) }{ (1-t^{-1}q^{-m})(1-tq^m)}
\frac{(1-tq^m)}{(1-t^2q^m)}t^{-\frac12}E_{-(m-1)}\mathbf{1}_Y
=t^{\frac12}\frac{(1-q^m)}{(1-tq^m)}E_{-m+1}(X)\mathbf{1}_Y.
\label{etaEmm}
\end{align}
By \eqref{tauonE0}, 
$$\eta E_0(X)\mathbf{1}_Y = t^{\frac12}\frac{(1-Y^2q)}{(1-tY^2q)}\tau^\vee_\pi \tau^\vee_1 E_0(X)\mathbf{1}_Y=0
=t^{\frac12}\frac{(1-q^0)}{(1-tq^0)}E_{-0+1}(X)\mathbf{1}_Y.$$
Thus, for $\ell\in \ZZ_{\ge 0}$ and $m\in \ZZ_{\ge 0}$,
\begin{align*}
\eta^\ell E_{-m}(X)\mathbf{1}_Y 
&= t^{\frac12\ell }\frac{(1-q^{m-\ell+1..m})}{(1-tq^{m-\ell+1..m})} E_{-m+\ell}(X)\mathbf{1}_Y
\\
&= t^{\frac12\ell }\ev_m\Big(
\frac{(1-q^{m-\ell+1})\cdots (1-q^{m-1})(1-q^m)) }{( 1-tq^{m-\ell+1})\cdots (1-tq^{m-1})(1-tq^m)}\Big) E_{-m+\ell}(X)\mathbf{1}_Y
\\
&= t^{\frac12\ell }\ev_m\Big(\frac{(t^{-1}Y^{-2}q^{-(\ell-1)};q)_\ell}{(Y^{-2}q^{-(\ell-1)};q)_\ell}\Big)E_{-m+\ell}(X)\mathbf{1}_Y,
\end{align*}
where the right hand side evaluates to $0$ if $\ell>m$ 
(because the denominator factors are all nonzero and the numerator contains a factor of $(1-q^0)=1-1=0$).

\smallskip
Assume $m\in \ZZ_{>0}$.
By \eqref{tauonE} and \eqref{Eeig},
\begin{align}
\eta^{-1} E_m(X)\mathbf{1}_Y
&=t^{\frac12} \frac{(1-Y^{-2})}{(1-tY^{-2})} \tau^\vee_1 \tau^\vee_\pi E_m(X)\mathbf{1}_Y
=t^{\frac12} \frac{(1-Y^{-2})}{(1-tY^{-2})} \tau^\vee_1 t^{-\frac12} E_{-(m-1)}(X)\mathbf{1}_Y
\nonumber \\
&= \frac{(1-Y^{-2})}{(1-tY^{-2})} 
t^{-\frac12} \frac{(1-tY^2)(1-tY^{-2})}{(1-Y^2)(1-Y^{-2})} E_{m-1}(X)\mathbf{1}_Y
\nonumber \\
&= \frac{(1-tq^{m-1})}{(1- t^2 q^{m-1})} 
t^{-\frac12} \frac{(1-q^{-(m-1)})(1-t^2 q^{m-1})}{(1-t^{-1}q^{-(m-1)})(1-t q^{m-1})} E_{m-1}(X)\mathbf{1}_Y
\nonumber \\
&=t^{\frac12} \frac{(1-q^{m-1} ) }{(1-tq^{m-1})} E_{m-1}(X)\mathbf{1}_Y.
\label{etainvEm}
\end{align}
In particular,
$\eta^{-1} E_1(X)\mathbf{1}_Y = 0$.  Thus, for $\ell\in \ZZ_{\ge 0}$ and $m\in \ZZ_{>0}$,
\begin{align*}
\eta^{-\ell} E_m(X)\mathbf{1}_Y 
&= t^{\frac12\ell} \frac{(1-q^{m-\ell})\cdots (1-q^{m-2})(1-q^{m-1}) }{ (1-t q^{m-\ell})\cdots (1-tq^{m-2})(1-tq^{m-1})} E_{m-\ell}(X)\mathbf{1}_Y
\\
&=t^{\frac12} \ev_m\Big(\frac{(t^{-1}Y^{-2}q^{-\ell};q)_\ell}{(Y^{-2}q^{-\ell};q)_\ell}\Big) E_{m-\ell}(X)\mathbf{1}_Y,
\end{align*}
where the right hand side evaluates to 0 if $\ell\ge m$.

Assume $m\in \ZZ_{\ge0}$.
By \eqref{tauonE} and \eqref{Eeig},
\begin{align}
\eta^{-1} E_{-m}(X)\mathbf{1}_Y
&=t^{\frac12} \frac{(1-Y^{-2})}{(1-tY^{-2})} \tau^\vee_1 \tau^\vee_\pi E_{-m}(X)\mathbf{1}_Y
=t^{\frac12} \frac{(1-Y^{-2})}{(1-tY^{-2})} \tau^\vee_1 t^{\frac12} E_{m+1}(X)\mathbf{1}_Y
\nonumber \\
&=t \frac{(1-Y^{-2})}{(1-tY^{-2})} t^{-\frac12}  E_{-(m+1)}(X)\mathbf{1}_Y
= t^{\frac12} \frac{(1-t^{-1}q^{-(m+1)})}{(1-tt^{-1}q^{-(m+1)})}  E_{-(m+1)}(X)\mathbf{1}_Y
\nonumber \\
&= t^{\frac12} \frac{(1-t^{-1}q^{-m-1}) }{(1-q^{-m-1} )}  E_{-m-1}(X)\mathbf{1}_Y
=t^{-\frac12} \frac{(1-tq^{m+1})}{ (1-q^{m+1})} E_{-m-1}(X)\mathbf{1}_Y.
\label{etainvEmm}
\end{align}
Thus, for $\ell\in \ZZ_{\ge 0}$ and $m\in \ZZ_{\ge 0}$,
\begin{align*}
\eta^{-\ell} E_{-m}(X)\mathbf{1}_Y 
&= t^{-\frac12\ell} \cdot \frac{(1-t q^{m+1})(1-q^{m+2})\cdots (1-q^{m+\ell}) }{ (1-q^{m+1})(1-q^{m+2})\cdots (1-q^{m+\ell}) } 
E_{-m-\ell}(X)\mathbf{1}_Y
\\
&=t^{-\frac12\ell}\ev_m\Big( \frac{(Y^{-2}q;q)_\ell}{(t^{-1}Y^{-2}q;q)_\ell}\Big)
E_{-m-\ell}(X)\mathbf{1}_Y.
\end{align*}
\end{proof}


\subsection{Some identities in $\tilde H$}

\begin{prop} Let $\ell\in \ZZ_{>0}$.
As elements of $\tilde H_{\mathrm{int}}$, 
\begin{equation}
E_{-\ell}(X)\mathbf{1}_0 
= t^{\frac12}\Big(T^{-1}_1+t^{-\frac12}\frac{(1-t)q^\ell t}{(1-q^\ell t)}\Big)E_\ell(X)\mathbf{1}_0,
\qquad\quad
E_{\ell+1}(X)\mathbf{1}_0 = t^{-\frac12}\tau_\pi^\vee E_{-\ell}(X)\mathbf{1}_0,
\label{E10}
\end{equation}
and
\begin{equation}
P_\ell(X)\mathbf{1}_0 = t^{\frac12}\mathbf{1_0}E_\ell(X)\mathbf{1}_0.
\label{smashE}
\end{equation}
Additionally, $E_1(X)\mathbf{1}_0 = t^{-\frac12}\tau_\pi^\vee E_0(X)\mathbf{1}_0$.
\end{prop}
\begin{proof} 
(a) Let $Q_1(X), Q_2(X)\in \CC[X,X^{-1}]$ be such that
$$t^{\frac12}\Big(T^{-1}_1+t^{-\frac12}\frac{(1-t)q^\ell t}{(1-q^\ell t)}\Big)E_\ell(X) = Q_1(X)T_1+Q_2(X).$$
Then, by \eqref{tauonE},
\begin{align*}
E_{-\ell}(X)\mathbf{1}_Y
&=t^{\frac12} \tau_1^\vee E_\ell(X)\mathbf{1}_Y
=t^{\frac12}\Big(T^{-1}_1+t^{-\frac12}\frac{(1-t)Y^{-2}}{(1-Y^{-2})}\Big)E_\ell(X)\mathbf{1}_Y
\\
&=t^{\frac12}\Big(T^{-1}_1+t^{-\frac12}\frac{(1-t)q^\ell t}{(1-q^\ell t)}\Big)E_\ell(X)\mathbf{1}_Y
=(Q_1(X)t^{\frac12}+Q_2(X))\mathbf{1}_Y.
\end{align*}
Since $\{X^k\mathbf{1}_Y\ |\ k\in \ZZ\}$ is a basis of $\tilde H\mathbf{1}_Y$ then
$E_{-\ell}(X) = Q_1(X)t^{\frac12}+Q_2(X).$
So
\begin{align*}
E_{-\ell}(X)\mathbf{1}_0 &= (Q_1(X)t^{\frac12}+Q_2(X))\mathbf{1}_0 \\
&=(Q_1(X)T_1+Q_2(X))\mathbf{1}_0
=t^{\frac12}\Big(T^{-1}_1+t^{-\frac12}\frac{(1-t)q^\ell t}{(1-q^\ell t)}\Big)E_\ell(X)\mathbf{1}_0.
\end{align*}
(b) Let $Q_1(X), Q_2(X)\in \CC[X,X^{-1}]$ 
such that $t^{-\frac12}\tau_\pi^\vee E_{-\ell}(X) = XT_1E_{-\ell}(X)=Q_1(X)T_1+Q_2(X).$
Then, by \eqref{tauonE},
$$E_{\ell+1}(X)\mathbf{1}_Y = t^{-\frac12} \tau_\pi^\vee E_{-\ell}(X)\mathbf{1}_Y
=(Q_1(X)T_1+Q_2(X))\mathbf{1}_Y
=(Q_1(X)t^{\frac12}+Q_2(X))\mathbf{1}_Y.
$$
Since $\{X^k\mathbf{1}_Y\ |\ k\in \ZZ\}$ is a basis of $\tilde H\mathbf{1}_Y$ then
$E_{\ell+1}(X) =(Q_1(X)t^{\frac12}+Q_2(X))$.  So
$$E_{\ell+1}(X)\mathbf{1}_0 = (Q_1(X)t^{\frac12}+Q_2(X))\mathbf{1}_0
=(Q_1(X)T_1+Q_2(X))\mathbf{1}_0 = t^{-\frac12}\tau_\pi^\vee E_{-\ell}(X)\mathbf{1}_0.$$
(c)  Let $Q_1(X), Q_2(X)$ be such that $t^{\frac12}\mathbf{1}_0 E_\ell(X) = Q_1(X)T_1+Q_2(X)$.
Then
$$P_\ell(X)\mathbf{1}_Y 
= t^{\frac12}\mathbf{1}_0E_\ell(X)\mathbf{1}_Y = (Q_1(X)T_1+Q_2(X))\mathbf{1}_Y
=(Q_1(X)t^{\frac12}+Q_2(X))\mathbf{1}_Y.
$$
So $P_\ell(X) = (Q_1(X)t^{\frac12}+Q_2(X))$ and
$$P_\ell(X)\mathbf{1}_0 = (Q_1(X)t^{\frac12}+Q_2(X))\mathbf{1}_0 = 
(Q_1(X)T_1+Q_2(X))\mathbf{1}_0 = t^{\frac12}\mathbf{1}_0 E_\ell(X)\mathbf{1}_0.
$$
\end{proof}

\begin{prop}
Let
$$c(Y) =  \frac{(1-tY^2)}{(1-Y^2)}\qquad\hbox{and}\qquad
F_\ell(Y) = \frac{(1-t)}{(1-t q^\ell)}\frac{(1-t Y^2 q^\ell)}{(1-Y^2)},\quad\hbox{for $\ell\in \ZZ_{>0}$.}
$$
Then, as elements of $\tilde H$,
\begin{equation}
E_{-\ell}(X)\mathbf{1}_0 = (\eta_{s_1}c(Y)+F_\ell(Y))E_\ell(X) \mathbf{1}_0
\quad\hbox{and}\quad
E_{\ell+1}(X)\mathbf{1}_0 = t^{-\frac12}\eta_\pi E_{-\ell}(X)\mathbf{1}_0.
\label{etaonE10}
\end{equation}
\end{prop}
\begin{proof}  Using the first identity in \eqref{E10}, \eqref{tau1Tinv} and \eqref{nintdefn},
\begin{align*}
E_{-\ell}(X)\mathbf{1}_0 
&= t^{\frac12} \Big(T^{-1}_1+t^{-\frac12}\frac{(1-t)q^\ell t}{(1-q^\ell t)}\Big)E_\ell(X)\mathbf{1}_0
\\
&= t^{\frac12} 
\Big(\tau^\vee_1 - t^{-\frac12}\frac{(1-t)Y^{-2}}{(1-Y^{-2})}+t^{-\frac12}\frac{(1-t)q^\ell t}{(1-q^\ell t)}\Big)E_\ell(X)\mathbf{1}_0
\\
&= t^{\frac12} \Big(t^{-\frac12}\frac{(1-tY^{-2})}{(1-Y^{-2})}\eta_{s_1} 
+ t^{-\frac12}\frac{(1-t)(-Y^{-2}+q^\ell t Y^{-2}+q^\ell t - q^\ell t Y^{-2})} {(1-Y^{-2})(1-q^\ell t)}\Big)E_\ell(X)\mathbf{1}_0
\\
&= t^{\frac12} \Big(t^{-\frac12}\eta_{s_1}\frac{(1-tY^{2})}{(1-Y^{2})} 
+ t^{-\frac12}\frac{(1-t)(-1+q^\ell t Y^2)} {(Y^2-1)(1-q^\ell t)}\Big)E_\ell(X)\mathbf{1}_0
\\
&= \Big(\eta_{s_1}\frac{(1-tY^{2})}{(1-Y^{2})} 
+ \frac{(1-t)(1-t Y^2 q^\ell )} {(1-t q^\ell)(1-Y^2)}\Big)E_\ell(X)\mathbf{1}_0,
\end{align*}
where the next to last equality follows from the second identity in \eqref{etapastY}.
Then, by \eqref{nintdefn}, the second identity in \eqref{E10} and \eqref{etadefn},
$$E_{\ell+1}(X)\mathbf{1}_0 = t^{-\frac12}\eta_\pi E_{-\ell}(X)\mathbf{1}_0
= t^{-\frac12}\Big(\eta \frac{(1-tY^{2})}{(1-Y^{2})} 
+ \eta_\pi \frac{(1-t)(1-t Y^2 q^\ell )} {(1-t q^\ell)(1-Y^2)}\Big)E_\ell(X)\mathbf{1}_0.
$$
\end{proof}

\section{Operator expansions}

Let $\CC(Y)$ be the field of fractions of $\CC[Y,Y^{-1}]$.
As indicated in Section \ref{DAHAdefn}, as a left $\CC(Y)$-module, the localised double affine Hecke algebra 
$$\hbox{$\tilde H$\quad has $\CC(Y)$-basis}\qquad
\{ X^k\ |\ k\in \ZZ\} \cup \{ X^kT_1\ |\ k\in \ZZ\}.$$
Then $\tilde{H}\mathbf{1}_0$ is a $\CC(Y)$-subspace of $\tilde H$ and
$$\hbox{$\tilde H\mathbf{1}_0$\quad has $\CC(Y)$-basis}\qquad
\{ X^k\mathbf{1}_0\ |\ k\in \ZZ\},$$
since, by the first relation in \eqref{10rels}, $T_1\mathbf{1}_0 = t^{\frac12}\mathbf{1}_0$.
The sets
$$\{ E_k(X)\mathbf{1}_0\ |\ k\in \ZZ\}
\qquad\hbox{and}\qquad
\{ \eta^k(X)\mathbf{1}_0\ |\ k\in \ZZ\}
\qquad\hbox{are also $\CC(Y)$-bases of $\tilde H \mathbf{1}_0$,}
$$
and the results of this section and the next provide explicit product formulas for the transition coefficients between these bases.

\subsection{Definition of $D_j^{(\ell)}(Y)$ and $K_j^{(\ell)}(Y)$}

Define functions $D_j^{(\ell-1)}(Y)$ for $\ell\in \ZZ_{>0}$ and $D_j^{(-\ell)}(Y)$ for $\ell\in \ZZ_{\ge 0}$ by the expansions
\begin{equation}
E_\ell(X)\mathbf{1}_0
= \sum_{j=0}^{\ell-1} \eta^{\ell-2j} D^{(\ell-1)}_j(Y)\mathbf{1}_0
\qquad\hbox{and}\qquad
E_{-\ell}(X)\mathbf{1}_0
=\sum_{j=0}^\ell \eta^{-\ell+2j} D_j^{(-\ell)}(Y)\mathbf{1}_0.
\label{Ddefn}
\end{equation}
Define $K_j^{(\ell)}(Y)$ for $\ell\in \ZZ_{\ge 0}$ and $j\in \{0,1, \ldots, \ell\}$ by 
\begin{equation}
\mathbf{1}_0 E_\ell(X)\mathbf{1}_0 = \sum_{j=0}^\ell \mathbf{1}_0 \eta^{\ell-2j}K_j^{(\ell)}(Y).
\label{Kdefn}
\end{equation}
For example,
$$
\begin{array}{rrrrrrrrrrr}
E_{3}(X)\mathbf{1}_0 = 
&\eta^3 D_0^{(2)}(Y)\mathbf{1}_0
&+\eta D_1^{(2)}(Y)\mathbf{1}_0
&+\eta^{-1}D_2^{(2)}(Y)\mathbf{1}_0,
\\
\\
E_{-3}(X)\mathbf{1}_0 = 
&\eta^3 D_3^{(-3)}(Y)\mathbf{1}_0
&+\eta D_2^{(-3)}(Y)\mathbf{1}_0
&+\eta^{-1}D_1^{(-3)}(Y)\mathbf{1}_0
&+\eta^{-3}D_0^{(-3)}(Y)\mathbf{1}_0,
\end{array}
$$
and
$$\mathbf{1}_0 E_3(X)\mathbf{1}_0 
= \mathbf{1}_0 \eta^3 K_0^{(3)}(Y) + \mathbf{1}_0\eta K_1^{(3)}(Y)
+\mathbf{1}_0\eta^{-1} K_2^{(3)}(Y)+\mathbf{1}_0\eta^{-3}K_3^{(3)} (Y).
$$
See Section \ref{DKSL2Ex} for examples of the first few of the functions $D_j^{(\ell)}(Y)$ and
$K_j^{(\ell)}(Y)$.
Proposition \ref{KfromD} below
provides a formula for the $K_j^{(\ell)}(Y)$ in terms of the $D_j^{(\ell)}(Y)$.

\subsection{A recursion for the $D_j^{(\ell)}(Y)$}

The following Proposition 
provides recursions determining the $D_j^{(\ell)}(Y)$,
showing that the $D_j^{(\ell)}$ for $\ell\in \ZZ_{\ge 0}$ and $j\in \{0, \ldots, \ell-1\}$ form 
something like a Pascal triangle,
$$
\begin{array}{cccccccc}
&&&D_0^{(0)} \\
&&D_0^{(1)}  &&D_1^{(1)} \\
&D_0^{(2)}  &&D_1^{(2)} &&D_2^{(2)} \\
D_0^{(3)}  &&D_1^{(3)} &&D_2^{(3)} &&D_3^{(3)} 
\end{array}
$$


\begin{prop} \label{Drecursion} Let $D_j^{(\ell)}(Y)$ and $D_j^{(-\ell)}(Y)$ be as defined in \eqref{Ddefn}.
\item[(a)] If $\ell\in \ZZ_{>0}$ and $j\in \{0, \ldots, \ell\}$ then
$D_j^{(-\ell)}(Y) = t^{\frac12} D^{(\ell)}_j(Y^{-1}).$ 
\item[(b)]  
The $D_j^{(\ell)}(Y)$ satisfy, and are determined by, $D_0^{(0)}(Y)=t^{-\frac12}$ and the recursion
$$
D_0^{(\ell)} = t^{\frac12} \frac{(1-t^{-1}Y^{-2} q^{\ell})}{(1-Y^{-2} q^{\ell})}D_0^{(\ell-1)}(Y),
\qquad
D_\ell^{(\ell)} = t^{-\frac12} \frac{(1-t)(1-tY^{-2})}{(1-tq^{\ell})(1-Y^{-2}q^{-\ell})}D_0^{(\ell-1)}(Y^{-1}),
$$
\begin{equation}
D_j^{(\ell)}(Y) = t^{\frac12} \frac{(1-t^{-1}Y^{-2} q^{\ell-2j})}{(1-Y^{-2} q^{\ell-2j})}D_j^{(\ell-1)}(Y)
+t^{-\frac12} \frac{(1-t)(1-q^\ell tY^{-2}q^{\ell-2j})}{(1-tq^{\ell})(1-Y^{-2}q^{\ell-2j})}D_{\ell-j}^{(\ell-1)}(Y^{-1}).
\label{Drec}
\end{equation}
\end{prop}
\begin{proof}  (a)
Using the first relation in \eqref{Ddefn}, the second relation in \eqref{E10}, the second relation in \eqref{Ddefn}, 
$$\sum_{j=0}^\ell \eta^{\ell+1-2j} D^{(\ell)}_j(Y)\mathbf{1}_0
=E_{\ell+1}(X)\mathbf{1}_0
= t^{-\frac12} \tau_\pi^\vee E_{-\ell}(X)\mathbf{1}_0
=t^{-\frac12} \eta_\pi \sum_{j=0}^\ell \eta^{-\ell+2j} D_j^{(-\ell)}(Y)\mathbf{1}_0.
$$
By \eqref{etarels}, \eqref{etapastY}, and the second relation in \eqref{10ineta},
\begin{align*}
\eta_\pi \eta^{-\ell+2j} D_j^{(-\ell)}(Y)\mathbf{1}_0
&=\eta_\pi \eta_{s_1}\eta_{s_1}\eta^{-\ell+2j} D_j^{(-\ell)}(Y)\mathbf{1}_0
=\eta\eta_{s_1}\eta^{-\ell+2j} D_j^{(-\ell)}(Y)\mathbf{1}_0
\\
&=\eta\eta^{\ell-2j} D_j^{(-\ell)}(Y^{-1})\eta_{s_1}\mathbf{1}_0
=\eta^{\ell-2j+1} D_j^{(-\ell)}(Y^{-1})\mathbf{1}_0.
\end{align*}
So $\displaystyle{
\sum_{j=0}^\ell \eta^{\ell+1-2j} D^{(\ell)}_j(Y)\mathbf{1}_0
=t^{-\frac12}\sum_{j=0}^\ell \eta^{\ell-2j+1}  D_j^{(-\ell)}(Y^{-1}) \mathbf{1}_0,
}$ giving
$D_j^{(\ell)}(Y) = t^{-\frac12} D_j^{(-\ell)}(Y^{-1})$.

\smallskip\noindent
(b)
Let
$$c(Y) =  \frac{(1-tY^2)}{(1-Y^2)}\qquad\hbox{and}\qquad
F_\ell(Y) = \frac{(1-t)}{(1- t q^\ell )}\frac{(1-q^\ell t Y^2  )}{(1-Y^2)},\quad\hbox{for $\ell\in \ZZ_{>0}$.}
$$
By \eqref{Ddefn} and the first relation in \eqref{etaonE10},
\begin{align*}
\sum_{j=0}^\ell \eta^{-\ell+2j} D_j^{(-\ell)}(Y)\mathbf{1}_0
&= E_{-\ell}(X)\mathbf{1}_0 
= (\eta_{s_1}c(Y)+F_\ell(Y))E_\ell(X) \mathbf{1}_0
\\
&=(\eta_{s_1}c(Y)+F_\ell(Y))\sum_{j=0}^{\ell-1} \eta^{\ell-2j}D_j^{(\ell-1)}(Y)\mathbf{1}_0
\end{align*}
By \eqref{etapastY} and the last relation in \eqref{etarels},
\begin{align*}
\eta_{s_1} c(Y)\eta^{\ell-2j}D^{(\ell-1)}_j(Y) 
&= \eta_{s_1} \eta^{\ell-2j}c(q^{-\frac12(\ell-2j)}Y)D^{(\ell-1)}_j(Y)
= \eta^{-(\ell-2j)}\eta_{s_1} c(q^{-\frac12(\ell-2j)}Y)D^{(\ell-1)}_j(Y) 
\\
&= \eta^{-(\ell-2j)} c(q^{-\frac12(\ell-2j)}Y^{-1})D^{(\ell-1)}_j(Y^{-1}) \eta_{s_1}
\end{align*}
and
$F_\ell(Y) \eta^{\ell-2j}D^{(\ell-1)}_j(Y)
=\eta^{\ell-2j}F_\ell(q^{-\frac12(\ell-2j)}Y) D^{(\ell-1)}_j(Y).
$
So
\begin{align*}
\sum_{j=0}^\ell &\eta^{-(\ell-2j)} D_j^{(-\ell)}(Y)\mathbf{1}_0
=(\eta_{s_1}c(Y)+F_\ell(Y))\sum_{j=0}^{\ell-1} \eta^{\ell-2j}D_j^{(\ell-1)}(Y)\mathbf{1}_0
\\
&=\sum_{j=0}^{\ell-1} \eta^{-(\ell-2j)}c(q^{-\frac12(\ell-2j)}Y^{-1})D_j^{(\ell-1)}(Y^{-1}) \mathbf{1}_0
+\sum_{j=0}^{\ell-1} \eta^{\ell-2j}F_\ell(q^{-\frac12(\ell-2j)}Y) D_j^{(\ell-1)}(Y)\mathbf{1}_0
\end{align*}
where the last equality uses the relation $\eta_{s_1}\mathbf{1}_0 = \mathbf{1}_0$ from \eqref{10ineta}.
Putting $k=\ell-j$ in the second sum makes $j=\ell-k$ and
\begin{align*}
\sum_{j=0}^{\ell-1} \eta^{\ell-2(\ell-k)} F_\ell(q^{-\frac12(\ell-2(\ell-k))}Y)D_j^{(\ell-1)}(Y)\mathbf{1}_0
&=\sum_{k=1}^{\ell} \eta^{-(\ell-2k)}F_\ell(q^{\frac12(\ell-2k)}Y)
D_{\ell-k}^{(\ell-1)}(Y)\mathbf{1}_0
\end{align*}
Thus
$\sum_{j=0}^\ell \eta^{-\ell+2j} D_j^{(-\ell)}(Y)\mathbf{1}_0$ is equal to
\begin{align*}
&\eta^{-\ell} c(q^{-\frac12\ell}Y^{-1})D_0^{(\ell-1)}(Y^{-1})\mathbf{1}_0
+\eta^\ell F_\ell(q^{-\frac12\ell}Y)D_0^{(\ell-1)}(Y)\mathbf{1}_0
\\
&\qquad
+\sum_{j=1}^{\ell-1} \eta^{-(\ell-2j)}\big(
c(q^{-\frac12(\ell-2j)}Y^{-1})D_j^{(\ell-1)}(Y^{-1}) + F_\ell(q^{\frac12(\ell-2j)}Y) D_{\ell-j}^{(\ell-1)}(Y)
\big)
\mathbf{1}_0,
\end{align*}
which implies
$$D_0^{(-\ell)}(Y) = c(q^{-\frac12\ell} Y^{-1})D_0^{(\ell-1)}(Y^{-1}), 
\qquad\qquad
D_\ell^{(-\ell)}(Y) = F_\ell(q^{-\frac12 \ell} Y)D_0^{(\ell-1)}(Y),
\qquad\hbox{and}
$$
\begin{equation*}
D_j^{(-\ell)}(Y) = c(q^{-\frac12(\ell-2j)}Y^{-1})D_j^{(\ell-1)}(Y^{-1})+F_\ell(q^{\frac12(\ell-2j)}Y) D_{\ell-j}^{(\ell-1)}(Y),
\end{equation*}
for $j\in \{1, \ldots, \ell-1\}$.
Combining this with the identity $D_j^{(\ell)}(Y) = t^{-\frac12} D_j^{(-\ell)}(Y^{-1})$ from (a)  gives
$$D_0^{(\ell)}(Y) = t^{-\frac12} c(q^{-\frac12\ell} Y)D_0^{(\ell)}(Y), 
\qquad\qquad
D_\ell^{(\ell)}(Y) = t^{-\frac12} F_\ell(q^{-\frac12 \ell} Y^{-1})D_0^{(\ell-1)}(Y^{-1}),
\qquad\hbox{and}
$$
$$D_j^{(\ell)}(Y) = t^{-\frac12} c(q^{-\frac12(\ell-2j)}Y)D_j^{(\ell-1)}(Y)
+t^{-\frac12} F_\ell(q^{\frac12(\ell-2j)}Y^{-1}) D_{\ell-j}^{(\ell-1)}(Y^{-1}).$$
\end{proof}

\subsection{A formula for $K_j^{(\ell)}(Y)$ in terms of the $D_j^{(\ell)}(Y)$}

Since $E_0(X) = 1$, equation
\eqref{Kdefn} and the first relation in \eqref{10intau} give
$\mathbf{1}_0 E_0(X) \mathbf{1}_0 = \mathbf{1}_0^2 = \mathbf{1}_0(t^{\frac12}+t^{-\frac12})$ so that
$$K_0^{(0)} = t^{\frac12}+t^{-\frac12}.$$


\begin{prop} \label{KfromD}  Let $\ell\in \ZZ_{>0}$ and let $K_j^{(\ell)}(Y)$ and $D_j^{(\ell-1)}(Y)$ be as defined in 
\eqref{Kdefn} and \eqref{Ddefn}.  Then
$$K_0^{(\ell)}(Y) = t^{-\frac12} D_0^{(\ell-1)}(Y) \frac{(1-tY^2)}{(1-Y^2)},
\qquad
K_\ell^{(\ell)}(Y) = t^{\frac12}D_0^{(\ell-1)}(Y^{-1})
\frac{(1-t^{-1}Y^2q^\ell)(1-tY^2)}{(1-tY^2q^\ell)(1-Y^2)},
$$
and, for $j\in \{1, \ldots, \ell-1\}$,
\begin{align}
K_j^{(\ell)}(Y) 
&= t^{\frac12} D_j^{(\ell-1)}(Y)\frac{(1-t^{-1}Y^{-2}) }{ (1-Y^{-2})}
+ t^{-\frac12} D_{\ell-j}^{(\ell-1)}(Y^{-1}) \frac{(1-tY^{-2} q^{\ell-2j})}{ (1-t^{-1}Y^{-2} q^{\ell-2j})}
\frac{(1-t^{-1}Y^2)}{(1-Y^{-2})}.
\label{Krec}
\end{align}
\end{prop}
\begin{proof}  Let 
$$c(Y) = \frac{(1-tY^2)}{(1-Y^2)} = t\frac{(1-t^{-1}Y^{-2})}{(1-Y^{-2})}.$$  
Then, by \eqref{Ddefn}, the first relation in \eqref{10ineta},
the second relation in \eqref{etapastY}
and the relation $\eta \eta_{s_1} = \eta_{s_1}\eta^{-1}$ from \eqref{etarels},
\begin{align*}
\mathbf{1}_0 E_\ell(X) \mathbf{1}_0
&= \mathbf{1}_0 \sum_{j=0}^{\ell-1} \eta^{\ell-2j} D_j^{(\ell-1)}(Y) \mathbf{1}_0
= \mathbf{1}_0 \sum_{j=0}^{\ell-1} \eta^{\ell-2j} D_j^{(\ell-1)}(Y) (1+\eta_{s_1}) t^{-\frac12} c(Y)
\\
&= \sum_{j=0}^{\ell-1} \mathbf{1}_0\eta^{\ell-2j} D_j^{(\ell-1)}(Y) t^{-\frac12} c(Y) 
+\sum_{j=0}^{\ell-1} \mathbf{1}_0\eta_{s_1}\eta^{-(\ell-2j)} D_j^{(\ell-1)}(Y^{-1}) t^{-\frac12} c(Y).
\end{align*}
By the last relation in \eqref{10ineta} and the last relation in \eqref{etapastY}, 
$$\mathbf{1}_0\eta_{s_1}\eta^{-(\ell-2j)}
=\mathbf{1}_0 t^{-1} \frac{(1-tY^{-2})}{(1-t^{-1}Y^{-2})} \eta^{-(\ell-2j)}
= \mathbf{1}_0 \eta^{-(\ell-2j)} t^{-1} \frac{(1-tY^{-2}q^{-(\ell-2j)}) }{ (1-t^{-1}Y^{-2}q^{-(\ell-2j)})}
$$
and thus, by reindexing with $k = \ell -j$ and using $-(\ell-2j) = \ell-2k$, the second sum is
\begin{align*}
\sum_{j=0}^{\ell-1} \mathbf{1}_0\eta_{s_1}\eta^{-(\ell-2j)} D_j^{(\ell-1)}(Y^{-1}) t^{-\frac12} c(Y)
&=\sum_{k=1}^\ell
\mathbf{1}_0   \eta^{\ell-2k} t^{-1} \frac{(1-tY^{-2}q^{\ell-2k}) }{ (1-t^{-1}Y^{-2}q^{\ell-2k})}
D_{\ell-k}^{(\ell-1)}(Y^{-1}) t^{-\frac12} c(Y).
\end{align*}
Hence, by \eqref{Kdefn}, 
\begin{align*}
\sum_{j=0}^\ell &\mathbf{1}_0 \eta^{\ell-2j}K_j^{(\ell)}(Y)
= \mathbf{1}_0 E_\ell(X)\mathbf{1}_0 
\\
&=\sum_{j=0}^{\ell-1} \mathbf{1}_0t^{-\frac12} \eta^{\ell-2j} D_j^{(\ell-1)}(Y)  c(Y) 
+
\sum_{k=1}^\ell
\mathbf{1}_0  t^{-\frac32} \eta^{\ell-2k}  \frac{(1-tY^{-2}q^{\ell-2k}) }{ (1-t^{-1}Y^{-2}q^{\ell-2k})}
D_{\ell-k}^{(\ell-1)}(Y^{-1}) c(Y).
\end{align*}
\end{proof}


\section{Product expressions for $D_j^{(\ell)}(Y)$ and $K_j^{(\ell)}(Y)$}

In this section we establish product formulas for the coefficients in the operators
\begin{equation*}
E_\ell(X)\mathbf{1}_0
= \sum_{j=0}^{\ell-1} \eta^{\ell-2j} D^{(\ell-1)}_j(Y)\mathbf{1}_0,
\qquad
E_{-\ell}(X)\mathbf{1}_0
=\sum_{j=0}^\ell \eta^{-\ell+2j} D_j^{(-\ell)}(Y)\mathbf{1}_0
\qquad\hbox{and}
\end{equation*}
\begin{equation*}
\mathbf{1}_0 E_\ell(X)\mathbf{1}_0 = \sum_{j=0}^\ell \mathbf{1}_0 \eta^{\ell-2j}K_j^{(\ell)}(Y).
\end{equation*}
These coefficients turn out to be something like generalized binomial coefficients, determined by the recursions that
were established in Proposition \ref{Drecursion} and \ref{KfromD}.  
These product formulas provide a kind of binomial theorem for
operators $E_\ell(X)\mathbf{1}_0$ and $\mathbf{1}_0 E_\ell(X)\mathbf{1}_0$, as elements of 
the double affine Hecke algebra $\tilde{H}$.  The final result is Theorem \ref{DandKresult}.

\subsection{$q$-$t$-binomial coefficients}

For $\ell\in \ZZ_{\ge 0}$ and $j\in \{0, \ldots, \ell\}$ define
\begin{equation}
(z;q)_j = (1-z)(1-zq)(1-zq^2)\cdots (1-zq^{j-1})
\qquad\hbox{and}\qquad
\genfrac[]{0pt}{0}{\ell}{j}_{q,t} 
= \frac{ \frac{(q;q)_\ell}{(t;q)_\ell} } { \frac{(q;q)_j}{(t;q)_j} \frac{(q;q)_{\ell-j}}{(t;q)_{\ell-j}} }.
\label{qtbin}
\end{equation}
For $a,b\in \ZZ$ with $a\le b$ let
$$(1-zq^{a..b}) = (1-zq^a)(1-zq^{a+1})\cdots (1-zq^{b-1})(1-zq^b).$$
With this notation,
\begin{align*}
\genfrac[]{0pt}{0}{\ell}{j}_{q,t} 
&
= \frac{(1-q^{1..\ell})}{(1-q^{1..j})(1-q^{1..\ell-j})}\frac{(1-tq^{0..j-1})(1-tq^{0..\ell-j-1})}{(1-tq^{0..\ell-1})}
=\frac{(1-q^{j+1..\ell})}{(1-q^{1..\ell-j})}\frac{(1-tq^{0..\ell-j-1})}{(1-tq^{j..\ell-1})}.
\end{align*}
Then
\begin{align}
\genfrac[]{0pt}{0}{\ell}{\ell-j}_{q,t}
=\genfrac[]{0pt}{0}{\ell}{j}_{q,t}
\qquad\hbox{and}\qquad
\genfrac[]{0pt}{0}{\ell+1}{j}_{q,t}
=\genfrac[]{0pt}{0}{\ell}{j}_{q,t}\cdot
\frac{(1-q^{\ell+1})(1-tq^{\ell-j})}{(1-q^{\ell+1-j})(1-tq^\ell)}.
\label{qtbinrels}
\end{align}

\subsection{$Y$-binomial coefficients}

For $\ell\in \ZZ_{\ge 0}$ and $j\in \{0,1,\ldots, \ell\}$ define a rational function in $Y$ by
\begin{align}
\genfrac(){0pt}{0}{\ell}{j}_{Y} 
=
\frac{(t^{-1}Y^{-2} q^{-(j-1)};q)_{\ell-j} (tY^{-2}q^{\ell-2j};q)_j}{  (Y^{-2} q;q)_{\ell-j} (Y^{-2}q^{-j};q)_j }.
\label{Ybin}
\end{align}
An alternative expression is
$$
\genfrac(){0pt}{0}{\ell}{j}_{Y} 
= 
\frac{(1-t^{-1}Y^{-2} q^{-(j-1)..\ell-2j}) (1-tY^{-2} q^{\ell-2j..\ell-j-1})}
{(1-Y^{-2}q^{1..\ell-j}) (1-Y^{-2}q^{-j..-1})}.
$$
which, in the alcove walk point of view of \cite{Yi10}, might be thought of as a weighted alcove walk
of total length $\ell$ with $j$ left moving crossings and $\ell-j$ right moving crossings.
$$
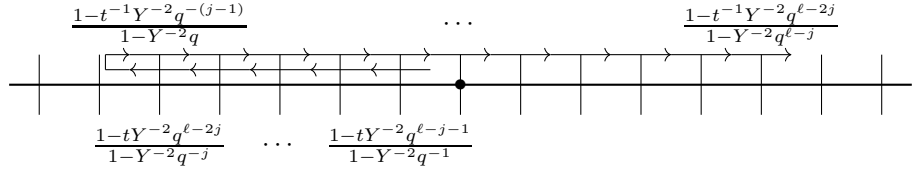
\begin{tikzpicture}[scale=0.8]
\tikzstyle{every node}=[font=\small]
	\node at (0, 0) {$\bullet$};  
      \draw[-,thick] (-7.5,0) -- (7.5,0) ;
      \draw[-] (-7,-0.5) -- (-7,0.5) ;
      \draw[-] (-6,-0.5) -- (-6,0.5) ;
      \draw[-] (-5,-0.5) -- (-5,0.5) ;
      \draw[-] (-4,-0.5) -- (-4,0.5) ;
      \draw[-] (-3,-0.5) -- (-3,0.5) ;
      \draw[-] (-2,-0.5) -- (-2,0.5) ;
      \draw[-] (-1,-0.5) -- (-1,0.5) ;
      \draw[-] (0,-0.5) -- (0,0.5) ;
      \draw[-] (1,-0.5) -- (1,0.5) ;
      \draw[-] (2,-0.5) -- (2,0.5) ;
      \draw[-] (3,-0.5) -- (3,0.5) ;
      \draw[-] (4,-0.5) -- (4,0.5) ;
      \draw[-] (5,-0.5) -- (5,0.5) ;
      \draw[-] (6,-0.5) -- (6,0.5) ;
      \draw[-] (7,-0.5) -- (7,0.5) ;
\draw[->] (-0.5,0.25) -- (-1.5,0.25);
\draw[->] (-1.5,0.25) -- (-2.5,0.25);
\draw[->] (-2.5,0.25) -- (-3.5,0.25);
\draw[->] (-3.5,0.25) -- (-4.5,0.25);
\draw[->] (-4.5,0.25) -- (-5.5,0.25);
\draw[-] (-5.5,0.25) -- (-5.9,0.25);
\draw[-] (-5.9,0.25) -- (-5.9,0.5);
\draw[->] (-5.9,0.5) -- (-5.5,0.5);
\draw[->] (-5.5,0.5) -- (-4.5,0.5);
\draw[->] (-4.5,0.5) -- (-3.5,0.5);
\draw[->] (-3.5,0.5) -- (-2.5,0.5);
\draw[->] (-2.5,0.5) -- (-1.5,0.5);
\draw[->] (-1.5,0.5) -- (-0.5,0.5);
\draw[->] (-0.5,0.5) -- (0.5,0.5);
\draw[->] (0.5,0.5) -- (1.5,0.5);
\draw[->] (1.5,0.5) -- (2.5,0.5);
\draw[->] (2.5,0.5) -- (3.5,0.5);
\draw[->] (3.5,0.5) -- (4.5,0.5);
\draw[->] (4.5,0.5) -- (5.5,0.5);
\node at (-5,1) {$\frac{1-t^{-1}Y^{-2}q^{-(j-1)}}{1-Y^{-2}q}$};
\node at (0,1) {$\cdots$};
\node at (5,1) {$\frac{1-t^{-1}Y^{-2}q^{\ell-2j}}{1-Y^{-2}q^{\ell-j}}$};
\node at (-5,-1) {$\frac{1-tY^{-2}q^{\ell-2j}}{1-Y^{-2}q^{-j}}$};
\node at (-3,-1) {$\cdots$};
\node at (-1,-1) {$\frac{1-tY^{-2}q^{\ell-j-1}}{1-Y^{-2}q^{-1}}$};
\end{tikzpicture}
$$
Then (see the examples in Section \ref{Ybinexamples})
\begin{align}
\genfrac(){0pt}{0}{\ell}{\ell-j}_{Y^{-1}} 
&=
\frac{(1-t^{-1}Y^{2} q^{-(\ell-j-1)..\ell-2(\ell-j)}) (1-tY^{2} q^{\ell-2(\ell-j)..\ell-(\ell-j)-1})}
{(1-Y^{2}q^{1..\ell-(\ell-j)}) (1-Y^{2}q^{-(\ell-j)..-1})}
\nonumber \\
&=
\frac{(1-t^{-1}Y^{2} q^{-(\ell-j-1)..-(\ell-2j)}) (1-tY^{2} q^{-(\ell-2j)..j-1})}
{(1-Y^{2}q^{1..j}) (1-Y^{2}q^{-(\ell-j)..-1})}
\nonumber \\
&= 
\frac{(1-tY^{-2} q^{\ell-2j..\ell-j-1}) (1-t^{-1}Y^{-2} q^{-(j-1)..\ell-2j})}
{(1-Y^{-2}q^{-j..-1}) (1-Y^{-2}q^{1..\ell-j})}
\nonumber \\
&\qquad
\cdot \frac{ (t^{-1}Y^2)^j (q^{-(\ell-j)})^j q^{\frac12j(j-1)} (tY^2)^{\ell-j} (q^{-(\ell-2j)-1})^{\ell-j}q^{\frac12(\ell-j)(\ell-j-1)}}
{Y^{2j} q^{\frac12j(j-1)} Y^{2(\ell-j)} (q^{-(\ell-j+1)})^{\ell-j} q^{\frac12(\ell-j)(\ell-j-1)}}
\nonumber \\
&= t^{\ell-2j} q^{0} \cdot \genfrac(){0pt}{0}{\ell}{j}_{Y} .
\label{Ybinflipped}
\end{align}
Also
\begin{align}
\genfrac(){0pt}{0}{\ell+1}{j}_{Y} 
&= 
\frac{(1-t^{-1}Y^{-2} q^{-(j-1)..\ell+1-2j}) (1-tY^{-2} q^{\ell+1-2j..\ell+1-j-1})}
{(1-Y^{-2}q^{1..\ell+1-j}) (1-Y^{-2}q^{-j..-1})}
\nonumber \\
&= \genfrac(){0pt}{0}{\ell}{j}_{Y} \cdot
\frac{(1-t^{-1}Y^{-2}q^{\ell+1-2j})(1-tY^{-2}q^{\ell-j})}{(1-tY^{-2}q^{\ell-2j})(1-Y^{-2}q^{\ell+1-j})}
\label{Ybinup1}
\end{align}

\subsection{Definition of the products $\tilde K_j^{(\ell)}(Y)$ and $\tilde D_j^{(\ell-1)}(Y)$} 

For $\ell\in \ZZ_{>0}$ and $j\in \{0, \ldots, \ell\}$ define
\begin{align}
\tilde D_j^{(\ell)}(Y) &= t^{-\frac12(\ell+1)}\cdot
t^{\ell-j} \cdot 
\genfrac[]{0pt}{0}{\ell}{j}_{q,t}\cdot 
\genfrac(){0pt}{0}{\ell}{j}_{Y}\cdot 
\frac{(1-tq^{\ell-j})}{(1-tq^\ell)}\cdot 
\frac{(1-tY^{-2}q^{\ell-j})}{(1-tY^{-2}q^{\ell-2j})} 
\label{tildeDdefn}
\end{align}
\begin{align}
\tilde D_j^{(-\ell)}(Y)
=t^{-\frac12\ell}\cdot (qt)^j\cdot 
\genfrac[]{0pt}{0}{\ell}{j}_{q,t}\cdot 
\genfrac(){0pt}{0}{\ell}{\ell-j}_{Y}\cdot 
\frac{(1-tq^{\ell-j})}{(1-tq^\ell)}\cdot 
\frac{(1-t^{-1}Y^{-2}q^{-(\ell-j)})}{(1-t^{-1}Y^{-2}q^{-(\ell-2j)})}
\label{Dnegprod}
\end{align}
\begin{align}
\tilde K_j^{(\ell)}(Y) 
&= t^{-\frac12(\ell-1)}\cdot 
t^{\ell-1-j}\cdot
\genfrac[]{0pt}{0}{\ell}{j}_{q,t} \cdot
\genfrac(){0pt}{0}{\ell}{j}_{Y} \cdot
\frac{(1-Y^{-2}q^{\ell-2j})}{(1-t^{-1}Y^{-2}q^{\ell-2j})}\cdot
\frac{(1-t^{-1}Y^{-2})}{(1-Y^{-2})}
\label{tildeKdefn}
\end{align}
The following Proposition provides useful relationships between these expressions which follow from
\eqref{qtbinrels}, \eqref{Ybin} and \eqref{Ybinflipped}.

\begin{prop} 
\begin{equation}
\frac{\tilde D_j^{(\ell-1)}(Y)}{ \tilde K_j^{(\ell)}(Y) }
= t^{-\frac12} \cdot 
\frac{(1-q^{\ell-j})}{(1-q^\ell)}\cdot \frac{(1-Y^{-2}q^{\ell-j})}{(1-Y^{-2}q^{\ell-2j})}\cdot \frac{(1-Y^{-2})}{(1-t^{-1}Y^{-2})}
\label{DtoK}
\end{equation}
\begin{equation}
\frac{\tilde D^{(\ell-1)}_{\ell-j}(Y^{-1}) }{\tilde K_j^{(\ell)}(Y)}
=t^{\frac12} q^{\ell-j}\cdot 
\frac{ (1-q^j) }{(1-q^\ell)}\cdot
\frac{(1-Y^{-2})}{(1-t^{-1}Y^{-2})}\cdot
\frac{(1-t^{-1}Y^{-2}q^{\ell-2j})}{(1-tY^{-2}q^{\ell-2j})}\cdot
\frac{(1-Y^{-2}q^{-j})}{(1-Y^{-2}q^{\ell-2j})} .
\label{DinvtoK}
\end{equation}
\begin{align}
\frac{\tilde D_j^{(\ell)}(Y) }{ \tilde D^{(\ell+1)}_j(Y)}
= t^{-\frac12}  \cdot
\frac{(1-tq^{\ell+1})}{(1-tq^{\ell+1-j})}\cdot 
\frac{(1-q^{\ell+1-j})}{(1-q^{\ell+1})}\cdot
\frac{(1-tY^{-2}q^{\ell+1-2j})}{(1-tY^{-2}q^{\ell+1-j})}\cdot
\frac{(1-Y^{-2}q^{\ell+1-j})}{(1-t^{-1}Y^{-2}q^{\ell+1-2j})}.
\label{Dup1}
\end{align}
\begin{align}
\frac{ \tilde D_{\ell+1-j}^{(\ell)}(Y^{-1}) }{\tilde D^{(\ell+1)}_j(Y) }
&= 
t^{\frac12} q^{\ell+1-j} \cdot
\frac{(1-tq^{\ell+1})}{(1-tq^{\ell+1-j})}\cdot \frac{(1-q^j)}{(1-q^{\ell+1})}\cdot 
\frac{(1-Y^{-2}q^{-j})}{(1-tY^{-2}q^{\ell+1-j})}
\label{DflippedB}
\end{align}
\begin{align}
\frac{ \tilde D^{(\ell)}_{\ell+1-j}(Y^{-1}) }{\tilde D^{(\ell)}_j(Y) }
&=  tq^{\ell+1-j}\cdot \frac{(1-q^j)}{(1-q^{\ell+1-j})}\cdot 
\frac{(1-Y^{-2}q^{-j})(1-t^{-1}Y^{-2}q^{\ell+1-2j})}
{(1-tY^{-2}q^{\ell+1-2j})(1-Y^{-2}q^{\ell+1-j})},
\label{DinvtoD}
\end{align}
\end{prop}
\begin{proof}
All of these are proved by using the relations \eqref{qtbinrels}, \eqref{Ybin} and \eqref{Ybinflipped}
and cancelling all common factors from the numerator and denominators.  
The proof of \eqref{DtoK} is similar to the proofs of \eqref{Dup1}, \eqref{DflippedB} and
\eqref{DinvtoD} and \eqref{DinvtoK}, which are as follows.

Using \eqref{tildeDdefn}, \eqref{Ybinup1} and the second relation in \eqref{qtbinrels} gives \eqref{Dup1}:
\begin{align}
&\tilde D_j^{(\ell)}(Y) = 
t^{-\frac12(\ell+1)}\cdot
t^{\ell-j} \cdot 
\genfrac[]{0pt}{0}{\ell}{j}_{q,t}\cdot \frac{(1-tq^{\ell-j})}{(1-tq^\ell)}\cdot 
\genfrac(){0pt}{0}{\ell}{j}_{Y}\cdot \frac{(1-tY^{-2}q^{\ell-j})}{(1-tY^{-2}q^{\ell-2j})} 
\nonumber \\
&= 
t^{\frac12}t^{-\frac12(\ell+2)}\cdot
t^{-1}t^{\ell+1-j} \cdot 
\genfrac[]{0pt}{0}{\ell+1}{j}_{q,t}\cdot \frac{(1-q^{\ell+1-j})\cancel{(1-tq^\ell)} }{(1-q^{\ell+1})\cancel{(1-tq^{\ell-j})}}\cdot 
\frac{\cancel{(1-tq^{\ell-j})}}{\cancel{(1-tq^\ell)} }\cdot 
\nonumber \\
&\qquad \cdot
\genfrac(){0pt}{0}{\ell+1}{j}_{Y}\cdot 
\frac{\cancel{(1-tY^{-2}q^{\ell-2j})} (1-Y^{-2}q^{\ell+1-j})}{(1-t^{-1}Y^{-2}q^{\ell+1-2j}) \cancel{(1-tY^{-2}q^{\ell-j})} }\cdot
\frac{\cancel{(1-tY^{-2}q^{\ell-j})} }{\cancel{(1-tY^{-2}q^{\ell-2j})} } 
\nonumber \\
&= t^{-\frac12} \tilde D^{(\ell+1)}_j(Y) \cdot
\frac{(1-tq^{\ell+1})}{(1-tq^{\ell+1-j})}\cdot 
\frac{(1-q^{\ell+1-j})}{(1-q^{\ell+1})}\cdot
\frac{(1-tY^{-2}q^{\ell+1-2j})}{(1-tY^{-2}q^{\ell+1-j})}\cdot
\frac{(1-Y^{-2}q^{\ell+1-j})}{(1-t^{-1}Y^{-2}q^{\ell+1-2j})}.
\nonumber
\end{align}
Using \eqref{tildeDdefn}, \eqref{Ybinflipped} and the first relation in \eqref{qtbinrels} gives \eqref{DflippedB}:
\begin{align}
&t^{\frac12(\ell+1)}\tilde D_{\ell+1-j}^{(\ell)}(Y^{-1})
= 
t^{\ell-(\ell+1-j)}\cdot 
\genfrac[]{0pt}{0}{\ell}{\ell+1-j}_{q,t} \cdot
\frac{(1-tq^{\ell-(\ell+1-j)})}{(1-tq^\ell)} \cdot
\genfrac(){0pt}{0}{\ell}{\ell+1-j}_{Y^{-1}} \cdot 
\frac{(1-tY^{2}q^{\ell-(\ell+1-j)})}{(1-t Y^{2}q^{\ell-2(\ell+1-j)})}
\nonumber \\
&= 
t^{j-1}\cdot 
\genfrac[]{0pt}{0}{\ell+1}{\ell+1-j}_{q,t} \cdot 
\frac{(1-q^{\ell+1-(\ell+1-j)}) \cancel{(1-tq^\ell)} }{(1-q^{\ell+1}) \cancel{(1-tq^{\ell-(\ell+1-j)})} }\cdot
\frac{ \cancel{(1-tq^{j-1})} }{ \cancel{(1-tq^\ell)} }
\nonumber \\
&\qquad
\cdot
\genfrac(){0pt}{0}{\ell+1}{\ell+1-j}_{Y^{-1}} \cdot 
\frac{ \cancel{(1-tY^{2}q^{\ell-2(\ell+1-j)})} (1-Y^{2}q^{\ell+1-(\ell+1-j)})}
{(1-t^{-1}Y^{2}q^{\ell+1-2(\ell+1-j)}) \cancel{(1-tY^{2}q^{\ell-(\ell+1-j)})} }\cdot
\frac{\cancel{(1-tY^{2}q^{\ell-(\ell+1-j)})} }{ \cancel{(1-t Y^{2}q^{\ell-2(\ell+1-j)})} }
\nonumber \\
&= 
t^{j-1}\cdot 
\genfrac[]{0pt}{0}{\ell+1}{j}_{q,t} \cdot \frac{(1-q^j) }{ (1-q^{\ell+1}) }\cdot
t^{\ell+1-2j}\cdot \genfrac(){0pt}{0}{\ell+1}{j}_{Y} \cdot 
\frac{(1-Y^2q^j)}{(1-t^{-1}Y^2 q^{-(\ell+1-2j)})}
\nonumber \\
&= 
t^{\ell-j}\cdot
\genfrac[]{0pt}{0}{\ell+1}{j}_{q,t} \cdot
\frac{(1-q^j)}{(1-q^{\ell+1})} \cdot 
\genfrac(){0pt}{0}{\ell+1}{j}_{Y} \cdot
\frac{(1-Y^{-2}q^{-j})}{(1-tY^{-2}q^{\ell+1-2j})}\cdot tq^{j+\ell+1-2j}
\nonumber \\
&= 
q^{\ell+1-j} \cdot
t^{\frac12(\ell+2)} \tilde D^{(\ell+1)}_j(Y)\cdot
\frac{(1-tq^{\ell+1})}{(1-tq^{\ell+1-j})}\cdot \frac{(1-q^j)}{(1-q^{\ell+1})}\cdot 
\frac{\cancel{(1-tY^{-2}q^{\ell+1-2j})} }{(1-tY^{-2}q^{\ell+1-j})}\cdot
\frac{(1-Y^{-2}q^{-j})}{\cancel{(1-tY^{-2}q^{\ell+1-2j})} }
\nonumber \\
&= 
q^{\ell+1-j} \cdot
t^{\frac12(\ell+2)} \tilde D^{(\ell+1)}_j(Y)\cdot
\frac{(1-tq^{\ell+1})}{(1-tq^{\ell+1-j})}\cdot \frac{(1-q^j)}{(1-q^{\ell+1})}\cdot 
\frac{(1-Y^{-2}q^{-j})}{(1-tY^{-2}q^{\ell+1-j})}.
\nonumber
\end{align}
Equation \eqref{DinvtoD} follows from
\eqref{DflippedB} and \eqref{Dup1}.
Using \eqref{DinvtoD} and \eqref{DtoK}
\begin{align*}
D^{(\ell-1)}_{\ell-j}(Y^{-1})
&=t q^{\ell-j} D^{(\ell-1)}_j(Y)\cdot
\frac{(1-q^j)}{(1-q^{\ell-j})}\frac{(1-t^{-1}Y^{-2}q^{\ell-2j})}{(1-tY^{-2}q^{\ell-2j})}\cdot
\frac{(1-Y^{-2}q^{-j})}{(1-Y^{-2}q^{\ell-j})}
\\
&=t^{\frac12} q^{\ell-j}\cdot K_j^{(\ell)}(Y)\cdot 
\frac{ \cancel{(1-q^{\ell-j})} }{(1-q^\ell)}\cdot
\frac{ \cancel{(1-Y^{-2}q^{\ell-j})} }{(1-Y^{-2}q^{\ell-2j})}\cdot
\frac{(1-Y^{-2})}{(1-t^{-1}Y^{-2})}
\\
&\qquad\cdot
\frac{(1-q^j)}{\cancel{(1-q^{\ell-j})} }\frac{(1-t^{-1}Y^{-2}q^{\ell-2j})}{(1-tY^{-2}q^{\ell-2j})}\cdot
\frac{(1-Y^{-2}q^{-j})}{\cancel{(1-Y^{-2}q^{\ell-j})} }
\end{align*}
which gives equation \eqref{DinvtoK}.
\end{proof}

\subsection{A recursion for the $\tilde D_j^{(\ell)}(Y)$}

\begin{prop} \label{tildeDrecursion} Let $\ell\in \ZZ_{\ge0}$ and let $\tilde D_j^{(\ell)}(Y)$ and $\tilde D_j^{(-\ell)}(Y)$ be as defined in \eqref{Ddefn}.
\item[(a)] If $\ell\in \ZZ_{>0}$ and $j\in \{0, \ldots, \ell\}$ then
$\tilde D_j^{(-\ell)}(Y) = t^{\frac12} \tilde D^{(\ell)}_j(Y^{-1}).$ 
\item[(b)]  
The $\tilde D_j^{(\ell)}(Y)$ satisfy, and are determined by, $\tilde D_0^{(0)}(Y)=t^{-\frac12}$ and the recursions
$$
\tilde D_0^{(\ell)}(Y) = t^{\frac12}\frac{(1-t^{-1}Y^{-2} q^{\ell})}{(1-Y^{-2} q^{\ell})}\tilde D_0^{(\ell-1)}(Y),
\qquad
\tilde D_\ell^{(\ell)}(Y) = t^{-\frac12} \frac{(1-t)(1-tY^{-2})}{(1-tq^{\ell})(1-Y^{-2}q^{-\ell})}\tilde D_0^{(\ell-1)}(Y^{-1}),
$$
and 
\begin{equation}
\tilde D_j^{(\ell)}(Y) = t^{\frac12} \frac{(1-t^{-1}Y^{-2} q^{\ell-2j})}{(1-Y^{-2} q^{\ell-2j})}\tilde D_j^{(\ell-1)}(Y)
+t^{-\frac12} \frac{(1-t)(1-q^\ell tY^{-2}q^{\ell-2j})}{(1-tq^{\ell})(1-Y^{-2}q^{\ell-2j})}\tilde D_{\ell-j}^{(\ell-1)}(Y^{-1}).
\label{tildeDrec}
\end{equation}
for $j\in \{1, \ldots, \ell-1\}$.
\end{prop}
\begin{proof} (a)
Using \eqref{Ybinflipped},
\begin{align*}
t^{\ell-j}\cdot &\genfrac(){0pt}{0}{\ell}{j}_{Y^{-1}}\cdot \frac{(1-tY^{2}q^{\ell-j})}{(1-tY^{2}q^{\ell-2j})} 
=
t^{\ell-j}\cdot t^{-(\ell-2j)}\genfrac(){0pt}{0}{\ell}{\ell-j}_{Y}\cdot 
\frac{(1-t^{-1}Y^{-2}q^{-(\ell-j)})}{(1-t^{-1}Y^{-2}q^{-(\ell-2j)})}\cdot q^j 
\\
&=
(qt)^j \genfrac(){0pt}{0}{\ell}{\ell-j}_{Y}\cdot 
\frac{(1-t^{-1}Y^{-2}q^{-(\ell-j)})}{(1-t^{-1}Y^{-2}q^{-(\ell-2j)})}.
\end{align*}
Thus, by \eqref{tildeDdefn},
\begin{equation}
t^{\frac12}D_j^{(\ell)}(Y^{-1})
=t^{-\frac12\ell}\cdot (qt)^j\cdot 
\genfrac[]{0pt}{0}{\ell}{j}_{q,t}\cdot \frac{(1-tq^{\ell-j})}{(1-tq^\ell)}\cdot 
\genfrac(){0pt}{0}{\ell}{\ell-j}_{Y}\cdot 
\frac{(1-t^{-1}Y^{-2}q^{-(\ell-j)})}{(1-t^{-1}Y^{-2}q^{-(\ell-2j)})},
\label{DYinvexp}
\end{equation}
which is equal to 
$\tilde D_j^{(-\ell)}(Y)$ as defined in \eqref{Dnegprod}.

\smallskip\noindent
(b) The first two identities are special cases of \eqref{Dup1} and \eqref{DflippedB}.\hfil\break
Assume $j\in \{1, \ldots, \ell-1\}$.
From \eqref{tildeDdefn}, \eqref{Dup1} and \eqref{DflippedB},
\begin{align*}
t^{\frac12}\frac{\tilde D_j^{(\ell)}(Y)}{\tilde D_j^{(\ell+1)}(Y) }
=
&
\frac{(1-tq^{\ell+1})}{(1-tq^{\ell+1-j})} \cdot
\frac{(1-q^{\ell+1-j})}{(1-q^{\ell+1})}\cdot
\frac{(1-tY^{-2}q^{\ell+1-2j}) }{ (1-tY^{-2}q^{\ell+1-j})}\cdot
\frac{(1-Y^{-2}q^{\ell+1-j})}{(1-t^{-1}Y^{-2}q^{\ell+1-2j})}
\\
\hbox{and}\qquad
t^{-\frac12}\frac{\tilde D_{\ell+1-j}^{(\ell)}(Y^{-1})}{\tilde D_j^{(\ell+1)}(Y)}
&= q^{\ell+1-j}\cdot
\frac{(1-tq^{\ell+1})}{(1-tq^{\ell+1-j})}\cdot \frac{(1-q^j)}{(1-q^{\ell+1})}\cdot 
\frac{(1-Y^{-2}q^{-j})}{(1-tY^{-2}q^{\ell+1-j})}.
\end{align*}
Thus
\begin{align*}
&t^{\frac12}\frac{(1-t^{-1}Y^{-2} q^{\ell+1-2j}) }{(1-Y^{-2} q^{\ell+1-2j})}
\frac{\tilde D_j^{(\ell)}(Y)}{\tilde D_j^{(\ell+1)}(Y)}
+t^{-\frac12} \frac{(1-t)(1-q^{\ell+1}tY^{-2} q^{\ell+1-2j})}{(1- t q^{\ell+1})(1-Y^{-2} q^{\ell+1-2j})}
\frac{\tilde D_{\ell+1-j}^{(\ell)}(Y^{-1})}{\tilde D_j^{(\ell+1)}(Y)}
\\
&\quad=\frac{\cancel{(1-t^{-1}Y^{-2} q^{\ell+1-2j})} }{(1-Y^{-2} q^{\ell+1-2j})}
\frac{(1-tq^{\ell+1})(1-q^{\ell+1-j})}{(1-tq^{\ell+1-j})(1-q^{\ell+1})}
\frac{(1-tY^{-2} q^{\ell+1-2j})}{(1-t Y^{-2} q^{\ell+1-j})}
\frac{(1-Y^{-2} q^{\ell+1-j})}{\cancel{(1-t^{-1}Y^{-2}q^{\ell+1-2j})} }
\\
&\quad\qquad
+
\frac{(1-t)(1-q^{\ell+1}tY^{-2}q^{\ell+1-2j})}{\cancel{(1-t q^{\ell+1})}(1-Y^{-2} q^{\ell+1-2j})}
q^{\ell+1-2j}
\frac{ \cancel{(1-tq^{\ell+1})} (1-q^j) }{(1-tq^{\ell+1-j})(1-q^{\ell+1})}
\frac{(1-Y^{-2} q^{-j})}{(1-t Y^{-2} q^{\ell+1-j})} ,
\end{align*}
which is equal to
$$\frac{\left(\begin{array}{l}
(1-q^{\ell+1-j})(1-tq^{\ell+1}) (1-Y^{-2} q^{\ell+1-j})(1-t Y^{-2} q^{\ell+1-2j}) 
\\
\qquad
+
q^{\ell+1-j}
(1-t)(1-q^j)
(1-t Y^{-2} q^{2(\ell+1-j)})(1-Y^{-2} q^{-j})
\end{array}
\right)}
{ (1-tq^{\ell+1-j})(1-q^{\ell+1})  (1-Y^{-2} q^{\ell+1-2j})(1-t Y^{-2} q^{\ell+1-j}) }
=1,
$$
since the numerator has roots at $tq^{\ell+1-j}=1$, $q^{\ell+1}=1$, $Y^{-2}q^{\ell+1-2j}=1$
and $tY^{-2}q^{\ell+1-j}=1$.
\end{proof}

\subsection{The $\tilde K_j^{(\ell)}(Y)$ in terms of the $\tilde D_j^{(\ell)}(Y)$}

\begin{prop} \label{tildeKfromD}  Let $\tilde K_j^{(\ell)}(Y)$ and $\tilde D_j^{(\ell-1)}(Y)$ be as defined in 
\eqref{tildeKdefn} and \eqref{tildeDdefn}.  Then
$$\tilde K_0^{(\ell)}(Y) = t^{\frac12} \tilde D_0^{(\ell-1)}(Y) \frac{(1-t^{-1}Y^{-2})}{(1-Y^{-2})},
\qquad
\tilde K_\ell^{(\ell)}(Y) = t^{-\frac12}\tilde D_0^{(\ell-1)}(Y^{-1})
\frac{(1-tY^{-2}q^{-\ell})} {(1-t^{-1}Y^{-2}q^{-\ell})}
\frac{(1-t^{-1}Y^{-2}) }{ (1-Y^{-2}) },
$$
and, for $j\in \{1, \ldots, \ell-1\}$,
\begin{align}
\tilde K_j^{(\ell)}(Y) 
&= t^{\frac12} \tilde D_j^{(\ell-1)}(Y)\frac{(1-t^{-1}Y^{-2}) }{ (1-Y^{-2})}
+ t^{-\frac12} \tilde D_{\ell-j}^{(\ell-1)}(Y^{-1}) \frac{(1-tY^{-2} q^{\ell-2j})}{ (1-t^{-1}Y^{-2} q^{\ell-2j})}
\frac{(1-t^{-1}Y^{-2})}{(1-Y^{-2})}.
\label{tildeKrec}
\end{align}
\end{prop}
\begin{proof}  The first two identities are special cases of \eqref{DtoK} and \eqref{DinvtoK}.\hfil\break
Let $j\in \{1, \ldots, \ell-1\}$.
Using \eqref{DinvtoD} gives
\begin{align*}
1&+t^{-1} \frac{\tilde D_{\ell+1-j}^{(\ell)}(Y^{-1}) }{\tilde D_j^{(\ell)}(Y)}\cdot
\frac{(1-tY^{-2} q^{\ell+1-2j}) }{ (1-t^{-1}Y^{-2} q^{\ell+1-2j})}
= 1 + q^{\ell+1-j} \frac{(1-q^j)}{(1-q^{\ell+1-j})}\frac{(1-Y^{-2}q^{-j})}{(1-Y^{-2}q^{\ell+1-j})}
\\
&=
\frac{(1-\cancel{Y^{-2}q^{\ell+1-j} } -\cancel{q^{\ell+1-j}} +Y^{-2}q^{2(\ell+1-j)}
+\cancel{q^{\ell+1-j}}-q^{\ell+1}-Y^{-2}q^{\ell+1-2j}+\cancel{Y^{-2}q^{\ell+1-j} }) }
{(1-q^{\ell+1-j})(1-Y^{-2}q^{\ell+1-j})}
\\
&=
\frac{(1-q^{\ell+1})(1-Y^{-2}q^{\ell+1-2j})}
{(1-q^{\ell+1-j})(1-Y^{-2}q^{\ell+1-j})}.
\end{align*}
Multiplying both sides by $t^{\frac12}\tilde D_j^{(\ell)}(Y)\frac{(1-t^{-1}Y^{-2})}{(1-Y^{-2})}$ gives
\begin{align*}
&t^{\frac12} \tilde D_j^{(\ell)}(Y)\frac{(1-t^{-1}Y^{-2}) }{ (1-Y^{-2})}
+ t^{-\frac12} \tilde D_{\ell+1-j}^{(\ell)}(Y^{-1}) \frac{(1-tY^{-2} q^{\ell+1-2j})}{ (1-t^{-1}Y^{-2} q^{\ell+1-2j})}
\frac{(1-t^{-1}Y^{-2})}{(1-Y^{-2})}
\\
&\qquad
=t^{\frac12} \tilde D_j^{(\ell)}(Y)\cdot 
\frac{(1-q^{\ell+1})(1-Y^{-2}q^{\ell+1-2j})}
{(1-q^{\ell+1-j})(1-Y^{-2}q^{\ell+1-j})}\cdot
\frac{(1-t^{-1}Y^{-2})}{(1-Y^{-2})}
=\tilde K_j^{(\ell+1)},
\end{align*}
where the last equality is \eqref{DtoK}.
\end{proof}

\begin{thm}\label{DandKresult}
As in \eqref{Ddefn} and \eqref{Kdefn}, let $D_j^{(\ell)}(Y)$ and $K_j^{(\ell)}(Y)$ be defined by the 
expansions
$$
E_\ell(X)\mathbf{1}_0
= \sum_{j=0}^{\ell-1} \eta^{\ell-2j} D^{(\ell-1)}_j(Y)\mathbf{1}_0
\qquad\hbox{and}\qquad
E_{-\ell}(X)\mathbf{1}_0
=\sum_{j=0}^\ell \eta^{-\ell+2j} D_j^{(-\ell)}(Y)\mathbf{1}_0
$$
$$
\hbox{and}\qquad
\mathbf{1}_0 E_\ell(X)\mathbf{1}_0 = \sum_{j=0}^\ell \mathbf{1}_0 \eta^{\ell-2j}K_j^{(\ell)}(Y)
$$
in the localized double affine Hecke algebra $\tilde H$.
Then
\begin{align*}
D_j^{(\ell)}(Y) &= t^{-\frac12(\ell+1)}\cdot
t^{\ell-j} \cdot 
\genfrac[]{0pt}{0}{\ell}{j}_{q,t}\cdot 
\genfrac(){0pt}{0}{\ell}{j}_{Y}\cdot 
\frac{(1-tq^{\ell-j})}{(1-tq^\ell)}\cdot 
\frac{(1-tY^{-2}q^{\ell-j})}{(1-tY^{-2}q^{\ell-2j})},
\\
D_j^{(-\ell)}(Y)
&=t^{-\frac12\ell}\cdot (qt)^j\cdot 
\genfrac[]{0pt}{0}{\ell}{j}_{q,t}\cdot 
\genfrac(){0pt}{0}{\ell}{\ell-j}_{Y}\cdot 
\frac{(1-tq^{\ell-j})}{(1-tq^\ell)}\cdot 
\frac{(1-t^{-1}Y^{-2}q^{-(\ell-j)})}{(1-t^{-1}Y^{-2}q^{-(\ell-2j)})}
\quad\hbox{and}
\\
K_j^{(\ell)}(Y) 
&= t^{-\frac12(\ell-1)}\cdot 
t^{\ell-1-j}\cdot
\genfrac[]{0pt}{0}{\ell}{j}_{q,t} \cdot
\genfrac(){0pt}{0}{\ell}{j}_{Y} \cdot
\frac{(1-Y^{-2}q^{\ell-2j})}{(1-t^{-1}Y^{-2}q^{\ell-2j})}\cdot
\frac{(1-t^{-1}Y^{-2})}{(1-Y^{-2})}.
\end{align*}
\end{thm}


\section{Products for type $SL_2$ Macdonald polynomials}

The formulas for $E_\ell(X)\mathbf{1}_0$ and $E_{-\ell}(X)\mathbf{1}_0$ in Theorem \ref{DandKresult}
serve as universal formulas for
products, containing, for all $m$ at once, the information of the products $E_\ell P_m$ and $E_{-\ell}P_m$
expanded in terms of electronic Macdonald polynomials.   In the same way the expansion of 
$\mathbf{1}_0 E_\ell(X)\mathbf{1}_0$ in $\tilde H$ which is given in Theorem \ref{DandKresult} is a universal formula for
the products $P_\ell P_m$ expanded in terms of the bosonic Macdonald polynomials $P_r(x)$.
In this section we use the results of Theorem \ref{DandKresult} to derive these products, thus accomplishing our goal,
in the $SL_2$ case, of
using double affine Hecke algebra tools to compute compact formulas for products of Macdonald polynomials.

\subsection{The universal coefficients $A_j^{(\ell)}(Y)$, $B_j^{(\ell)}(Y)$ and $C_j^{(\ell)}(Y)$}

Define
\begin{align*}
C_j^{(\ell)}(Y) 
&= \genfrac[]{0pt}{0}{\ell}{j}_{q,t} \cdot 
\frac{(t^{-1}Y^{-2}q^{-(j-1)};q)_j (tY^{-2} q^{\ell-2j};q)_j}
{(Y^{-2}q^{\ell-2j+1};q)_j (Y^{-2} q^{-j};q)_j}
\\
A_j^{(\ell)}(Y) &= 
\genfrac[]{0pt}{0}{\ell}{j}_{q,t}\cdot \frac{(1-q^{\ell-j})}{(1-q^{\ell})}\cdot 
\frac{(t^{-1}Y^{-2}q^{-(j-1)};q)_j (tY^{-2}q^{\ell-2j};q)_{j} }{ (Y^{-2}q^{-j};q)_j (Y^{-2}q^{\ell-2j};q)_j}
\\
B_j^{(\ell)}(Y) &= q^j\cdot
\genfrac[]{0pt}{0}{\ell}{j}_{q,t}\cdot \frac{(1-q^{\ell-j})}{(1-q^{\ell})}\cdot 
\frac{(t^{-1}Y^{-2}q^{-(\ell-j-1)};q)_{\ell-j} (tY^{-2}q^{-(\ell-2j-1)};q)_{\ell-j-1} }
{ (Y^{-2}q^{-(\ell-2j-1)};q)_{\ell-j} (Y^{-2}q^{-(\ell-j-1)};q)_{\ell-j-1} }
\end{align*}
Then
\begin{align}
A_j^{(\ell)}(Y) &= 
C_j^{(\ell)}(Y)\cdot 
\frac{(1-q^{\ell-j})}{(1-q^\ell)}\cdot \frac{(1-Y^{-2}q^{\ell-j})}{(1-Y^{-2}q^{\ell-2j})}
\qquad\hbox{and}
\label{CtoA}
\\
B_{\ell-j}^{(\ell)}(Y)
&=q^{\ell-j} \cdot
\genfrac[]{0pt}{0}{\ell}{j}_{q,t} \cdot  \frac{(1-q^{j})}{(1-q^{\ell})}\cdot 
\frac{(t^{-1}Y^{-2}q^{-(j-1)};q)_j (tY^{-2}q^{\ell-2j+1};q)_{j-1}}
{(Y^{-2}q^{\ell-2j+1};q)_j (Y^{-2}q^{-(j-1)};q)_{j-1} }
\nonumber \\
&=C_j^{(\ell)}(Y)\cdot  q^{\ell-j}\cdot \frac{(1-q^{j})}{(1-q^{\ell})}\cdot 
\frac{(1-Y^{-2}q^{-j})}{(1-tY^{-2}q^{\ell-2j})}
\label{CtoB}
\end{align}
The following proposition gives formulas for $A_j^{(\ell)}$, $B_j^{(\ell)}$ and $C_j^{(\ell)}$ in terms of 
$D_j^{(\ell)}$ and $K_j^{(\ell)}$.

\begin{prop} \label{ABCtoDK}
\begin{equation}
t^{-\frac12}C_j^{(\ell)}(Y)
=
K_j^{(\ell)}(Y)\cdot
t^{-\frac12(\ell-2j)}\cdot
\frac{(t^{-1}Y^{-2}q^{\ell-2j};q)_j}{(Y^{-2}q^{\ell-2j};q)_j}
\cdot \frac{(Y^{-2};q)_{\ell-j}}{(t^{-1}Y^{-2};q)_{\ell-j}}.
\label{KtoC}
\end{equation}
\begin{equation}
A_j^{(\ell)}(Y)
=D_j^{(\ell-1)}(Y) \cdot
t^{-\frac12(\ell-2j)}\cdot
\frac{(t^{-1}Y^{-2}q^{\ell-2j};q)_j}{(Y^{-2}q^{\ell-2j};q)_j}
\cdot \frac{(Y^{-2};q)_{\ell-j}}{(t^{-1}Y^{-2};q)_{\ell-j}} \cdot
t \frac{(1-t^{-1}Y^{-2})}{(1-Y^{-2})} 
\label{DtoA}
\end{equation}
\begin{equation}
B_j^{(\ell)}(Y)
=
D_j^{(\ell-1)}(Y^{-1})
t^{\frac12(\ell-2j)}\cdot
\frac{(t^{-1}Y^{-2}q^{-(\ell-2j)+1};q)_{\ell-j} }{(Y^{-2}q^{-(\ell-2j)+1};q)_{\ell-j} }\cdot
\frac{(Y^{-2}q;q)_j }{(t^{-1}Y^{-2}q;q)_j}
\label{DinvtoB}
\end{equation}
\end{prop}  
\begin{proof}
\smallskip\noindent
(a) Using \eqref{tildeKdefn} gives
\begin{align*}
&K_j^{(\ell)}(Y)\cdot
t^{-\frac12(\ell-2j)}\cdot
\frac{(t^{-1}Y^{-2}q^{\ell-2j};q)_j}{(Y^{-2}q^{\ell-2j};q)_j}
\cdot \frac{(Y^{-2};q)_{\ell-j}}{(t^{-1}Y^{-2};q)_{\ell-j}}
\\
&=t^{-\frac12(\ell-1)}\cdot t^{\ell-1-j}\cdot \genfrac[]{0pt}{0}{\ell}{j}_{q,t} \cdot 
\genfrac(){0pt}{0}{\ell}{j}_{Y} \cdot 
\frac{(1-Y^{-2}q^{\ell-2j})}{(1-t^{-1}Y^{-2}q^{\ell-2j})}\cdot \frac{(1-t^{-1}Y^{-2})}{(1-Y^{-2})}
\\
&\qquad
\cdot t^{-\frac12(\ell-2j)}\cdot 
\frac{(t^{-1}Y^{-2}q^{\ell-2j};q)_j}{(Y^{-2}q^{\ell-2j};q)_j}
\cdot \frac{(Y^{-2};q)_{\ell-j}}{(t^{-1}Y^{-2};q)_{\ell-j}}
\\
&=t^{-\frac12} \genfrac[]{0pt}{0}{\ell}{j}_{q,t} \cdot 
\frac{(t^{-1}Y^{-2} q^{-(j-1)};q)_{\ell-j} (tY^{-2}q^{\ell-2j};q)_j}{  (Y^{-2} q;q)_{\ell-j} (Y^{-2}q^{-j};q)_j }
\cdot
\frac{(t^{-1}Y^{-2}q^{\ell-2j};q)_j}{(Y^{-2}q^{\ell-2j};q)_j}
\cdot \frac{(Y^{-2};q)_{\ell-j}}{(t^{-1}Y^{-2};q)_{\ell-j}}
\\
&\qquad\cdot
\frac{(1-Y^{-2}q^{\ell-2j})}{(1-t^{-1}Y^{-2}q^{\ell-2j})}\cdot \frac{(1-t^{-1}Y^{-2})}{(1-Y^{-2})}
\\
&=t^{-\frac12} \genfrac[]{0pt}{0}{\ell}{j}_{q,t} \cdot 
\frac{(t^{-1}Y^{-2}q^{-(j-1)};q)_{\ell-j} (1-t^{-1}Y^{-2}) (t^{-1}Y^{-2}q^{\ell-2j};q)_j}{(t^{-1}Y^{-2};q)_{\ell-j} (1-t^{-1}Y^{-2}q^{\ell-2j})} \\
&\qquad\qquad \cdot
\frac{ (Y^{-2};q)_{\ell-j}}{(1-Y^{-2})(Y^{-2}q;q)_{\ell-j} } \cdot
\frac{ (1-Y^{-2}q^{\ell-2j}) }{  (Y^{-2}q^{\ell-2j};q)_j }\cdot \frac{1}{(Y^{-2}q^{-j};q)_j } \cdot  (tY^{-2}q^{\ell-2j};q)_j
\\
&=t^{-\frac12} \genfrac[]{0pt}{0}{\ell}{j}_{q,t} \cdot 
\frac{(t^{-1}Y^{-2}q^{-(j-1)};q)_{\ell-j-1}  (t^{-1}Y^{-2}q^{\ell-2j};q)_j} {(t^{-1}Y^{-2}q;q)_{\ell-j-1} } \\
&\qquad\qquad \cdot
\frac{1}{(1-Y^{-2}q^{\ell-j})} \cdot
\frac{ 1 }{  (Y^{-2}q^{\ell-2j+1};q)_{j-1} }\cdot \frac{1}{(Y^{-2}q^{-j};q)_j } \cdot  (tY^{-2}q^{\ell-2j};q)_j
\\
&=t^{-\frac12}\genfrac[]{0pt}{0}{\ell}{j}_{q,t} \cdot 
\frac{(t^{-1}Y^{-2}q^{-(j-1)};q)_j (tY^{-2} q^{\ell-2j};q)_j}
{(Y^{-2}q^{\ell-2j+1};q)_j (Y^{-2} q^{-j};q)_j}
= t^{-\frac12} C_j^{(\ell)}.
\end{align*}
(b) Using \eqref{DtoK} and \eqref{KtoC} and \eqref{CtoA}, gives
\begin{align*}
D_j^{(\ell-1)}(Y)\cdot &t^{-\frac12(\ell-2j)}\cdot
\frac{(t^{-1}Y^{-2}q^{\ell-2j};q)_j}{(Y^{-2}q^{\ell-2j};q)_j}
\cdot \frac{(Y^{-2};q)_{\ell-j}}{(t^{-1}Y^{-2};q)_{\ell-j}} \cdot
t \frac{(1-t^{-1}Y^{-2})}{(1-Y^{-2})} 
\\
&= t^{-\frac12} K_j^{(\ell)}(Y)\cdot t^{-\frac12(\ell-2j)}\cdot \frac{(1-q^{\ell-j})}{(1-q^\ell)}\cdot \frac{(1-Y^{-2}q^{\ell-j})}{(1-Y^{-2}q^{\ell-2j})}\cdot
\frac{\cancel{(1-Y^{-2})} }{\cancel{(1-t^{-1}Y^{-2})} }
\\
&\qquad\cdot
\frac{(t^{-1}Y^{-2}q^{\ell-2j};q)_j}{(Y^{-2}q^{\ell-2j};q)_j}
\cdot \frac{(Y^{-2};q)_{\ell-j}}{(t^{-1}Y^{-2};q)_{\ell-j}} \cdot
t \cdot \frac{\cancel{(1-t^{-1}Y^{-2})} }{\cancel{(1-Y^{-2})} } 
\\
&= t^{\frac12}\cdot t^{-\frac12} C_j^{(\ell)}(Y)\cdot \frac{(1-q^{\ell-j})}{(1-q^\ell)}\cdot \frac{(1-Y^{-2}q^{\ell-j})}{(1-Y^{-2}q^{\ell-2j})}
=A_j^{(\ell)}(Y).
\end{align*}
(c) Using \eqref{DinvtoK} and \eqref{KtoC} and \eqref{CtoB} gives
\begin{align*}
&D_{\ell-j}^{(\ell-1)}(Y^{-1})\cdot 
t^{-\frac12(\ell-2j)}\cdot
\frac{(t^{-1}Y^{-2}q^{(\ell-2j)+1};q)_{j} }{(Y^{-2}q^{(\ell-2j)+1};q)_{j} }\cdot
\frac{(Y^{-2}q;q)_{\ell-j} }{(t^{-1}Y^{-2}q;q)_{\ell-j}}
\\
&=K_j^{(\ell)}(Y)\cdot t^{-\frac12(\ell-2j)}\cdot
\frac{(t^{-1}Y^{-2}q^{(\ell-2j)+1};q)_{j} }{(Y^{-2}q^{(\ell-2j)+1};q)_{j} }\cdot
\frac{(Y^{-2}q;q)_{\ell-j} }{(t^{-1}Y^{-2}q;q)_{\ell-j}}
\cdot \frac{(1-t^{-1}Y^{-2}q^{\ell-2j})}{(1-t^{-1}Y^{-2})}
\cdot \frac{(1-Y^{-2})}{(1-Y^{-2}q^{\ell-2j})}
\\
&\qquad
\cdot  t^{\frac12} q^{\ell-j}\cdot
\frac{(1-q^{j})}{(1-q^{\ell})}\cdot 
\frac{(1-Y^{-2}q^{-j})}{(1-tY^{-2}q^{\ell-2j})}
\\
&=K_j^{(\ell)}(Y)\cdot t^{-\frac12(\ell-2j)}\cdot
\frac{(t^{-1}Y^{-2}q^{(\ell-2j)};q)_{j} }{(Y^{-2}q^{(\ell-2j)};q)_{j} }\cdot
\frac{(Y^{-2};q)_{\ell-j} }{(t^{-1}Y^{-2};q)_{\ell-j}}
\cdot \frac{(1-t^{-1}Y^{-2}q^{\ell-j})}{(1-t^{-1}Y^{-2}q^{\ell-j})}
\cdot \frac{(1-Y^{-2}q^{\ell-j})}{(1-Y^{-2}q^{\ell-j})}
\\
&\qquad
\cdot t^{\frac12} q^{\ell-j}\cdot
\frac{(1-q^{j})}{(1-q^{\ell})}\cdot 
\frac{(1-Y^{-2}q^{-j})}{(1-tY^{-2}q^{\ell-2j})}
\\
&=t^{-\frac12} C_j^{(\ell)}(Y)\cdot  t^{\frac12} q^{\ell-j}\cdot \frac{(1-q^{j})}{(1-q^{\ell})}\cdot 
\frac{(1-Y^{-2}q^{-j})}{(1-tY^{-2}q^{\ell-2j})}
= B_{\ell-j}^{(\ell)}(Y).
\end{align*}
Then, replacing $j$ with $\ell-j$ gives
$$B_j^{(\ell)}(Y) = 
D_j^{(\ell-1)}(Y^{-1})\cdot 
t^{\frac12(\ell-2j)}\cdot
\frac{(t^{-1}Y^{-2}q^{-(\ell-2j)+1};q)_{\ell-j} }{(Y^{-2}q^{-(\ell-2j)+1};q)_{\ell-j} }\cdot
\frac{(Y^{-2}q;q)_j }{(t^{-1}Y^{-2}q;q)_j }.
$$
\end{proof}

\subsection{Product formulas for type $SL_2$ Macdonald polynomials}

The following theorem provides formulas for the products $E_\ell(x) E_m(x)$, $E_{-\ell}(x) P_m(x)$
and $P_\ell(x)P_m(x)$ expanded in terms of Macdonald polynomials.  It is useful to note that, in 
Theorem \ref{finalthm},
\begin{align*}
\ev_m(A_j^{(\ell)}(Y)) &= 0 \ \ \hbox{if $m+\ell-2j<0$,} \quad
&&
\ev_m(t B_j^{(\ell+1)}(Y)) = 0 \ \ \hbox{if $m-(\ell-2j)<0$,} 
\\
\ev_m(B_j^{(\ell)}(Y)) &= 0 \ \ \hbox{if $-m+\ell-2j>0$,} \quad
&&
\ev_m(A_j^{(\ell+1)}(Y)) = 0 \ \ \hbox{if $-m-(\ell-2j)>0$,} 
\end{align*}
and
$$
\ev_m(C_j^{(\ell)}(Y)) = 0 \ \ \hbox{if $m+\ell-2j<0$,} 
$$
since a factor in the numerator of each of these expressions evaluates to $(1-1)=0$.

\begin{thm} \label{finalthm}
Let $\ell, m\in \ZZ_{>0}$.
Let $\ev_m\colon \CC[Y,Y^{-1}]\to \CC$ be the homomorphism given by $\ev_m(Y) = t^{-\frac12} q^{-\frac12 m}$
and extend $\ev_m$ to elements of $\CC(Y)$ such that the denominator does not evaluate to $0$.  Then
$$
P_\ell(x) P_m(x) = \sum_{j=0}^{\ell} \ev_m(C_j^{(\ell)}(Y)) P_{m+\ell-2j}(x),$$
$$
E_\ell(x) P_m(x) 
= \sum_{j=0}^{\ell-1} \ev_m(A_j^{(\ell)}(Y)) E_{m+\ell-2j}(x)
+ \sum_{j=0}^{\ell-1} \ev_m(B_j^{(\ell)}(Y)) E_{-m+\ell-2j}(x)
\qquad\hbox{and}
$$
$$
E_{-\ell}(x) P_m(x) 
= \sum_{j=0}^{\ell} \ev_m(t B_j^{(\ell+1)}(Y)) E_{m-(\ell-2j)}(x)
+ \sum_{j=0}^{\ell} \ev_m(A_j^{(\ell+1)}(Y)) E_{-m-(\ell-2j)}(x).
$$
\end{thm}
\begin{proof} 
By \eqref{Eeig},
if $f\in \CC(Y)$ such that $\ev_m(f(Y))$ is defined then $f(Y)E_m(X)\mathbf{1}_Y = \ev_m(f(Y)E_m(X)\mathbf{1}_Y$.

\smallskip\noindent
(a) By \eqref{creationP}, \eqref{smashE}, \eqref{Kdefn} and \eqref{Eeig},
\begin{align*}
P_\ell(X)P_m(X)\mathbf{1}_Y 
&= P_\ell(X) t^{\frac12} \mathbf{1}_0 E_m(X)\mathbf{1}_Y
= t \mathbf{1}_0 E_\ell(X) \mathbf{1}_0 E_m(X) \mathbf{1}_Y
\\
&
=t \Big(\sum_{j=0}^\ell \mathbf{1}_0\eta^{\ell-2j} K_j^{(\ell)}(Y)  \Big) E_m(X) \mathbf{1}_Y
\end{align*}
and, using \eqref{normp},
\begin{align*}
\mathbf{1}_0 &\eta^{\ell-2j} t K_j^{(\ell)}(Y)E_m(X)\mathbf{1}_Y
= t\cdot \ev_m(K_j^{(\ell)}(Y)) \mathbf{1}_0 \eta^{-j}\eta^{\ell-j} E_m(X)\mathbf{1}_Y
\\
&= t\cdot \ev_m\Big( K_j^{(\ell)}(Y) t^{-\frac12(\ell-2j)} 
\frac{(t^{-1}Y^{-2}q^{\ell-2j};q)_j}{(Y^{-2}q^{\ell-2j};q)_j}\cdot
\frac{(Y^{-2};q)_{\ell-j} }{(t^{-1}Y^{-2};q)_{\ell-j} }\Big) 
\mathbf{1}_0 E_{m+\ell-2j}(X)\mathbf{1}_Y
\\
&= t\cdot \ev_m\Big( K_j^{(\ell)}(Y) t^{-\frac12(\ell-2j)} 
\frac{(t^{-1}Y^{-2}q^{\ell-2j};q)_j}{(Y^{-2}q^{\ell-2j};q)_j}\cdot
\frac{(Y^{-2};q)_{\ell-j} }{(t^{-1}Y^{-2};q)_{\ell-j} }\Big) 
t^{-\frac12} P_{m+\ell-2j}(X)\mathbf{1}_Y
\\
&= t^{\frac12}\ev_m(t^{-\frac12}C_j^{(\ell)}(Y)) P_{m+\ell-2j}(X)\mathbf{1}_Y,
\end{align*}
where the last equality is \eqref{KtoC}.

\smallskip\noindent
(b)
By \eqref{creationP}, \eqref{Ddefn} and \eqref{Pdefn},
\begin{align*}
E_\ell(X)P_m(X)\mathbf{1}_Y 
&= E_\ell(X)t^{\frac12} \mathbf{1}_0 E_m(X) \mathbf{1}_Y 
=t^{\frac12}\Big(\sum_{j=0}^{\ell-1} \eta^{\ell-2j} D_j^{(\ell-1)}(Y) \Big) \mathbf{1}_0E_m(X) \mathbf{1}_Y
\\
&=t^{\frac12}\Big(\sum_{j=0}^{\ell-1} \eta^{\ell-2j} D_j^{(\ell-1)}(Y) \Big) t^{-\frac12} P_m(X) \mathbf{1}_Y
\\
&=\Big(\sum_{j=0}^{\ell-1} \eta^{\ell-2j} D_j^{(\ell-1)}(Y) \Big)
\Big( \frac{t(1-q^m)}{(1-tq^m)} E_m(X) + E_{-m}(X)\Big)\mathbf{1}_Y.
\end{align*}
Using \eqref{normp},
\begin{align*}
\eta^{\ell-2j}&D_j^{(\ell-1)}(Y) t\frac{(1-q^m)}{(1-tq^m)}E_m(X)\mathbf{1}_Y
=\ev_m\big(D_j^{(\ell-1)}(Y)\big) t\frac{(1-q^m)}{(1-tq^m)} \eta^{-j}\eta^{\ell-j} E_m(X)\mathbf{1}_Y
\\
&=\ev_m\Big(
D_j^{(\ell-1)}(Y) t \frac{(1-t^{-1}Y^{-2})}{(1-Y^{-2})} 
t^{-\frac12(\ell-2j)}
\frac{(t^{-1}Y^{-2}q^{\ell-2j};q)_j}{(Y^{-2}q^{\ell-2j};q)_j}\cdot
\frac{(Y^{-2};q)_{\ell-j} }{(t^{-1}Y^{-2};q)_{\ell-j} }
\Big)
E_{m+\ell-2j}(X)\mathbf{1}_Y
\\
&= \ev_m( A_j^{(\ell)}(Y)) E_{m+\ell-2j}(X)\mathbf{1}_Y,
\end{align*}
where the last equality follows from \eqref{DtoA}.
Using \eqref{Eeig} and \eqref{normn},
\begin{align*}
\eta^{\ell-2j}&D_j^{(\ell-1)}(Y) E_{-m}(X)\mathbf{1}_Y
= \ev_m(D_j^{(\ell-1)}(Y^{-1})) \eta^{\ell-j}\eta^{-j}E_{-m}(X)\mathbf{1}_Y
\\
&= \ev_m\Big( D_j^{(\ell-1)}(Y^{-1}) t^{\frac12(\ell-2j)} 
\frac{(t^{-1}Y^{-2}q^{-(\ell-2j)+1};q)_{\ell-j} }{(Y^{-2}q^{-(\ell-2j)+1};q)_{\ell-j} }
\frac{(Y^{-2}q;q)_j }{(t^{-1}Y^{-2}q;q)_j} \Big) E_{-m+\ell-2j}(X)\mathbf{1}_Y
\\
&= \ev_m(B_j^{(\ell)}(Y)) E_{-m+\ell-2j}(X)\mathbf{1}_Y,
\end{align*}
where the last equality is \eqref{DinvtoB}.

\smallskip\noindent
(c)  By \eqref{creationP}, \eqref{Ddefn}, \eqref{Pdefn} and Proposition \ref{Drecursion}(a),
\begin{align*}
E_{-\ell}(X)P_m(X)\mathbf{1}_Y 
&= E_{-\ell}(X)t^{\frac12} \mathbf{1}_0 E_m(X) \mathbf{1}_Y 
=t^{\frac12}\Big(\sum_{j=0}^{\ell} \eta^{-(\ell-2j)} D_j^{(-\ell)}(Y) \Big) \mathbf{1}_0E_m(X) \mathbf{1}_Y
\\
&=t^{\frac12}\Big(\sum_{j=0}^{\ell} \eta^{-(\ell-2j)} D_j^{(-\ell)}(Y) \Big) t^{-\frac12} P_m(X) \mathbf{1}_Y
\\
&=\Big(\sum_{j=0}^{\ell} \eta^{-(\ell-2j)} D_j^{(-\ell)}(Y) \Big)
\Big( \frac{t(1-q^m)}{(1-tq^m)} E_m(X) + E_{-m}(X)\Big)\mathbf{1}_Y.
\\
&=\Big(\sum_{j=0}^{\ell} \eta^{-(\ell-2j)} t^{\frac12}D_j^{(\ell)}(Y^{-1}) \Big)
\Big( \frac{t(1-q^m)}{(1-tq^m)} E_m(X) + E_{-m}(X)\Big)\mathbf{1}_Y.
\end{align*}
Using 
\begin{align*}
\frac{(1-t^{-1}Y^{-2})}{(1-Y^{-2})} 
&\frac{(t^{-1}Y^{-2}q^{-(\ell-2j)};q)_{\ell-j}}{(Y^{-2}q^{-(\ell-2j)};q)_{\ell-j} }\cdot
\frac{(Y^{-2};q)_j }{(t^{-1}Y^{-2};q)_j }
\\
&=\frac{(t^{-1}Y^{-2}q^{-(\ell-2j)};q)_{\ell+1-j}}{(Y^{-2}q^{-(\ell-2j)};q)_{\ell+1-j} }\cdot
\frac{(1-Y^{-2}q^{j})}{(1-t^{-1}Y^{-2}q^{j})}\cdot
\frac{(Y^{-2}q;q)_{j-1} }{(t^{-1}Y^{-2}q;q)_{j-1} }
\\
&= \frac{(t^{-1}Y^{-2}q^{-(\ell-2j)};q)_{\ell+1-j}}{(Y^{-2}q^{-(\ell-2j)};q)_{\ell+1-j} }\cdot
\frac{(Y^{-2}q;q)_j }{(t^{-1}Y^{-2}q;q)_j }
\end{align*}
and \eqref{normpB}gives
\begin{align*}
&\eta^{-(\ell-2j)} t^{\frac12} D_j^{(\ell)}(Y^{-1}) t\frac{(1-q^m)}{(1-tq^m)}E_m(X)\mathbf{1}_Y
= \ev_m\Big(t^{\frac12}D_j^{(\ell)}(Y^{-1}) t\frac{(1-t^{-1}Y^{-2})}{(1-Y^{-2})} \Big) 
\eta^{-(\ell-j)}\eta^jE_m(X)\mathbf{1}_Y
\\
&=\ev_m\Big(
t^{\frac12} D_j^{(\ell)}(Y^{-1}) t \frac{(1-t^{-1}Y^{-2})}{(1-Y^{-2})} 
t^{\frac12(\ell-2j)}
\frac{(t^{-1}Y^{-2}q^{-(\ell-2j)};q)_{\ell-j}}{(Y^{-2}q^{-(\ell-2j)};q)_{\ell-j} }\cdot
\frac{(Y^{-2};q)_j }{(t^{-1}Y^{-2};q)_j }
\Big)
E_{m-(\ell-2j)}(X)\mathbf{1}_Y
\\
&=\ev_m\Big(
t D_j^{(\ell)}(Y^{-1}) \cdot t^{\frac12(\ell+1-2j)}\cdot
\frac{(t^{-1}Y^{-2}q^{-(\ell-2j)};q)_{\ell+1-j}}{(Y^{-2}q^{-(\ell-2j)};q)_{\ell+1-j} }\cdot
\frac{(Y^{-2}q;q)_j }{(t^{-1}Y^{-2}q;q)_j }
\Big)
E_{m-(\ell-2j)}(X)\mathbf{1}_Y
\\
&= \ev_m( t B^{(\ell+1)}_j(Y)) E_{m-(\ell-2j)}(X)\mathbf{1}_Y,
\end{align*}
where the last equality is \eqref{DinvtoB}.
Using \eqref{Eeig} and \eqref{normnB} gives
\begin{align*}
&\eta^{-(\ell-2j)}t^{\frac12} D_j^{(\ell)}(Y^{-1}) E_{-m}(X)\mathbf{1}_Y
= \ev_m(t^{\frac12} D_j^{(\ell)}(Y)) \eta^{j}\eta^{-(\ell-j)}E_{-m}(X)\mathbf{1}_Y
\\
&= \ev_m\Big( t^{\frac12} D_j^{(\ell)}(Y) t^{-\frac12(\ell-2j)} 
\frac{(t^{-1}Y^{-2}q^{\ell-2j+1};q)_j }{(Y^{-2}q^{\ell-2j+1};q)_j }
\frac{(Y^{-2}q;q)_{\ell-j} }{(t^{-1}Y^{-2}q;q)_{\ell-j} } \Big) E_{-m-(\ell-2j)}(X)\mathbf{1}_Y
\\
&= \ev_m\Big( D_j^{(\ell)}(Y) t^{-\frac12(\ell+1-2j)} 
\frac{(t^{-1}Y^{-2}q^{\ell-2j+1};q)_j }{(Y^{-2}q^{\ell-2j+1};q)_j }
\frac{(Y^{-2};q)_{\ell+1-j} }{(t^{-1}Y^{-2};q)_{\ell+1-j} } \cdot
t\frac{(1-t^{-1}Y^{-2})}{(1-Y^{-2})}
 \Big) E_{-m-(\ell-2j)}(X)\mathbf{1}_Y
\\
&= \ev(A_j^{(\ell+1)}(Y)) E_{-m-(\ell-2j)}(X)\mathbf{1}_Y,
\end{align*}
where the last equality is \eqref{DtoA}.
\end{proof}

\begin{remark}
After replacing $Y^{-2}$ by $X_1X_2^{-1}$ the expression for $C_j^{(\ell)}(Y)$
coincides with
the expression for the Macdonald Littlewood-Richardson coefficient given in \cite[Theorem 1.4]{MW23}.
\end{remark}

\newpage

\section{Examples}

For $j\in \ZZ_{>0}$ and $a,b\in \ZZ$ with $a\le b$ define
\begin{align*}
(z;q)_j &= (1-z)(1-qz)(1-q^2z)\cdots (1-q^{j-1}z)
\qquad\hbox{and} \\
(1-zq^{a..b}) &= (1-zq^a)(1-zq^{a+1})\cdots (1-zq^{b-1})(1-zq^b)
\quad\hbox{so that}\quad
(1-zq^{a..b}) = (1-zq^a)_{b-a+1}.
\end{align*}

\subsection{Examples of the $q$-$t$-binomial coefficients}

Let
$$\genfrac[]{0pt}{0}{k}{j}_{q,t}  
= \frac{ \frac{(q;q)_k}{(t;q)_k} } { \frac{(q;q)_j}{(t;q)_j} \frac{(q;q)_{k-j}}{(t;q)_{k-j}} }
= \frac{(1-q^{j+1..k})}{(1-q^{1..k-j})}\frac{(1-tq^{0..k-j-1})}{(1-tq^{j..k-1})}.$$
Then
$$
\begin{array}{c}
\genfrac[]{0pt}{0}{0}{0}_{q,t} =  1, 
\\
\\
\genfrac[]{0pt}{0}{1}{0}_{q,t} = 1,
\qquad
\genfrac[]{0pt}{0}{1}{1}_{q,t} =1,
\\
\\
\genfrac[]{0pt}{0}{2}{0}_{q,t} =1
\qquad
\displaystyle{\genfrac[]{0pt}{0}{2}{1}_{q,t} = \frac{(1-q^2)(1-t)}{(1-q)(1-tq)},}
\qquad
\genfrac[]{0pt}{0}{2}{2}_{q,t} = 1,
\\
\\
\genfrac[]{0pt}{0}{3}{0}_{q,t} =1,
\qquad 
\displaystyle{\genfrac[]{0pt}{0}{3}{1}_{q,t} =
\frac{(1-t)(1-q^3)}{(1-q)(1-tq^2)},}
\qquad 
\displaystyle{\genfrac[]{0pt}{0}{3}{2}_{q,t} =
\frac{(1-t)(1-q^3)}{(1-q)(1-tq^2)},}
\qquad \genfrac[]{0pt}{0}{3}{3}_{q,t} =1
\end{array}
$$

\subsection{Examples of the shifted $q$-$t$-binomial coefficients}

Let
$$
\genfrac\{\}{0pt}{0}{k}{j}_{q,t}  
= \genfrac[]{0pt}{0}{k}{j}_{q,t}  \frac{(1-tq^{k-j})}{(1-tq^k)}
= \frac{(1-q^{j+1..k})}{(1-q^{1..k-j})}\frac{(1-tq^{1..k-j-1})}{(1-tq^{j..k})}.$$
Then
$$
\begin{array}{c}
\genfrac\{\}{0pt}{0}{0}{0}_{q,t} =  1, 
\\
\\
\genfrac\{\}{0pt}{0}{1}{0}_{q,t} = 1,
\qquad
\displaystyle{
\genfrac\{\}{0pt}{0}{1}{1}_{q,t} =\frac{(1-t)}{(1-tq)},
}
\\
\\
\genfrac\{\}{0pt}{0}{2}{0}_{q,t} =1
\qquad
\displaystyle{\genfrac\{\}{0pt}{0}{2}{1}_{q,t} = \frac{(1-q^2)(1-t)}{(1-q)(1-tq^2)},}
\qquad
\genfrac\{\}{0pt}{0}{2}{2}_{q,t} = \frac{(1-t)}{(1-tq^2)},
\\
\\
\genfrac\{\}{0pt}{0}{3}{0}_{q,t} =1,
\qquad 
\displaystyle{\genfrac\{\}{0pt}{0}{3}{1}_{q,t} =
\frac{(1-t)(1-q^3)}{(1-q)(1-tq^3)},}
\qquad 
\displaystyle{\genfrac\{\}{0pt}{0}{3}{2}_{q,t} =
\frac{(1-t)(1-tq)}{(1-q)(1-tq^2)},}
\qquad \genfrac\{\}{0pt}{0}{3}{3}_{q,t} = \frac{(1-t)}{(1-tq^3)}.
\end{array}
$$

\newpage

\subsection{Examples of the $Y$-binomial coefficients}\label{Ybinexamples}

\begin{align*}
\genfrac(){0pt}{0}{\ell}{j}_{Y} 
&= 
\frac{(1-t^{-1}Y^{-2} q^{-(j-1)..\ell-2j}) (1-tY^{-2} q^{\ell-2j..\ell-j-1})}
{(1-Y^{-2}q^{1..\ell-j}) (1-Y^{-2}q^{-j..-1})}
=
\frac{(t^{-1}Y^{-2} q^{-(j-1)};q)_{\ell-j} (tY^{-2}q^{\ell-2j};q)_j}{  (Y^{-2} q;q)_{\ell-j} (Y^{-2}q^{-j};q)_j }.
\end{align*}

Then
\begin{align*}
\genfrac(){0pt}{0}{0}{0}_{Y} &=  1, 
\\
\\
\genfrac(){0pt}{0}{1}{0}_{Y} &= \frac{(1-t^{-1}Y^{-2}q)}{(1-Y^{-2}q)} 
= t^{-1}\cdot \frac{(1-tY^2q^{-1})}{(1-Y^2q^{-1})}
\\
\genfrac(){0pt}{0}{1}{1}_{Y} &=\frac{(1-tY^{-2}q^{-1})}{(1-Y^{-2}q^{-1})} 
= t\cdot \frac{(1-t^{-1}Y^2q)}{(1-Y^2q)}
\\
\\
\genfrac(){0pt}{0}{2}{0}_{Y} 
&=\frac{(1-t^{-1}Y^{-2}q)(1-t^{-1}Y^{-2}q^2)}{(1-Y^{-2}q)(1-Y^{-2}q^2)}
=t^{-2}\cdot \frac{(1-tY^{2}q^{-1})(1-tY^{2}q^{-2})}{(1-Y^{2}q^{-1})(1-Y^{2}q^{-2})}
\\
\genfrac(){0pt}{0}{2}{1}_{Y} 
&= \frac{(1-t^{-1}Y^{-2})}{(1-Y^{-2}q)}\cdot \frac{(1-tY^{-2})}{(1-Y^{-2}q^{-1})} 
= \frac{(1-tY^{2})}{(1-Y^{2}q^{-1})}\cdot \frac{(1-t^{-1}Y^{2})}{(1-Y^{2}q)} 
\\
\genfrac(){0pt}{0}{2}{2}_{Y} 
&= \frac{(1-tY^{-2}q^{-2})(1-tY^{-2}q^{-1})}{(1-Y^{-2}q^{-2})(1-Y^{-2}q^{-1})}
= t^2\cdot \frac{(1-t^{-1}Y^{2}q^{2})(1-t^{-1}Y^{2}q)}{(1-Y^{2}q^{2})(1-Y^{2}q)}
\\
\\
\genfrac(){0pt}{0}{3}{0}_{Y} 
&=\frac{(1-t^{-1}Y^{-2}q)(1-t^{-1}Y^{-2}q^2)(1-t^{-1}Y^{-2}q^3)}{(1-Y^{-2}q)(1-Y^{-2}q^2)(1-Y^{-2}q^3)}
\\
\genfrac(){0pt}{0}{3}{1}_{Y} 
&=
\frac{(1-t^{-1}Y^{-2})(1-t^{-1}Y^{-2}q)}{(1-Y^{-2}q)(1-Y^{-2}q^2)}
\cdot \frac{(1-tY^{-2}q)}{(1-Y^{-2}q^{-1})} 
\\
\genfrac(){0pt}{0}{3}{2}_{Y} 
&=
\frac{(1-t^{-1}Y^{-2}q^{-1})}{(1-Y^{-2}q)}
\cdot \frac{(1-tY^{-2}q^{-1})(1-tY^{-2})}{(1-Y^{-2}q^{-2})(1-Y^{-2}q^{-1})} 
\\
\genfrac(){0pt}{0}{3}{3}_{Y} 
&=\frac{(1-tY^{-2}q^{-3})(1-tY^{-2}q^{-2})(1-tY^{-2}q^{-1})}
{(1-Y^{-2}q^{-3})(1-Y^{-2}q^{-2})(1-Y^{-2}q^{-1})}
\end{align*}

\newpage

\subsection{Examples of the $D_j^{(\ell-1)}(Y)$}\label{DKSL2Ex}

The general product formula for the $D_j^{(\ell)}(Y)$ is
\begin{align*}
D_j^{(\ell)}(Y) &= t^{-\frac12(\ell+1)}\cdot
t^{\ell-j} \cdot 
\genfrac[]{0pt}{0}{\ell}{j}_{q,t}\cdot \frac{(1-tq^{\ell-j})}{(1-tq^\ell)}\cdot 
\genfrac(){0pt}{0}{\ell}{j}_{Y}\cdot \frac{(1-tY^{-2}q^{\ell-j})}{(1-tY^{-2}q^{\ell-2j})} 
\end{align*}
The first few of the $D_j^{(\ell-1)}(Y)$ are
\begin{align*}
D_0^{(0)}(Y) &= t^{-\frac12}\cdot 1, \\
\\
D_0^{(1)}(Y) &= t^{-\frac22}\cdot
\frac{(1-tY^2q^{-1})}{(1-Y^2q^{-1})} = t^{-\frac22}\cdot t\cdot \frac{(1-t^{-1}Y^{-2}q)}{(1-Y^{-2}q)}
\\
D_1^{(1)}(Y) 
&= t^{-\frac22}\cdot qt\cdot  
\frac{(1-t)}{(1-tq)}
\cdot \frac{(1-t^{-1}Y^2)}{(1-Y^2q)}
= 
t^{-\frac22}\cdot \frac{(1-t)}{(1-tq)}
\cdot  \frac{(1-tY^{-2})}{(1-Y^{-2}q^{-1})}, \\
\\
D_0^{(2)}(Y) 
&= t^{-\frac32}\cdot \frac{(1-tY^2q^{-2})(1-tY^2q^{-1})}{(1-Y^2q^{-2})(1-Y^2q^{-1})}
= t^{-\frac32}\cdot t^2\cdot \frac{ (1-t^{-1}Y^{-2}q)(1-t^{-1}Y^{-2}q^2) }{ (1-Y^{-2}q)(1-Y^{-2}q^2) }
\\
D_1^{(2)}(Y) 
&=  t^{-\frac32}\cdot qt\cdot 
\frac{(1-q^2)(1-t)}{(1-q)(1-tq^2)}
\cdot \frac{ (1-tY^2) }{ (1-Y^2q^{-1}) }\cdot \frac{(1-t^{-1}Y^2q^{-1})}{(1-Y^2q)}, \\
&=  t^{-\frac32}\cdot t\cdot 
\frac{(1-q^2)(1-t)}{(1-q)(1-tq^2)}
\cdot \frac{ (1-t^{-1}Y^{-2}) }{ (1-Y^{-2}q) }\cdot \frac{(1-tY^{-2}q)}{(1-Y^{-2}q^{-1})}
\\
D_2^{(2)}(Y) 
&= t^{-\frac32}\cdot q^2t^2\cdot 
\frac{(1-t)}{(1-tq^2)}
\cdot  \frac{(1-t^{-1}Y^2)(1-t^{-1}Y^2q)}{(1-Y^2q)(1-Y^2q^2)} \\
&= 
t^{-\frac32}\cdot \frac{(1-t)}{(1-tq^2)}
\cdot \frac{ (1-tY^{-2}q^{-1})(1-tY^{-2}) }{ (1-Y^{-2}q^{-2})(1-Y^{-2}q^{-1}) }.
\end{align*}

\newpage

\subsection{Examples of the $K_j^{(\ell)}(Y)$}

The general product formula for the $K_j^{(\ell)}(Y)$ for $\ell\in \ZZ_{>0}$ is
\begin{align*}
K_j^{(\ell)}(Y) 
&= t^{-\frac12(\ell-1)}\cdot t^{\ell-1-j}\cdot \genfrac[]{0pt}{0}{\ell}{j}_{q,t} \cdot
\genfrac(){0pt}{0}{\ell}{j}_{Y} \cdot
\frac{(1-Y^{-2}q^{\ell-2j})}{(1-t^{-1}Y^{-2}q^{\ell-2j})}\cdot
\frac{(1-t^{-1}Y^{-2})}{(1-Y^{-2})}
\end{align*}
The first few of the $K_j^{(\ell)}(Y)$ are
\begin{align*}
K_0^{(0)}(Y) &= (t^{\frac12}+t^{-\frac12})
\qquad (\hbox{since $\mathbf{1}_0E_0(X)\mathbf{1}_0 = \mathbf{1}_0^2 = \mathbf{1}_0 (t^{\frac12}+t^{-\frac12})$}), \\
\\
K_0^{(1)}(Y) &= 1\cdot  1\cdot
\frac{\cancel{(1-t^{-1}Y^{-2}q)} }{ \cancel{(1-Y^{-2}q)} }\cdot 
\frac{ \cancel{(1-Y^{-2}q)} }{ \cancel{(1-t^{-1}Y^{-2}q)} } \cdot
\frac{(1-t^{-1}Y^{-2})}{(1-Y^{-2})}, \\
K_1^{(1)}(Y) &= 1\cdot t^{-1}\cdot \frac{(1-tY^{-2}q^{-1})}{\cancel{(1-Y^{-2}q^{-1})} }\cdot 
\frac{\cancel{(1-Y^{-2}q^{-1})} }{(1-t^{-1}Y^{-2}q^{-1})} \cdot
\frac{(1-t^{-1}Y^{-2})}{(1-Y^{-2})}, \\
\\
K_0^{(2)}(Y) &= t^{-\frac12} \cdot t\cdot 
\frac{(1-t^{-1}Y^{-2}q)\cancel{(1-t^{-1}Y^{-2}q^2)}}{(1-Y^{-2}q)\cancel{(1-Y^{-2}q^2)}}\cdot
\frac{\cancel{(1-Y^{-2}q^2)}}{\cancel{(1-t^{-1}Y^{-2}q^2)} } \cdot
\frac{(1-t^{-1}Y^{-2})}{(1-Y^{-2})}
\\
K_1^{(2)}(Y) &= t^{-\frac12}\cdot 1\cdot
\frac{(1-q^2)(1-t)}{(1-q)(1-tq)}\cdot 
\frac{ \cancel{(1-t^{-1}Y^{-2})} }{(1-Y^{-2}q)}\cdot \frac{(1-tY^{-2})}{(1-Y^{-2}q^{-1})} \cdot
\frac{ \cancel{(1-Y^{-2})} }{\cancel{(1-t^{-1}Y^{-2})} } \cdot
\frac{(1-t^{-1}Y^{-2})}{\cancel{(1-Y^{-2})} }
\\
K_2^{(2)}(Y) &=t^{-\frac12} \cdot t^{-1}\cdot
\frac{(1-tY^{-2}q^{-2})(1-tY^{-2}q^{-1})}{ \cancel{(1-Y^{-2}q^{-2})} (1-Y^{-2}q^{-1}) }\cdot
\frac{ \cancel{(1-Y^{-2}q^{-2})} }{(1-t^{-1}Y^{-2}q^{-2})} \cdot
\frac{(1-t^{-1}Y^{-2})}{(1-Y^{-2})}
\end{align*}
\begin{align*}
K_0^{(3)}(Y) &=t^{-1}\cdot t^2\cdot
\frac{(1-t^{-1}Y^{-2}q)(1-t^{-1}Y^{-2}q^2)(1-t^{-1}Y^{-2}q^3)}{(1-Y^{-2}q)(1-Y^{-2}q^2)(1-Y^{-2}q^3)} \cdot
\frac{(1-Y^{-2}q^{3})}{(1-t^{-1}Y^{-2}q^{3})} \cdot
\frac{(1-t^{-1}Y^{-2})}{(1-Y^{-2})}
\\
K_1^{(3)}(Y) &=t^{-1}\cdot t\cdot
\frac{(1-t)(1-q^3)}{(1-q)(1-tq^2)}\cdot
\frac{(1-t^{-1}Y^{-2})(1-t^{-1}Y^{-2}q)}{(1-Y^{-2}q)(1-Y^{-2}q^2)}
\cdot \frac{(1-tY^{-2}q)}{(1-Y^{-2}q^{-1})}
\\
&\qquad
\cdot \frac{(1-Y^{-2}q)}{(1-t^{-1}Y^{-2}q)} \cdot
\frac{(1-t^{-1}Y^{-2})}{(1-Y^{-2})}
\\
K_2^{(3)}(Y) &= t^{-1}\cdot 1\cdot
\frac{(1-t)(1-q^3)}{(1-q)(1-tq^2)}\cdot
\frac{(1-t^{-1}Y^{-2}q^{-1})}{(1-Y^{-2}q)}
\cdot \frac{(1-tY^{-2}q^{-1})(1-tY^{-2})}{(1-Y^{-2}q^{-2})(1-Y^{-2}q^{-1})} 
\\
&\qquad
\cdot
\frac{(1-Y^{-2}q^{-1})}{(1-t^{-1}Y^{-2}q^{-1})} \cdot
\frac{(1-t^{-1}Y^{-2})}{(1-Y^{-2})}
\\
K_3^{(3)}(Y) &= 
t^{-1}\cdot t^{-1}\cdot
\frac{(1-tY^{-2}q^{-3})(1-tY^{-2}q^{-2})(1-tY^{-2}q^{-1})}
{(1-Y^{-2}q^{-3})(1-Y^{-2}q^{-2})(1-Y^{-2}q^{-1})} \cdot
\frac{(1-Y^{-2}q^{-3})}{(1-t^{-1}Y^{-2}q^{-3})} \cdot
\frac{(1-t^{-1}Y^{-2})}{(1-Y^{-2})}
\end{align*}

\newpage

\subsection{Examples of $K_j^{(\ell)}(Y)$ to $C_j^{(\ell)}(Y)$}

The following are examples of the identity \eqref{KtoC} from Proposition \ref{ABCtoDK} which says
$$K_j^{(\ell)}(Y)\cdot t^{-\frac12(\ell-2j)}\cdot \frac{(t^{-1}Y^{-2}q^{\ell-2j};q)_j}{(Y^{-2}q^{\ell-2j};q)_j}\cdot 
\frac{(Y^{-2};q)_{\ell-j}}{(t^{-1}Y^{-2};q)_{\ell-j}} = C_j^{(\ell)}(Y),
$$
where
$$C_j^{(\ell)}(Y) 
= \genfrac[]{0pt}{0}{\ell}{j}_{q,t} \cdot 
\frac{(t^{-1}Y^{-2}q^{-(j-1)};q)_j (tY^{-2} q^{\ell-2j};q)_j}
{(Y^{-2}q^{\ell-2j+1};q)_j (Y^{-2} q^{-j};q)_j}.
$$
The case $\ell=0$.
\begin{align*}
K_0^{(0)}(Y)\cdot &t^{-\frac12(0-2\cdot 0)}\cdot \frac{(t^{-1}Y^{-2}q^{\ell-2j};q)_0}{(Y^{-2}q^{\ell-2j};q)_0}\cdot
\frac{(Y^{-2};q)_0 }{(t^{-1}Y^{-2};q)_0 }  \\
&=(t^{\frac12}+t^{-\frac12})\cdot 1\cdot \frac{1}{1}\cdot \frac{1}{1} 
= t^{\frac12}+t^{-\frac12}.
\end{align*}
The case $\ell=1$.
\begin{align*}
K_0^{(1)}(Y)\cdot &t^{-\frac12(1-2\cdot 0)}\cdot \frac{(t^{-1}Y^{-2}q^{\ell-2j};q)_0}{(Y^{-2}q^{\ell-2j};q)_0}\cdot
\frac{(Y^{-2};q)_1 }{(t^{-1}Y^{-2};q)_1 } \\
&=\frac{(1-t^{-1}Y^{-2})}{(1-Y^{-2})}\cdot  t^{-\frac12} \frac{1}{1}\cdot \frac{(1-Y^{-2})}{(1-t^{-1}Y^{-2})} 
= t^{-\frac12}.
\end{align*}
\begin{align*}
K_1^{(1)}(Y)\cdot &t^{-\frac12(1-2\cdot 1)}\cdot \frac{(t^{-1}Y^{-2}q^{1-2\cdot1};q)_1}{(Y^{-2}q^{1-2\cdot1};q)_1}\cdot
\frac{(Y^{-2};q)_0 }{(t^{-1}Y^{-2};q)_0 } \\
&=t^{-1} \frac{(1-tY^{-2}q^{-1})}{(1-t^{-1}Y^{-2}q^{-1})}\frac{(1-t^{-1}Y^{-2})}{(1-Y^{-2})}
\cdot  t^{\frac12} \frac{(1-t^{-1}Y^{-2}q^{-1})}{(1-Y^{-2}q^{-1})} \cdot \frac{1}{1}
\\
&= t^{-\frac12} \frac{(1-t^{-1}Y^{-2})}{(1-Y^{-2})} \frac{(1-tY^{-2}q^{-1})}{(1-Y^{-2}q^{-1})}.
\end{align*}
The case $\ell=2$.
\begin{align*}
K_0^{(2)}(Y)\cdot &t^{-\frac12(2-2\cdot 0)}\cdot \frac{(t^{-1}Y^{-2}q^{\ell-2j};q)_0}{(Y^{-2}q^{\ell-2j};q)_0}\cdot
\frac{(Y^{-2};q)_2 }{(t^{-1}Y^{-2};q)_2 }  \\
&=t^{\frac12}\frac{(1-t^{-1}Y^{-2}q)(1-t^{-1}Y^{-2})}{(1-Y^{-2}q)(1-Y^{-2})}\cdot 
t^{-1}\cdot \frac{1}{1}\cdot \frac{(1-Y^{-2})(1-Y^{-2}q)}{(1-t^{-1}Y^{-2})(1-t^{-1}Y^{-2}q)}
=t^{-\frac12}.
\end{align*}
\begin{align*}
K_1^{(2)}(Y)\cdot &t^{-\frac12(2-2\cdot 1)}\cdot \frac{(t^{-1}Y^{-2}q^{\ell-2j};q)_1}{(Y^{-2}q^{\ell-2j};q)_1}\cdot
\frac{(Y^{-2};q)_1 }{(t^{-1}Y^{-2};q)_1 }  \\
&=t^{-\frac12}\frac{(1-q^2)(1-t)}{(1-q)(1-tq)} \frac{(1-tY^{-2})(1-t^{-1}Y^{-2})}{(1-Y^{-2}q)(1-Y^{-2}q^{-1})}\cdot 
1\cdot \frac{(1-t^{-1}Y^{-2})}{(1-Y^{-2})}\cdot \frac{(1-Y^{-2})}{(1-t^{-1}Y^{-2})}
\\
&=t^{-\frac12}\frac{(1-q^2)(1-t)}{(1-q)(1-tq)} \frac{(1-tY^{-2})(1-t^{-1}Y^{-2})}{(1-Y^{-2}q)(1-Y^{-2}q^{-1})}.
\end{align*}
\begin{align*}
K_2^{(2)}(Y)\cdot &t^{-\frac12(2-2\cdot 2)}\cdot \frac{(t^{-1}Y^{-2}q^{2-2\cdot 2};q)_2}{(Y^{-2}q^{2-2\cdot 2};q)_2}\cdot
\frac{(Y^{-2};q)_0 }{(t^{-1}Y^{-2};q)_0 }  \\
&=t^{-\frac32} \frac{(1-tY^{-2}q^{-2})(1-tY^{-2}q^{-1})(1-t^{-1}Y^{-2})}{(1-Y^{-2}q^{-1})(1-t^{-1}Y^{-2}q^{-2})(1-Y^{-2})}\cdot 
t^{-1}\cdot \frac{(1-t^{-1}Y^{-2}q^{-2})(1-t^{-1}Y^{-2}q^{-1}) }{ (1-Y^{-2}q^{-2})(1-Y^{-2}q^{-1})}\cdot \frac{1}{1}
\\
&=t^{-\frac12} \frac{(1-tY^{-2}q^{-2})(1-t Y^{-2}q^{-1})}{(1-Y^{-2}q^{-2})(1-Y^{-2}q^{-1})}
\cdot\frac{(1-t^{-1}Y^{-2}q^{-1})(1-t^{-1}Y^{-1})}{(1-Y^{-2}q^{-1})(1-Y^{-2})}
\end{align*}

\newpage	

\subsection{Examples of $E_\ell(x)$ and $P_\ell(x)$}
	
\begin{align*}
\vdots\quad \\
E_{-2}(x) &= x^{-2} + \frac{(1-t)(1-q^2)}{(1-q)(1-q^2t)} + \dfrac{(1-t)}{(1-q^2t)}x^{2} \\
E_{-1}(x) &= x^{-1} + \frac{1-t}{1-qt}x, \\
E_0(x) &= 1, \\
E_1(x) &= x, \\
E_2(x) &= x^2 + q \frac{(1-t)}{(1-qt)} , \\
E_3(x) &= x^3 + \Big(\frac{(1-t)q}{(1-tq)}+\frac{(1-t)q^2}{(1-tq^2)}\frac{(1-t)}{(1-tq)}\Big)x
+\frac{(1-t)q^2}{(1-tq^2)}x^{-1}, \\
\vdots\quad
\end{align*}
and
\begin{align*}
P_0(x) &= 1, \\
P_1(x) &= x + x^{-1}, \\
P_2(x) &= (x^2 + x^{-2}) + \frac{(1-q^2)(1-t)}{(1-q)(1-qt)} , \\
P_3(x) &= (x^3 + x^{-3}) + \dfrac{(1-q^3)(1-t)}{(1-q^2t)(1-q)}(x+x^{-1}), \\
P_4(x) &= (x^4 + x^{-4}) + \dfrac{(1-q^4)(1-t)}{(1-q^3t)(1-q)}(x^2 + x^{-2}) 
+ \dfrac{(1-q^4)(1-q^3)(1-qt)(1-t)}{(1-q^3t)(1-q^2t)(1-q^2)(1-q)}, \\
\vdots\quad
\end{align*}

\newpage

\subsection{Examples of products $E_\ell P_m$}

\begin{align*}
E_1 P_m 
&= E_{m+1}+\frac{(1-q^m)}{(1-tq^m)}E_{-m+1} 
= E_{m+1}+\ev_m\Big(\frac{(1-t^{-1}Y^{-2})}{(1-Y^{-2})}\Big)E_{-m+1}, \\
\\
E_2 P_m 
&= E_{m+2}+\frac{(1-t)}{(1-tq)}\cdot \frac{(1-q^m)}{(1-tq^{m-1})}\cdot \frac{(1-t^2 q^m)}{(1-tq^m)}
E_m \\
&\qquad
+ \frac{(1-q^{m-1})}{(1-tq^{m-1})}\frac{(1-q^m)}{(1-t q^m)} \cdot \frac{(1-t^2q^{m-1})}{(1-t q^{m-1})} E_{-m+2}
+q\frac{(1-t)}{(1-tq)}\cdot \frac{(1-q^m)}{(1-t q^{m+1})}E_{-m}
\\
\\
&= E_{m+2}+ \genfrac\{\}{0pt}{0}{1}{1}_{q,t} \ev_m\Big(\frac{(1-t^{-1}Y^{-2})}{(1-Y^{-2}q^{-1})}\cdot 
\frac{(1-tY^{-2})}{(1-Y^{-2})}\Big)E_m
\\
&\qquad
+ \ev_m\Big(\frac{(1-t^{-1}Y^{-2}q^{-1})}{(1-Y^{-2}q^{-1})}\frac{(1-t^{-1}Y^{-2})}{(1-Y^{-2})} \cdot 
\frac{(1-tY^{-2}q^{-1})}{(1-Y^{-2}q^{-1})}\Big) E_{-m+2}
\\
&\qquad
+q\genfrac\{\}{0pt}{0}{1}{1}_{q,t} \ev_m\Big(\frac{(1-t^{-1}Y^{-2})}{(1-Y^{-2}q)}\Big)E_{-m},
\end{align*}
\begin{align*}
E_3 P_m
&= E_{m+3} 
+ \frac{(1-t)(1-q^2)}{(1-q)(1-t q^2)}\cdot
\frac{(1-q^m)}{(1-t q^{m-1})}\cdot
\frac{(1-t^2 q^{m+1})}{(1-t q^{m+1})} 
\ E_{m+1}
\\
&\qquad
+\frac{(1-t)}{(1-t q^2)}\cdot
\frac{(1-q^{m-1})(1-q^m)}{(1- t q^{m-2})(1- t q^{m-1})}\cdot
\frac{(1-t^2 q^{m-1})(1-t^2 q^m)}{(1- t q^{m-1})(1- t q^m)}
\ E_{m-1}
\\
&\qquad
+ \frac{(1-q^{m-2})(1-q^{m-1}) (1-q^m)}{(1-tq^{m-2})(1-tq^{m-1})(1-tq^m )}\cdot
\frac{(1-t^2 q^{m-2})(1-t^2 q^{m-1}) }
{ (1-tq^{m-2})(1-tq^{m-1}) }
\ E_{-m+3}
\\
&\qquad
+q\frac{(1-t)(1-q^2)}{(1-q)(1-tq^2)}\cdot
\frac{(1-q^{m-1})(1-q^m)}{(1-tq^m )(1-tq^{m+1}) }\cdot
\frac{(1-t^2 q^m)}{(1-tq^{m-1})}
\ E_{-m+1} 
\\
&\qquad + q^2\frac{(1-t)}{(1- t q^2)}\cdot
\frac{(1-q^m)}{(1- t q^{m+2})} 
\ E_{-m-1} 
\\
\\
&= E_{m+3} 
+ \genfrac\{\}{0pt}{0}{2}{1}_{q,t} \ev_m\Big(
\frac{(1-t^{-1}Y^{-2})}{(1-Y^{-2}q^{-1})}\cdot
\frac{(1-tY^{-2}q)}{(1-Y^{-2}q)} \Big)
E_{m+1}
\\
&\qquad
+\genfrac\{\}{0pt}{0}{2}{2}_{q,t} \ev_m\Big(
\frac{(1-t^{-1}Y^{-2}q^{-1})(1-t^{-1}Y^{-2})}{(1- Y^{-2}q^{-2})(1- Y^{-2}q^{-1})}\cdot
\frac{(1-tY^{-2}q^{-1})(1-tY^{-2})}{(1- Y^{-2}q^{-1})(1- Y^{-2})}\Big)
E_{m-1}
\\
&\qquad
+ \ev_m\Big(
\frac{(1-t^{-1}Y^{-2}q^{-2})(1-t^{-1}Y^{-2}q^{-1}) (1-t^{-1}Y^{-2})}{(1-Y^{-2}q^{-2})(1-Y^{-2}q^{-1})(1-Y^{-2} )}\cdot
\frac{(1-tY^{-2}q^{-2})(1-tY^{-2}q^{-1}) }
{ (1-Y^{-2}q^{-2})(1-Y^{-2}q^{-1}) }\Big)
\ E_{-m+3}
\\
&\qquad
+q \genfrac\{\}{0pt}{0}{2}{1}_{q,t} \ev_m\Big(
\frac{(1-t^{-1}Y^{-2}q^{-1})(1-t^{-1}Y^{-2})}{(1-Y^{-2} )(1-Y^{-2}q) }\cdot
\frac{(1-tY^{-2})}{(1-Y^{-2}q^{-1})}\Big)
\ E_{-m+1} 
\\
&\qquad 
+q^2\genfrac\{\}{0pt}{0}{2}{2}_{q,t} \ev_m\Big(
\frac{(1-t^{-1}Y^{-2})}{(1- Y^{-2}q^{2})} \Big)
E_{-m-1}.
\end{align*}

\newpage

\subsection{General formulas for $E_\ell P_m$ and $E_{-\ell}P_m$}

The general formula for $E_\ell P_m$ with $\ell\in \ZZ_{>0}$ is
\begin{align*}
&E_\ell(x) P_m(x) \\
&= \sum_{j=0}^{\ell-1} 
\genfrac\{\}{0pt}{0}{\ell-1}{j}_{q,t} \ev_m\Big( 
\frac{1-t^{-1}Y^{-2}q^{-(j-1)..0})(1-tY^{-2}q^{\ell-2j..\ell-1-j})}
{(1-Y^{-2}q^{-j..-1})(1-Y^{-2}q^{\ell-2j..\ell-1-j})}\Big)E_{m+\ell-2j}(x)
\\
&\qquad+ \sum_{j=0}^{\ell-1} 
q^j \genfrac\{\}{0pt}{0}{\ell-1}{j}_{q,t} \ev_m\Big( 
\frac{1-t^{-1}Y^{-2}q^{-(\ell-j-1)..0})(1-tY^{-2}q^{-(\ell-2j)..j-1})}{(1-Y^{-2}q^{j-(\ell-j-1)..j})(1-Y^{-2}q^{-(\ell-j-1)..-1})}\Big)E_{-m+\ell-2j}(x)
\\
&= \sum_{j=0}^{\ell-1} 
\genfrac\{\}{0pt}{0}{\ell-1}{j}_{q,t} \ev_m\Big( 
\frac{(t^{-1}Y^{-2}q^{-(j-1)};q)_j (tY^{-2}q^{\ell-2j};q)_{j} }{ (Y^{-2}q^{-j};q)_j (Y^{-2}q^{\ell-2j})_j}\Big)E_{m+\ell-2j}(x)
\\
&\qquad+ \sum_{j=0}^{\ell-1} 
q^j \genfrac\{\}{0pt}{0}{\ell-1}{j}_{q,t} \ev_m\Big( 
\frac{(t^{-1}Y^{-2}q^{-(\ell-j-1)};q)_{\ell-j} (tY^{-2}q^{-(\ell-2j-1)};q)_{\ell-j-1} }
{ (Y^{-2}q^{-(\ell-2j-1)};q)_{\ell-j} (Y^{-2}q^{-(\ell-j-1)};q)_{\ell-j-1} }\Big)
E_{-m+\ell-2j}(x)
\\
&= \sum_{j=0}^{\ell-1} 
\ev_m(A^{(\ell)}_m(Y)) E_{m+\ell-2j}(x)
+ \sum_{j=0}^{\ell-1} \ev_m(B^{(\ell)}_j(Y)) E_{-m+\ell-2j}(x),
\end{align*}
where
\begin{align*}
A^{(\ell)}_j(Y) 
&=\genfrac\{\}{0pt}{0}{\ell-1}{j}_{q,t}
\frac{(t^{-1}Y^{-2}q^{-(j-1)};q)_j (tY^{-2}q^{\ell-2j};q)_{j} }{ (Y^{-2}q^{-j};q)_j (Y^{-2}q^{\ell-2j})_j}
\qquad\hbox{and} 
\\
B^{(\ell)}_j(Y) 
&= q^j \genfrac\{\}{0pt}{0}{\ell-1}{j}_{q,t} 
\frac{(t^{-1}Y^{-2}q^{-(\ell-j-1)};q)_{\ell-j} (tY^{-2}q^{-(\ell-2j-1)};q)_{\ell-j-1} }
{ (Y^{-2}q^{-(\ell-2j-1)};q)_{\ell-j} (Y^{-2}q^{-(\ell-j-1)};q)_{\ell-j-1} }.
\end{align*}
The general formula for $E_{-\ell}P_m$ with $\ell\in \ZZ_{\ge 0}$ is 
\begin{align*}
&E_{-\ell}(x) P_m(x) \\
&= \sum_{j=0}^{\ell} 
\genfrac\{\}{0pt}{0}{\ell}{j}_{q,t} \ev_m\Big( 
\frac{1-t^{-1}Y^{-2}q^{-(j-1)..0})(1-tY^{-2}q^{\ell+1-2j..\ell-j})}
{(1-Y^{-2}q^{-j..-1})(1-Y^{-2}q^{\ell+1-2j..\ell-j})}\Big)E_{-m-\ell+2j}(x)
\\
&\qquad+ \sum_{j=0}^{\ell} 
t\cdot q^j \genfrac\{\}{0pt}{0}{\ell}{j}_{q,t} \ev_m\Big( 
\frac{1-t^{-1}Y^{-2}q^{-(\ell-j)..0})(1-tY^{-2}q^{-(\ell+1-2j)..j-1})}{(1-Y^{-2}q^{j-(\ell-j)..j})(1-Y^{-2}q^{-(\ell-j)..-1})}\Big)E_{m-\ell+2j}(x)
\\
&= \sum_{j=0}^{\ell} 
\genfrac\{\}{0pt}{0}{\ell}{j}_{q,t} \ev_m\Big( 
\frac{(t^{-1}Y^{-2}q^{-(j-1)};q)_j (tY^{-2}q^{\ell+1-2j};q)_{j} }{ (Y^{-2}q^{-j};q)_j (Y^{-2}q^{\ell+1-2j})_j}\Big)E_{-m-\ell+2j}(x)
\\
&\qquad+ \sum_{j=0}^{\ell} 
t\cdot q^j \genfrac\{\}{0pt}{0}{\ell}{j}_{q,t} \ev_m\Big( 
\frac{(t^{-1}Y^{-2}q^{-(\ell-j)};q)_{\ell+1-j} (tY^{-2}q^{-(\ell-2j)};q)_{\ell-j} }
{ (Y^{-2}q^{-(\ell-2j)};q)_{\ell-j} (Y^{-2}q^{-(\ell-j)};q)_{\ell-j} }\Big)
E_{m-\ell+2j}(x)
\\
&= \sum_{j=0}^{\ell} \ev_m(A^{(\ell+1)}_j(Y)) E_{-m-\ell+2j}(x)
+ \sum_{j=0}^{\ell} t\cdot \ev_m(B^{(\ell+1)}_j(Y)) E_{m-\ell+2j}(x).
\end{align*}

\newpage
		
\subsection{Examples of products $E_{-\ell+1} P_m$}
		
\begin{align*}
E_0 P_m 
&= E_{-m}+t\frac{(1-q^m)}{(1-tq^m)}E_{m} 
= E_{-m}+t\cdot\ev_m\Big(\frac{(1-t^{-1}Y^{-2})}{(1-Y^{-2})}\Big)E_{m}, \\
\\
E_{-1} P_m 
&= E_{-m-1}+\frac{(1-t)}{(1-tq)}\cdot \frac{(1-q^m)}{(1-tq^{m-1})}\cdot \frac{(1-t^2 q^m)}{(1-tq^m)}
E_{-m+1} \\
&\qquad
+ t\frac{(1-q^{m-1})}{(1-tq^{m-1})}\frac{(1-q^m)}{(1-t q^m)} \cdot \frac{(1-t^2q^{m-1})}{(1-t q^{m-1})} E_{m-1}
+tq\frac{(1-t)}{(1-tq)}\cdot \frac{(1-q^m)}{(1-t q^{m+1})}E_{m+1}
\\
\\
&= E_{-m-1}+ \genfrac\{\}{0pt}{0}{1}{1}_{q,t} \ev_m\Big(\frac{(1-t^{-1}Y^{-2})}{(1-Y^{-2}q^{-1})}\cdot 
\frac{(1-tY^{-2})}{(1-Y^{-2})}\Big)E_{-m+1}
\\
&\qquad
+ t\cdot\ev_m\Big(\frac{(1-t^{-1}Y^{-2}q^{-1})}{(1-Y^{-2}q^{-1})}\frac{(1-t^{-1}Y^{-2})}{(1-Y^{-2})} \cdot 
\frac{(1-tY^{-2}q^{-1})}{(1-Y^{-2}q^{-1})}\Big) E_{m-1}
\\
&\qquad
+t\cdot q\genfrac\{\}{0pt}{0}{1}{1}_{q,t} \ev_m\Big(\frac{(1-t^{-1}Y^{-2})}{(1-Y^{-2}q)}\Big)E_{m+1},
\end{align*}
\begin{align*}
E_{-2} P_m
&= E_{-m-2} 
+ \frac{(1-t)(1-q^2)}{(1-q)(1-t q^2)}\cdot
\frac{(1-q^m)}{(1-t q^{m-1})}\cdot
\frac{(1-t^2 q^{m+1})}{(1-t q^{m+1})} 
\ E_{-m}
\\
&\qquad
+\frac{(1-t)}{(1-t q^2)}\cdot
\frac{(1-q^{m-1})(1-q^m)}{(1- t q^{m-2})(1- t q^{m-1})}\cdot
\frac{(1-t^2 q^{m-1})(1-t^2 q^m)}{(1- t q^{m-1})(1- t q^m)}
\ E_{-m+2}
\\
&\qquad
+ t\cdot \frac{(1-q^{m-2})(1-q^{m-1}) (1-q^m)}{(1-tq^{m-2})(1-tq^{m-1})(1-tq^m )}\cdot
\frac{(1-t^2 q^{m-2})(1-t^2 q^{m-1}) }
{ (1-tq^{m-2})(1-tq^{m-1}) }
\ E_{m-2}
\\
&\qquad
+t\cdot q\frac{(1-t)(1-q^2)}{(1-q)(1-tq^2)}\cdot
\frac{(1-q^{m-1})(1-q^m)}{(1-tq^{m-1})(1-tq^m )}\cdot
\frac{(1-t^2 q^m)}{(1-tq^{m+1})}
\ E_{m} 
\\
&\qquad + t\cdot q^2\frac{(1-t)}{(1- t q^2)}\cdot
\frac{(1-q^m)}{(1- t q^{m+2})} 
\ E_{m+2} 
\\
\\
&= E_{-m-2} 
+ \genfrac\{\}{0pt}{0}{2}{1}_{q,t} \ev_m\Big(
\frac{(1-t^{-1}Y^{-2})}{(1-Y^{-2}q^{-1})}\cdot
\frac{(1-tY^{-2}q)}{(1-Y^{-2}q)} \Big)
E_{-m}
\\
&\qquad
+\genfrac\{\}{0pt}{0}{2}{2}_{q,t} \ev_m\Big(
\frac{(1-t^{-1}Y^{-2}q^{-1})(1-t^{-1}Y^{-2})}{(1- Y^{-2}q^{-2})(1- Y^{-2}q^{-1})}\cdot
\frac{(1-tY^{-2}q^{-1})(1-tY^{-2})}{(1- Y^{-2}q^{-1})(1- Y^{-2})}\Big)
E_{-m+2}
\\
&\qquad
+ t\cdot \ev_m\Big(
\frac{(1-t^{-1}Y^{-2}q^{-2})(1-t^{-1}Y^{-2}q^{-1}) (1-t^{-1}Y^{-2})}{(1-Y^{-2}tq^{-2})(1-Y^{-2}tq^{-1})(1-Y^{-2} )}\cdot
\frac{(1-tY^{-2}q^{-2})(1-tY^{-2}q^{-1}) }
{ (1-Y^{-2}q^{-2})(1-Y^{-2}q^{-1}) }\Big)
\ E_{m-2}
\\
&\qquad
+t\cdot q \genfrac\{\}{0pt}{0}{2}{1}_{q,t} \ev_m\Big(
\frac{(1-t^{-1}Y^{-2}q^{-1})(1-t^{-1}Y^{-2})}{(1-Y^{-2}q^{-1})(1-Y^{-2} )}\cdot
\frac{(1-tY^{-2})}{(1-Y^{-2}q)}\Big)
\ E_{m} 
\\
&\qquad 
+t\cdot q^2\genfrac\{\}{0pt}{0}{2}{2}_{q,t} \ev_m\Big(
\frac{(1-t^{-1}Y^{-2})}{(1- Y^{-2}q^{2})} \Big)
E_{m+2}.
\end{align*}

\newpage

\subsection{Examples of products $P_\ell P_m$}

The general formula is
$$P_\ell(x)P_m(x) 
= \sum_{j=0}^\ell  \ev_m\big( C_j^{(\ell)}(Y)\big)
P_{m+\ell-2j}(x),
$$
where
$$
C_j^{(\ell)}(Y) 
= \genfrac[]{0pt}{0}{\ell}{j}_{q,t} \cdot
\frac{(1-t^{-1}Y^{-2}q^{-(j-1)..0})(1-tY^{-2} q^{\ell-2j..\ell-j-1})}
{(1-Y^{-2}q^{\ell-2j+1..\ell-j})(1-Y^{-2} q^{-j..-1})} .
$$
The first few cases are
\begin{align*}
P_1 P_m 
&= P_{m+1}+\frac{(1-q^m)}{(1-t q^m)}\frac{(1-t^2 q^{m-1})}{(1-tq^{m-1})} P_{m-1}
\\
&= P_{m+1}+\ev\Big( \frac{(1-t^{-1}Y^{-2})}{(1-Y^{-2})}\frac{(1-tY^{-2}q^{-1})}{(1-Y^{-2}q^{-1})} \Big)P_{m-1},
\\
\\
P_2 P_m
&=
P_{m+2} 
+
\frac{(1-q^2)(1-t)}{(1-tq)(1-q)}\cdot 
\frac{(1-q^m)}{(1- t q^{m+1})}\cdot \frac{(1- t^2 q^m)}{(1- t q^{m-1})} P_m \\
&\qquad +  \frac{(1-q^{m-1}) (1-q^m) }{ (1- t q^{m-1}) (1- t q^m) }\cdot
\frac{ (1- t^2 q^{m-2})(1- t^2 q^{m-1}) }{ (1- t q^{m-2})(1- t q^{m-1}) } P_{m-2}
\\
\\
&=
P_{m+2} 
+ \genfrac[]{0pt}{0}{2}{1}_{q,t} \ev\Big(
\frac{(1-t^{-1}Y^{-2})}{(1- Y^{-2}q)}\cdot \frac{(1- tY^{-2})}{(1- Y^{-2}q^{-1})} \Big) P_m 
\\
&\qquad + \ev\Big(
\frac{(1-t^{-1}Y^{-2}q^{-1}) (1-t^{-1}Y^{-2}) }{ (1- Y^{-2}q^{-1}) (1- Y^{-2}) }\cdot
\frac{ (1- tY^{-2}q^{-2})(1- tY^{-2}q^{-1}) }{ (1- Y^{-2}q^{-2})(1- Y^{-2}q^{-1}) }  \Big) P_{m-2},
\\
\\
P_3  P_m
&= P_{m+3}
+ \frac{(1-t)(1-q^3)}{(1-q)(1-tq^2)}\cdot
\frac{(1-q^m)}{(1-tq^{m+2})}\cdot
\frac{(1-t^2q^{m+1})}{(1-tq^{m-1})}
\ P_{m+1}
\\
&\qquad
+\frac{(1-t)(1-q^3) }{ (1-q)(1- t q^2)}\cdot
\frac{(1-q^{m-1})(1-q^m)}{ (1- t q^m)(1- t q^{m+1}) }  \cdot
\frac{(1- t^2 q^{m-1})(1- t^2 q^m )}{ (1- t q^{m-2})(1- t q^{m-1}) }
\ P_{m-1}
\\
&\qquad
+\frac{(1-q^{m-2})(1-q^{m-1})(1-q^m)}{(1- t q^{m-2})(1- t q^{m-1})(1- t q^m)}
\cdot \frac{(1- t^2 q^{m-3})(1- t^2 q^{m-2})(1- t^2 q^{m-1})}{(1- t q^{m-3})(1- t q^{m-2})(1- t q^{m-1})}
P_{m-3}
\\
\\
&= P_{m+3}
+ \genfrac[]{0pt}{0}{3}{1}_{q,t} \ev\Big(
\frac{(1-t^{-1}Y^{-2})}{(1-Y^{-2}q^{2})}\cdot
\frac{(1-tY^{-2}q)}{(1-Y^{-2}q^{-1})}\Big)
P_{m+1}
\\
&\qquad
+ \genfrac[]{0pt}{0}{3}{2}_{q,t} \ev\Big(
\frac{(1-t^{-1}Y^{-2}q^{-1})(1-t^{-1}Y^{-2})}{ (1- Y^{-2})(1- Y^{-2}q) }  \cdot
\frac{(1- tY^{-2}q^{-1})(1- tY^{-2} )}{ (1- Y^{-2}q^{-2})(1- Y^{-2}q^{-1}) }\Big)
P_{m-1}
\\
&\qquad
+\ev\Big( \frac{(1-t^{-1}Y^{-2}q^{-2..0})}{(1- Y^{-2}q^{-2..0})}
\cdot \frac{(1- tY^{-2}q^{-3..-1})}{(1- Y^{-2}q^{-3..-1})}
\Big)
P_{m-3}
\end{align*}

\newpage

\subsection{Proof of the $q$-$t$-binomial formulas for $E_\ell$ and $P_\ell$}\label{qtbintoMac}

\begin{prop}
The electronic Macdonald polynomials are given by
$$ E_{-\ell}(x) =  \sum_{j=0}^\ell \genfrac[]{0pt}{0}{\ell}{j}_{q,t}
\frac{(1-tq^j)}{(1-tq^{\ell})} x^{\ell-2j}
\qquad\hbox{and}\qquad
E_{\ell}(x) = \sum_{j=0}^{\ell-1} \genfrac[]{0pt}{0}{\ell-1}{j}_{q,t}
\frac{q^{\ell-1-j} (1-tq^j)}{(1-tq^{\ell-1})} x^{-\ell+2j+2},
$$
and the bosonic Macdonald polynomials are given by 
$$
P_\ell(x) =  \sum_{j=0}^\ell \genfrac[]{0pt}{0}{\ell}{j}_{q,t} x^{\ell-2j}.
$$
\end{prop}
\begin{proof}
Following \cite[\S6.2]{Mac03}, where $q^k=t$,
\begin{align*}
\genfrac[]{0pt}{0}{k+a}{b}
&=\frac{(q;q)_{k+a}}{(q;q)_b (q;q)_{k+a-b}}
=\frac{(1-q^{k+a-b+1})\cdots(1-q^{k+a})}{(q;q)_b}
\\
&=\frac{(1-tq^{a-b+1})\cdots(1-tq^a)}{(q;q)_b} = \frac{(tq^{a-(b-1)};q)_b}{(q;q)_b}.
\end{align*}
Then, from \cite[(6.2.7)]{Mac03},
\begin{align*}
E_{-m}
&= \genfrac[]{0pt}{0}{k+m}{m}^{-1} \sum_{i+j=m} 
\genfrac[]{0pt}{0}{k+i-1}{i}
\genfrac[]{0pt}{0}{k+j}{j}
x^{i-j}
\\
&= \frac{(q;q)_m}{(tq^{m-(m-1)};q)_m} \sum_{i+j=m}
\frac{(tq^{i-1-(i-1)};q)_i}{(q;q)_i}
\frac{(tq^{j-(j-1)};q)_j}{(q;q)_j} x^{i-j}
\\
&= \frac{(q;q)_m}{(tq;q)_m} \sum_{j=0}^m \frac{(t;q)_{m-j}}{(q;q)_{m-j}} \frac{(tq;q)_j}{(q;q)_j} x^{m-j-j}
\\
&= \frac{(1-t)}{(1-tq^m)} \frac{(q;q)_m}{(t;q)_m} 
\sum_{j=0}^m \frac{(t;q)_{m-j}}{(q;q)_{m-j}} \frac{(1-tq^j)}{(1-t)}\frac{(t;q)_j}{(q;q)_j} x^{m-2j}
\\
&=  \sum_{j=0}^m \genfrac[]{0pt}{0}{m}{j}_{q,t}
\frac{(1-tq^j)}{(1-tq^{m})} x^{m-2j}
\end{align*}
and, from \cite[(6.2.8)]{Mac03},
\begin{align*}
E_{m+1}
&= \genfrac[]{0pt}{0}{k+m}{m}^{-1} \sum_{i+j=m} 
\genfrac[]{0pt}{0}{k+i-1}{i}
\genfrac[]{0pt}{0}{k+j}{j}
q^i x^{-i+j+1}
=  \sum_{j=0}^m \genfrac[]{0pt}{0}{m}{j}_{q,t}
\frac{(1-tq^j)}{(1-tq^{m})} q^{m-j} x^{-m+2j+1}
\end{align*}
so that
$$E_m =  \sum_{j=0}^{m-1} \genfrac[]{0pt}{0}{m-1}{j}_{q,t}
\frac{(1-tq^j)}{(1-tq^{m-1})} q^{m-1-j} x^{-m+2j+2}.$$
From \cite[(6.3.7)]{Mac03},
\begin{align*}
P_m
&= \genfrac[]{0pt}{0}{k+m-1}{m}^{-1} \sum_{i+j=m} 
\genfrac[]{0pt}{0}{k+i-1}{i}
\genfrac[]{0pt}{0}{k+j-1}{j}
x^{i-j}
\\
&= \frac{(q;q)_m}{(tq^{m-1-(m-1)};q)_m} \sum_{i+j=m}
\frac{(tq^{i-1-(i-1)};q)_i}{(q;q)_i}
\frac{(tq^{j-1-(j-1)};q)_j}{(q;q)_j} x^{i-j}
\\
&= \frac{(q;q)_m}{(t;q)_m} \sum_{j=0}^m \frac{(t;q)_{m-j}}{(q;q)_{m-j}} \frac{(t;q)_j}{(q;q)_j} x^{m-j-j}
=  \sum_{j=0}^m \genfrac[]{0pt}{0}{m}{j}_{q,t} x^{m-2j}.
\end{align*}
\end{proof}

\newpage

\subsection{Examples of $E_\ell(X)\mathbf{1}_0$}

This page provides examples of the identities in \eqref{E10}.
Since $E_1(X) = X$ and $E_{-1}(X) = X^{-1}+\frac{(1-t)}{(1-qt)}X$ then
\begin{align*}
E_1(X)\mathbf{1}_0 
&= X\mathbf{1}_0 = \tau_\pi^\vee T^{-1}_1\mathbf{1}_0 = t^{-\frac12}\tau_\pi^\vee \mathbf{1}_0
=t^{-\frac12}\tau_\pi^\vee E_0(X)\mathbf{1}_0 = t^{-\frac12} \eta_\pi \mathbf{1}_0,
\\
E_{-1}(X)\mathbf{1}_0 
&= \Big(X^{-1}+\frac{(1-t)}{(1-qt)}X\Big)\mathbf{1}_0
=\Big( T_1\tau_\pi^\vee+\frac{(1-t)}{(1-tq)}\tau_\pi^\vee T^{-1}_1\Big)\mathbf{1}_0
\\
&=\Big(T_1+t^{-\frac12}\frac{(1-t)}{(1-tq)}\Big)\tau_\pi^\vee \mathbf{1}_0
=\Big(T_1+t^{-\frac12}\frac{(1-t)}{(1-tq)}\Big)t^{\frac12}E_1(X) \mathbf{1}_0
\end{align*}
Since $E_3(X) = X^3 + \Big(\frac{(1-t)q}{(1-tq)}+\frac{(1-t)q^2}{(1-tq^2)}\frac{(1-t)}{(1-tq)}\Big)X
+\frac{(1-t)q^2}{(1-tq^2)}X^{-1}$ then
\begin{align*}
E_3(X)\mathbf{1}_0
&= \Big( X^3 + \Big(\frac{(1-t)q}{(1-tq)}+\frac{(1-t)q^2}{(1-tq^2)}\frac{(1-t)}{(1-tq)}\Big)X
+\frac{(1-t)q^2}{(1-tq^2)}X^{-1}\Big)\mathbf{1}_0
\\
&= \Big( \tau_\pi^\vee T^{-1}_1\tau_\pi^\vee T^{-1}_1\tau_\pi^\vee T^{-1}_1 + t^{-\frac12}\frac{t^{-\frac12}(1-t)tq}{(1-tq)}\tau_\pi^\vee T^{-1}_1
\\
&\qquad
+ \frac{t^{-\frac12} (1-t)tq^2}{(1-tq^2)} \frac{t^{-\frac12} (1-t)}{(1-tq)} \tau_\pi^\vee T^{-1}_1
+t^{-\frac12}\frac{t^{-\frac12}(1-t)tq^2}{(1-tq^2)} T_1\tau_\pi^\vee \Big)\mathbf{1}_0
\\
&= \Big( \tau_\pi^\vee T^{-1}_1\tau_\pi^\vee T^{-1}_1\tau_\pi^\vee T^{-1}_1 + t^{-\frac12}\frac{t^{-\frac12}(1-t)tq}{(1-tq)}\tau_\pi^\vee T^{-1}_1
\\
&\qquad
+ \frac{t^{-\frac12} (1-t)tq^2}{(1-tq^2)} \frac{t^{-\frac12} (1-t)}{(1-tq)} \tau_\pi^\vee T^{-1}_1
+t^{-\frac12}\frac{t^{-\frac12}(1-t)tq^2}{(1-tq^2)} (T^{-1}_1+(t^{\frac12}-t^{-\frac12}))\tau_\pi^\vee \Big)\mathbf{1}_0
\\
&= \Big( \tau_\pi^\vee T^{-1}_1\tau_\pi^\vee T^{-1}_1\tau_\pi^\vee t^{-\frac12}
+ t^{-\frac12}\frac{t^{-\frac12}(1-t)tq}{(1-tq)}\tau_\pi^\vee T^{-1}_1\tau_\pi^\vee \tau_\pi^\vee
\\
&\qquad
+ t^{-\frac12} \frac{t^{-\frac12} (1-t)tq^2}{(1-tq^2)} \frac{t^{-\frac12} (1-t)}{(1-tq)} \tau_\pi^\vee
+t^{-\frac12}\frac{t^{-\frac12}(1-t)tq^2}{(1-tq^2)} T^{-1}_1\tau_\pi^\vee
\\
&\qquad
-t^{-\frac12}\frac{t^{-\frac12}(1-t)tq^2}{(1-tq^2)}\frac{t^{-\frac12}(1-t)}{(1-tq)}(1-tq) \tau_\pi^\vee \Big)\mathbf{1}_0
\\
&= \Big( \tau_\pi^\vee T^{-1}_1\tau_\pi^\vee T^{-1}_1\tau_\pi^\vee t^{-\frac12}
+ t^{-\frac12}\frac{t^{-\frac12}(1-t)tq}{(1-tq)}\tau_\pi^\vee T^{-1}_1\tau_\pi^\vee \tau_\pi^\vee
\\
&\qquad
+t^{-\frac12}\frac{t^{-\frac12}(1-t)tq^2}{(1-tq^2)} T^{-1}_1\tau_\pi^\vee
\\
&\qquad
+t^{-\frac12}\frac{t^{-\frac12}(1-t)tq^2}{(1-tq^2)}\frac{t^{-\frac12}(1-t)}{(1-tq)}tq \tau_\pi^\vee \Big)\mathbf{1}_0
\\
&= \Big( t^{-\frac12} \tau_\pi^\vee T^{-1}_1\tau_\pi^\vee T^{-1}_1\tau_\pi^\vee 
+ t^{-\frac12}\frac{t^{-\frac12}(1-t)tq}{(1-tq)}\tau_\pi^\vee T^{-1}_1\tau_\pi^\vee \tau_\pi^\vee
\\
&\qquad
+t^{-\frac12}\frac{t^{-\frac12}(1-t)tq^2}{(1-tq^2)} \tau_\pi^\vee \tau_\pi^\vee T^{-1}_1\tau_\pi^\vee
+t^{-\frac12}\frac{t^{-\frac12}(1-t)tq^2}{(1-tq^2)}\frac{t^{-\frac12}(1-t)tq}{(1-tq)} \tau_\pi^\vee \tau_\pi^\vee \tau_\pi^\vee \Big)\mathbf{1}_0
\\
&= t^{-\frac12}\tau_\pi^\vee \Big(T^{-1}_1 + \frac{t^{-\frac12}(1-t)tq^2}{(1-tq^2)}\Big)
\tau_\pi^\vee \Big(T^{-1}_1 + \frac{t^{-\frac12}(1-t)tq}{(1-tq)}\Big) \tau_\pi^\vee \mathbf{1}_0.
\end{align*}

\newpage

\subsection{Examples of the $E_\ell(X)\mathbf{1}_0$ expansion} 

Let
$$c(Y) = \frac{(1-tY^2)}{(1-Y^2)}
\qquad\hbox{and}\qquad
F_\ell(Y) = \frac{(1-t)(1-tY^2q^\ell)}{(1-tq^\ell)(1-Y^2)}.
$$
Then
\begin{align*}
E_1(X)\mathbf{1}_0 &= t^{-\frac12} \eta_\pi \mathbf{1}_0 
= t^{-\frac12} \eta_\pi \eta_{s_1}\mathbf{1}_0 = \eta D^{(0)}_0(Y)\mathbf{1}_0, 
\end{align*}
\begin{align*}
E_2(X)\mathbf{1}_0 
&= t^{-\frac22}(\eta c(Y)+\eta_\pi F_1(Y))\eta \mathbf{1}_0 \\
&= t^{-\frac22}\eta^2 c(q^{-\frac12}Y)\mathbf{1}_0+t^{-\frac22}F_1(q^{-\frac12}Y^{-1})\mathbf{1}_0
\\
&= \eta^2 t^{-\frac22}\frac{(1-tY^2 q^{-1})}{(1-Y^2q^{-1})}\mathbf{1}_0
+t^{-\frac22}\frac{(1-t)(1-tq^{-1}Y^{-2}q)}{(1-tq)(1-q^{-1}Y^{-2})}\mathbf{1}_0
\\
&= \eta^2 t^{-\frac22}\cdot \frac{(1-tY^2 q^{-1})}{(1-Y^2q^{-1})}\mathbf{1}_0
+t^{-\frac22}\cdot qt \frac{(1-t)(1-t^{-1}Y^2)}{(1-tq)(1-Y^2q)}\mathbf{1}_0
\\
&=\eta^2 D_0^{(1)}(Y)\mathbf{1}_0+D_1^{(1)}(Y)\mathbf{1}_0
\end{align*}
and
\begin{align*}
E_3(X)\mathbf{1}_0
&=t^{-\frac12}(\eta c(Y)+\eta_\pi F_2(Y))(\eta^2 D_0^{(1)}(Y)\mathbf{1}_0+D_1^{(1)}\mathbf{1}_0) 
\\
&=\eta^3 t^{-\frac12}c(q^{-2}Y)D_0^{(1)}(Y)\mathbf{1}_0+\eta t^{-\frac12}c(Y) D_1^{(1)}(Y)\mathbf{1}_0
\\
&\qquad + \eta^{-1} t^{-\frac12}F_2(q^{-2}Y^{-1})D_0^{(1)}(Y^{-1})\mathbf{1}_0 
+ \eta t^{-\frac12}F_2(Y^{-1})D_1^{(1)}(Y^{-1})\mathbf{1}_0
\\
&= \eta^3 t^{-\frac12}c(q^{-2}Y)D_0^{(1)}(Y) \mathbf{1}_0 
\\
&\qquad + \eta t^{-\frac12}(c(Y)D_1^{(1)}(Y)+F_2(Y^{-1})D_1^{(1)}(Y^{-1}))\mathbf{1}_0
\\
&\qquad + \eta^{-1} t^{-\frac12}F_2(q^{-2}Y^{-1})D_0^{(1)}(Y^{-1})\mathbf{1}_0
\\
&= \eta^3 t^{-\frac32} \frac{(1-tY^2 q^{-2})(1-tY^2 q^{-1})}{(1-Y^2 q^{-2})(1-Y^2 q^{-1})} \mathbf{1}_0
\\
&\qquad + \eta t^{-\frac12} \frac{(1-t)(1-q^2)}{(1-tq^2)(1-q)}\cdot q\frac{(1-tY^2)(1-tqY^{-2})}{(1-Y^2q)(1-Y^{-2}q)} \mathbf{1}_0
\\
&\qquad + \eta^{-1} t^{-\frac32} \frac{(1-t)}{(1-tq^2)}\cdot 
\frac{(1-tY^{-2})(1-tY^{-2}q^{-1})}{(1-Y^{-2}q^{-2})(1-Y^{-2}q^{-1})}\mathbf{1}_0
\\
&= \eta^3 D_0^{(2)}(Y) \mathbf{1}_0 +\eta D_1^{(2)}(Y)\mathbf{1}_0 + \eta^{-1}D_2^{(2)}(Y)\mathbf{1}_0.
\end{align*}

\end{document}